\documentclass[a4paper,11pt,twoside,notitlepage]{amsart}

\usepackage{amsmath,amssymb,amsthm,mathtools,mathrsfs}
\usepackage[shortlabels]{enumitem}
\usepackage[ocgcolorlinks,linkcolor=blue,citecolor=blue,urlcolor=blue]{hyperref}
\usepackage{url}
\numberwithin{equation}{section}
\mathtoolsset{showonlyrefs}
\allowdisplaybreaks

\newtheorem{theorem}{Theorem}[section]
\newtheorem{proposition}[theorem]{Proposition}
\newtheorem{lemma}[theorem]{Lemma}
\newtheorem{corollary}[theorem]{Corollary}
\newtheorem{definition}[theorem]{Definition}
\newtheorem{remark}[theorem]{Remark}
\newtheorem{question}{Question}[section]

\DeclareMathOperator{\diam}{diam}
\DeclareMathOperator{\dist}{dist}
\DeclareMathOperator{\supp}{supp}

\DeclareMathOperator{\esssup}{ess\,sup}

\newcommand{\R}{\mathbb R}

\newcommand{\Sph}{\mathbb S^2}
\newcommand{\p}{\partial}
\newcommand{\eps}{\varepsilon}
\newcommand{\ol}{\overline}

\newcommand{\Q}{\mathcal Q}
\newcommand{\B}{\mathcal B}
\newcommand{\A}{\mathcal A}
\newcommand{\G}{\mathcal G}

\newcommand{\C}{\mathscr C}
\newcommand{\para}[1]{\vspace{3mm}\noindent\textbf{#1.}}

\title[Tomography of reactive Boltzmann kernels]{Tomography of reactive Boltzmann kernels from product fluxes}

\author[Y.-H. Lin]{Yi-Hsuan Lin}
\address{Department of Applied Mathematics, National Yang Ming Chiao Tung University, Hsinchu, Taiwan \& Fakult\"at f\"ur Mathematik, University of Duisburg-Essen, Essen, Germany}
\email{yihsuanlin3@gmail.com}

\author{Hongyu Liu}
\address{Department of Mathematics, City University of Hong Kong, Kowloon, Hong Kong, China}
\email{hongyu.liuip@gmail.com, hongyliu@cityu.edu.hk}

\keywords{Inverse kinetic problems, reactive Boltzmann equation, collision kernels, boundary measurements, high-velocity asymptotics, X-ray transform}
\subjclass[2020]{35R30, 35Q20, 82C40, 44A12}

\begin{document}
	
\begin{abstract}
	We study the reversible reaction $1+2\rightleftarrows3+4$ in a dilute gas described by a Boltzmann system. We prove that the outgoing flux of one product species determines the forward collision kernel throughout the domain on a prescribed admissible set of relative velocities and outgoing directions. The elastic collision laws are known, and the reverse kernel is fixed by a prescribed reciprocity relation. The measurements use two synchronized incoming packets with fixed relative velocity and a common velocity tending to infinity. Under bounded angularly integrated collision rates, the short transit time allows us to isolate the contribution of one forward reaction, with estimates uniform as the packets concentrate. Suitable tests of the product flux yield the spatial line integrals of the kernel and a Fourier reconstruction, without differentiating the nonlinear boundary map. We also estimate the effects of finite beam speed, packet width, detector resolution, and measurement noise on the recovered line integrals.
\end{abstract}
	
	\maketitle
	\tableofcontents
	
	\section{Introduction}
	
	We study an inverse problem for the reversible reaction $1+2\rightleftarrows3+4$ in a dilute gas. The question is whether the outgoing flux of one product species determines the forward collision kernel $K(x,g,\omega)$, where $x$ is the collision point, $g$ is the incoming relative velocity, and $\omega\in\Sph$ is the direction of the outgoing relative velocity. For a spatially homogeneous gas, let $\sigma_{12\to34}(g,\cdot)$ denote the cross section measure between prescribed internal states, and suppose that it has density $b_{12\to34}(g,\omega)$ with respect to $d\omega$. We use the convention $K(g,\omega)=|g|b_{12\to34}(g,\omega)$.
	
	Measurements of product velocities are available in experiments using crossed beams and velocity map imaging. Merged beams also allow small relative velocities even when both beams have nonzero laboratory speeds; see \cite{MichaelsenEtAl2017,VogelsEtAl2014,WeiLyuksyutovHerschbach2012}. In this paper, we use two incoming packets centered at $Re+w_1$ and $Re+w_2$, where $e\in\Sph$ and $w_1,w_2$ are fixed. The common velocity translation leaves the relative velocities and reaction energy unchanged. As $R\to\infty$, the time spent by the particles in the domain is of order $R^{-1}$.
	
	We prescribe zero initial data for all four species and inject only the first two. The third species has no contribution from free transport, and its leading outgoing signal comes from a single forward reaction between the two incoming populations. Terms involving further collisions are of smaller order as $R\to\infty$. We use measurements of the third species throughout the paper. The fourth species gives the same information with a corresponding change of the velocity filter; see Remark~\ref{rem:detected-product-channel}. The incoming data are prescribed at the boundary, where we also measure the product flux as a function of exit time, position, and velocity. The measurements used in the proof are integrals of this flux against prescribed detector responses.
	
	The dependence on $x$ allows the reaction law to vary within the domain. For example, one may consider $K(x,g,\omega)=K_{\rm bg}(g,\omega)+q(x,g,\omega)$, where $K_{\rm bg}$ is a known background kernel. The transport equation contains no force term, so particles move along straight lines between collisions. If the spatial variation of the reaction law is produced by an external control, its effect on these trajectories must be negligible during the observation interval. Inverse transport problems with an external force are studied in \cite{LaiZhou2021ExternalForce}.
	
	\para{The model and boundary measurements} Let $\Omega\subset\R^3$ be a bounded $C^3$ strictly convex domain. We consider four nonnegative distribution functions $F_i=F_i(t,x,v)$, $1\le i\le4$, and write $F=(F_1,F_2,F_3,F_4)$. The distributions satisfy the reactive Boltzmann system
	\begin{equation}\label{eq:intro-system}
		\begin{aligned}
			\p_tF_i+v\cdot\nabla_xF_i&=\Q_{K,i}(F),\\
			\Q_{K,i}(F)&:=Q_i^{\rm el}(F)+Q_i^{\rm r}(F), &&1\le i\le4.
		\end{aligned}
	\end{equation}
	Here $Q_i^{\rm el}$ is the known elastic collision operator and $Q_i^{\rm r}$ is the reactive operator. The forward kernel is unknown, while the reverse kernel is specified by the reciprocity condition in Definition~\ref{def:microreversible-kernel}. The reaction maps and the full gain--loss operators are defined in Section~\ref{sec:model}. All elastic and reactive kernels satisfy the bounds on the collision rates in \eqref{eq:kernel-finite-rate-class} and \eqref{eq:Lambda-el}. We write $\Q_K=(\Q_{K,1},\ldots,\Q_{K,4})$. An operator with one argument acts on the full distribution vector. Two scalar arguments, as in $Q_{3,+}^{K}(f_1,f_2)$, specify the bilinear gain from one collision channel.
	
	The coefficient to be recovered is
	\begin{equation}\label{eq:intro-unknown}
		K=K(x,g,\omega), \quad (x,g,\omega)\in\Omega\times\R^3\times\Sph.
	\end{equation}
	The variable $g$ is the incoming relative velocity of the first two components, and $\omega$ is the direction of the outgoing relative velocity. Formula \eqref{eq:third-gain-weak} specifies how $K$ enters the gain term for the third species, including its normalization. Whenever $K(y+\sigma e,g,\omega)$ is integrated over $\sigma\in\R$, we set $K(x,g,\omega)=0$ for $x\in\R^3\setminus\ol\Omega$. This convention concerns only the spatial variable, and the values on $\p\Omega$ do not affect the line integrals.
	
	Let $n(x)$ be the outward unit normal to $\p\Omega$, and set
	\begin{equation}\label{eq:intro-gamma}
		\Gamma_\pm=\{(x,v)\in\p\Omega\times\R^3:\ \pm n(x)\cdot v>0\}.
	\end{equation}
	Fix $T>\diam(\Omega)$. The temporal profiles used below are chosen so that all relevant scaled entry and exit times lie compactly in $(0,T)$. For $R>1$, let $F^R$ solve \eqref{eq:intro-system} on $0<t<T/R$, with zero initial data and incoming boundary values
	\begin{equation}\label{eq:intro-initial-boundary}
		\begin{aligned}
			F_1^R\big|_{(0,T/R)\times\Gamma_-}&=G_{1,R},
			&F_2^R\big|_{(0,T/R)\times\Gamma_-}&=G_{2,R},\\
			F_3^R\big|_{(0,T/R)\times\Gamma_-}&=0,
			&F_4^R\big|_{(0,T/R)\times\Gamma_-}&=0,
		\end{aligned}
	\end{equation}
	and $F_i^R(0,x,v)=0$ for $1\le i\le4$. Thus, only the first two components are injected.
	
	The measured quantity is the outgoing particle flux of the third component,
	\begin{equation}\label{eq:intro-product-flux}
		J_{3,R}(t,x,v):=(n(x)\cdot v)_+\,\gamma_+F_3^R(t,x,v), \quad (t,x,v)\in(0,T/R)\times\Gamma_+,
	\end{equation}
	where $(a)_+=\max\{a,0\}$ and $\gamma_+F_3^R$ is the outgoing characteristic representative defined in Section~\ref{sec:scaled-forward}. We define the boundary response by
	\begin{equation}\label{eq:intro-boundary-map}
		\A_{K,R}(G_{1,R},G_{2,R}):=J_{3,R}.
	\end{equation}
	The response $\A_{K,R}$ records the product flux as a function of exit time $t$, boundary position $x$, and velocity $v$. It uses no interior measurements of $F^R$.
	
	A physical detector does not need to evaluate $J_{3,R}$ at one point. If $\Psi$ is a prescribed compactly supported $C^1$ detector response on $(0,T/R)\times\Gamma_+$, its reading is
	\begin{equation}\label{eq:intro-general-detector-reading}
		\big\langle\A_{K,R}(G_{1,R},G_{2,R}),\Psi\big\rangle
		:=\int_0^{T/R}\int_{\Gamma_+}J_{3,R}(t,x,v)\Psi(t,x,v)\,dS_x\,dv\,dt.
	\end{equation}
	The response $\Psi$ specifies the time gate, boundary aperture, and detected velocity range. For $\Psi\ge0$, this is a weighted product count. Signed responses are obtained by subtracting two nonnegative readings. The inverse argument uses only these integrated flux measurements.
	
	\para{The incoming data and detector responses} We now describe the incident pulses and detector readings used in the theorem. Let $e\in\Sph$ be the common propagation direction, let $R=\eps^{-1}$, and fix an accessible incoming relative velocity $g_0$. The first two incoming distributions are synchronized in time, with velocities concentrated within $O(\delta)$ of $Re+w_1$ and $Re+w_2$, respectively, where $w_1-w_2=g_0$ and $m_1w_1+m_2w_2=0$. Definition~\ref{def:beam-data} specifies the velocity profiles and the incoming data in scaled boundary coordinates. The corresponding physical data $G_{i,R}^{\eps,\delta}$, $i=1,2$, are defined in \eqref{eq:physical-beam-data-definition}.
	
	Let $P_{e^\perp}$ be the orthogonal projection onto $e^\perp$ and put $\Omega_e=P_{e^\perp}(\Omega)$. A point $y\in\Omega_e$ parametrizes the chord $(y+\R e)\cap\Omega$, whose length is denoted by $L(y,e)$. A function $\eta\in C_c^\infty(\Omega_e)$ selects a compact family of these chords. We use the scaled time $s=Rt$ and denote the center of the incident pulses by $s_0$. The detector measures particles exiting near the expected time $s_0+L(y,e)$. For fixed $g_0$, the limiting velocities of the third species, after subtracting $Re$, lie on a known sphere. Given a function $\varphi$ of the outgoing direction, we choose a smooth velocity filter $\chi_{g_0,\varphi}$ on $\R^3$ whose restriction to this sphere agrees with $\varphi$ under the parametrization by outgoing directions. The time response, boundary coordinates, and velocity filter are defined in Sections~\ref{sec:scaled-forward} and \ref{sec:measurements}.
	
	Let $\Psi_{\eps}^{e,g_0;\eta,\varphi}$ denote the resulting physical detector response on $(0,T/R)\times\Gamma_+$. The restricted scalar measurement used below is
	\begin{equation}\label{eq:intro-measurement-symbol}
		\begin{split}
			\mathscr M_K^{\eps,\delta}(e,g_0;\eta,\varphi)
			:=R\big\langle\A_{K,R}(G_{1,R}^{\eps,\delta},G_{2,R}^{\eps,\delta}),\Psi_{\eps}^{e,g_0;\eta,\varphi}\big\rangle, \quad R=\eps^{-1}.
		\end{split}
	\end{equation}
	The product count from a single reaction is of order $R^{-1}$, so multiplication by $R$ gives its leading coefficient in the limit. The scalar measurement $\mathscr M_K^{\eps,\delta}$ is obtained from $\A_{K,R}$ by this normalization. Its expression in scaled coordinates is given in \eqref{eq:physical-scaled-measurement}.
	
	\begin{question}[Recovery of the collision kernel]\label{Q:inverse}
		Does the family of restricted detector readings \eqref{eq:intro-measurement-symbol}, obtained for all accessible propagation directions, relative velocities, spatial apertures, and product angular filters, determine the interior kernel $K(x,g,\omega)$?
	\end{question}
	
	\para{Main result} Let $\mathscr G\Subset\R^3$ be a bounded open set of relative velocities whose closure lies in the region where the forward reaction is energetically admissible, and let $\mathcal O\Subset\Sph$ be an open set of outgoing directions. The reaction threshold $g_{\rm th}$ is introduced in Subsection~\ref{subsec:kinematics}. The larger set $\mathscr G^\sharp$ of relative velocities and the admissible class $\mathfrak K=\mathfrak K(\mathscr G^\sharp,M_0,M_1)$ of forward kernels are defined in \eqref{eq:collision-windows} and \eqref{eq:admissible-kernel-class}. Throughout the paper, $K\in\mathfrak K$ denotes the forward kernel $K=K^+$; the reverse kernel $K^-$ is fixed by the prescribed reciprocity relation. For $U\Subset\Omega_e$, we take $\eta\in C_c^\infty(U)$ and $\varphi\in C_c^\infty(\mathcal O)$.
	
	The proof is based on the identity
	\begin{equation}\label{eq:intro-tested-xray}
		\begin{aligned}
			\lim_{\delta\to0}\lim_{\eps\to0}\mathscr M_K^{\eps,\delta}(e,g_0;\eta,\varphi)
			&=\int_U\eta(y)\int_{\Sph}\varphi(\omega)\bigg[\int_\R K(y+\sigma e,g_0,\omega)\,d\sigma\bigg]d\omega\,dy.
		\end{aligned}
	\end{equation}
	By varying $\eta$ and $\varphi$, we recover the Euclidean X-ray transform of $K(\cdot,g_0,\omega)$ for every accessible collision configuration.
	
	\begin{theorem}[Global recovery and reconstruction]\label{thm:main}
		Let $\Omega\subset\R^3$ be a bounded $C^3$ strictly convex domain. Let $K_1,K_2\in\mathfrak K$, with $\mathfrak K$ defined by \eqref{eq:admissible-kernel-class}, be two forward reactive kernels with the same masses and internal energies satisfying the assumptions of Subsection~\ref{subsec:kinematics}. Assume that the known elastic kernels are nonnegative and measurable, satisfy the interchange symmetry in Subsection~\ref{subsec:collision-operators}, and obey \eqref{eq:Lambda-el}. Use the same elastic kernels for both forward problems. For $j=1,2$, let $\mathscr M_{K_j}^{\eps,\delta}$ denote the normalized product measurements in \eqref{eq:intro-measurement-symbol}, constructed using the same incident packets and detector profiles.
		
		Assume that, for every $e\in\Sph$, every $g_0\in\mathscr G$, every open set $U\Subset\Omega_e$, every $\eta\in C_c^\infty(U)$, and every $\varphi\in C_c^\infty(\mathcal O)$, there exist $\eps_0,\delta_0>0$ such that
		\begin{equation}\label{eq:intro-measurement-equality}
			\mathscr M_{K_1}^{\eps,\delta}(e,g_0;\eta,\varphi)=\mathscr M_{K_2}^{\eps,\delta}(e,g_0;\eta,\varphi)
		\end{equation}
		whenever $0<\eps<\eps_0$ and $0<\delta<\delta_0$. The constants $\eps_0$ and $\delta_0$ may depend on the chosen measurement configuration. Then
		\begin{equation}\label{eq:intro-main-conclusion}
			K_1(x,g,\omega)=K_2(x,g,\omega)
		\end{equation}
		for every $(x,g,\omega)\in\Omega\times\mathscr G\times\mathcal O$.
		
		For a single kernel $K\in\mathfrak K$, let $\widetilde K_{g,\omega}$ be the zero extension of $K(\cdot,g,\omega)$ to $\R^3$. For every $e\in\Sph$, $g\in\mathscr G$, $\omega\in\mathcal O$, and $y\in e^\perp$, the same measurements determine the spatial X-ray transform
		\begin{equation}\label{eq:intro-line-data}
			X_xK(y,e;g,\omega)=\int_\R \widetilde K_{g,\omega}(y+\sigma e)\,d\sigma.
		\end{equation}
		
		For $\xi\ne0$, choose any $e_\xi\in\Sph$ satisfying $e_\xi\cdot\xi=0$. Then
		\begin{equation}\label{eq:intro-fourier-reconstruction}
			\widehat{\widetilde K}_{g,\omega}(\xi)=\int_{e_\xi^\perp}e^{-\mathsf i y\cdot\xi}X_xK(y,e_\xi;g,\omega)\,dy,
		\end{equation}
		where $\mathsf i=\sqrt{-1}$ and $\widehat{\widetilde K}_{g,\omega}(\xi)=\int_{\R^3}e^{-\mathsf i x\cdot\xi}\widetilde K_{g,\omega}(x)\,dx$. The right-hand side is independent of the choice of $e_\xi$. At $\xi=0$, one has $\widehat{\widetilde K}_{g,\omega}(0)=\int_{e^\perp}X_xK(y,e;g,\omega)\,dy$ for any $e\in\Sph$.
		
		Finally, for every $x\in\Omega$,
		\begin{equation}\label{eq:intro-inverse-fourier}
			K(x,g,\omega)=\lim_{\kappa\downarrow0}(2\pi)^{-3}\int_{\R^3}e^{\mathsf i x\cdot\xi}e^{-\kappa|\xi|^2}\widehat{\widetilde K}_{g,\omega}(\xi)\,d\xi.
		\end{equation}
		The convergence is pointwise in $\Omega$ and locally uniform with respect to $x$ on compact subsets of $\Omega$.
	\end{theorem}
	
	The theorem recovers the coefficient on all of $\Omega\times\mathscr G\times\mathcal O$ without a smallness assumption on its difference from a reference kernel. The incoming profiles have fixed size after integration in velocity. For each fixed packet width, they are independent of $R$ after the velocity translation; their pointwise values need not remain bounded as the packets concentrate. Proposition~\ref{prop:finite-resolution} gives a joint error estimate in $\eps$ and $\delta$, so the two limits in \eqref{eq:intro-tested-xray} may also be taken simultaneously for fixed detector tests. Proposition~\ref{prop:Xray-Plancherel} gives stability identities for the recovered X-ray transforms. Corollary~\ref{cor:raw-flux-noise} estimates the effect of errors in the measured flux on individual line integrals along chords that meet the boundary transversally.
	
	\para{Outline of the proof} We first write the full collision operator in event variables. The reaction maps commute with a common velocity translation, and the reactive change of variables fixes the forward and reverse normalizations. In particular, the detector recovers $K$ with respect to the angular measure $d\omega$.
	
	The small parameter is the transit time. Set $\eps=R^{-1}$ and $s=Rt$, and define
	\begin{equation}\label{eq:intro-H-scaling}
		H_i^\eps(s,x,w)=F_i^R(\eps s,x,Re+w),\quad R=\eps^{-1}.
	\end{equation}
	In these variables, the physical system becomes
	\begin{equation}\label{eq:intro-scaled}
		\big[\p_s+(e+\eps w)\cdot\nabla_x\big]H_i^\eps=\eps\Q_{K,i}(H^\eps).
	\end{equation}
	Under the bounds on the angularly integrated collision rates, the full collision operator is locally Lipschitz in $L^1_w(L^\infty_{s,x})$. For sufficiently small $\eps$, we prove well-posedness on the fixed interval $0<s<T$ for incoming data of fixed size in the norm used below. Let $U^\eps$ be the free transport solution with the same initial and incoming data, let $Q_{3,+}^K(f_1,f_2)$ denote the gain in the third component produced by one forward reaction, and let $\G_{\eps,3}$ be the corresponding characteristic Duhamel operator. Since the third and fourth free components vanish, we obtain
	\begin{equation}\label{eq:intro-one-reaction}
		H_3^\eps=\eps\G_{\eps,3}Q_{3,+}^K(U_1^\eps,U_2^\eps)+O(\eps^2).
	\end{equation}
	The third component gives the leading interaction of the two incoming populations without differentiating the boundary map.
	
	The solution space provides integrable bounds in velocity that are uniform in $(s,x)$, but it does not define a trace at every boundary point. We define the outgoing characteristic representative almost everywhere and estimate its integral against the detector response. Parametrizing the boundary by lines parallel to $e+\eps w$ gives the exact flux Jacobian and relates the scaled functional to the measured flux.
	
	For a fixed incoming relative velocity, the limiting velocities of the third species form a known sphere in the translated coordinates. A smooth velocity filter can be chosen to give any prescribed angular test on that sphere. The synchronized incoming pulses make the leading total flight time independent of the reaction point along each chord. The detector then gives the spatial line integral in \eqref{eq:intro-tested-xray}, and Fourier inversion recovers the kernel.
	
	\para{Earlier literature} Inverse problems for linear transport use the decomposition of the albedo kernel into ballistic, single scattering, and multiple scattering terms; see \cite{ChoulliStefanov1996,BalJollivet2008,Chung2021}. The contribution from a single scattering event contains line integrals or integrals along broken rays. We also separate the contributions according to the number of collisions, but here the first product signal comes from a bilinear interaction of two incoming populations.
	
	For the stationary nonlinear Boltzmann equation, \cite{LaiUhlmannYang2021} uses mixed second order linearization to recover information about the collision kernel from concentrated free solutions. Its kernel is independent of position, and uniqueness is proved under a monotonicity condition. Recovery near a Maxwellian equilibrium is studied in \cite{LiOuyang2023} by linearizing once under additional assumptions on the kernel. Logarithmic stability for a structured kernel depending on time is obtained in \cite{LaiYan2024} through second order linearization and the light ray transform. Nonlinear particle interactions in the relativistic Boltzmann equation are used in \cite{BalehowskyKujanpaaLassasLiimatainen2022} to recover a Lorentzian metric.
	
	Related inverse problems for semilinear transport equations are studied in \cite{LaiUhlmannZhou2024,LaiZhou2024NonlinearTransport}, using linearization together with Carleman estimates or weighted light ray transforms. The coupled nonscattering model in \cite{KlingenbergLaiLi2021} gives reconstruction of an emission coefficient from one experimental configuration. These models have local nonlinearities in place of the binary collision operator considered here.
	
	Concentrated beams with large velocities are used in \cite{LaiLiSun2024} to recover the doping profile in the stationary Vlasov--Poisson system through the X-ray transform of the self-consistent field. The BGK equation near a Maxwellian is studied in \cite{LaiZhou2026BGK}, where concentrated incoming data are applied to the linearized equations to recover a collision frequency. An affine Radon transform is obtained in \cite{Lin2026kHessian} by studying a fully nonlinear elliptic equation as the size of the boundary data tends to infinity. For the infinity Laplacian, \cite{Lin2026InfinityLaplacian} uses large affine boundary data to recover a positive potential. The asymptotic boundary measurements determine weighted chord integrals, and combining measurements in opposite directions gives the X-ray transform.
	
	Here the two incident species have a large common velocity and a fixed relative velocity. The reaction configuration stays fixed, while the time available for collisions tends to zero. Since no product species is injected, its outgoing flux has no ballistic contribution. We derive the spatial X-ray transform of a forward kernel depending on position from these product measurements, without differentiating the boundary map.
	
	The reaction kinematics and the scalar reciprocity law belong to the kinetic theory of reacting gas mixtures; see \cite{CercignaniIllnerPulvirenti1994,RossaniScarfone2003,AnwasiaEtAl2019,SarnaOblapenkoTorrilhon2020}. Cutoff hard potential reactive kernels are studied in \cite{BondesanTang2024}. Our forward estimates use bounded angularly integrated rates in both reaction directions. This assumption is stronger than an angular cutoff and is stated separately from the local regularity required for reconstruction.
	
	\para{Organization of the paper} Section~\ref{sec:model} introduces the collision operators and proves the estimates needed for the forward problem. Section~\ref{sec:scaled-forward} establishes well-posedness for large common velocities, derives the leading contribution from a single reaction, and defines the outgoing flux. Section~\ref{sec:measurements} constructs the incoming pulses and detector responses and proves the limit giving the X-ray transform. Section~\ref{sec:recovery} gives the uniqueness and reconstruction proofs, the stability identities, and the estimates for finite detector resolution and errors in the measured flux. It also discusses the interpretation and scope of the measurements.

	\section{Reactive collision model and cutoff estimates}\label{sec:model}
	
	\subsection{Reaction kinematics and microscopic reversibility}\label{subsec:kinematics}
	
	Let $m_i>0$ and $E_i\in\R$, $1\le i\le4$, be the masses and internal energies. We assume
	\begin{equation}\label{eq:mass-balance}
		m_1+m_2=m_3+m_4=:M.
	\end{equation}
	We use the nonrelativistic approximation in which the mass defect is neglected, and the reaction energy is included in the internal energies. Set $\Delta E=E_3+E_4-E_1-E_2$ and assume $\Delta E\ge0$. The reduced masses are $\mu_{12}=m_1m_2/M$ and $\mu_{34}=m_3m_4/M$. If the forward reaction is exothermic, the two channel pairs may be interchanged so that the energy gap is nonnegative.
	
	Set $g_{\rm th}=\Big(\frac{2\Delta E}{\mu_{12}}\Big)^{1/2}$. For $r\ge g_{\rm th}$, define
	\begin{equation}\label{eq:rho-plus}
		\rho_+(r)=\Big[\frac{\mu_{12}}{\mu_{34}}\Big(r^2-\frac{2\Delta E}{\mu_{12}}\Big)\Big]^{1/2},
	\end{equation}
	and, for $s\ge0$, define
	\begin{equation}\label{eq:rho-minus}
		\rho_-(s)=\Big[\frac{\mu_{34}}{\mu_{12}}s^2+\frac{2\Delta E}{\mu_{12}}\Big]^{1/2}.
	\end{equation}
	Then $\rho_-(\rho_+(r))=r$ for $r\ge g_{\rm th}$ and $\rho_+(\rho_-(s))=s$ for $s\ge0$. We use the forward reaction maps only when $|g|>g_{\rm th}$ and set $\rho_+(g_{\rm th})=0$ to extend the inverse relation continuously to the threshold.
	
	Given incoming velocities $v_1,v_2$, set $g=v_1-v_2$ and $V_{12}=(m_1v_1+m_2v_2)/M$. If $|g|>g_{\rm th}$ and $\omega\in\Sph$, define the forward reaction maps by
	\begin{equation}\label{eq:forward-map}
		\mathcal V_3^+(v_1,v_2,\omega)=V_{12}+\frac{m_4}{M}\rho_+(|g|)\omega, \quad
		\mathcal V_4^+(v_1,v_2,\omega)=V_{12}-\frac{m_3}{M}\rho_+(|g|)\omega.
	\end{equation}
	If $v_3=\mathcal V_3^+(v_1,v_2,\omega)$ and $v_4=\mathcal V_4^+(v_1,v_2,\omega)$, then $v_3-v_4=\rho_+(|g|)\omega$ and $(m_3v_3+m_4v_4)/M=V_{12}$.
	
	For reverse incoming velocities $v_3,v_4$, set $h=v_3-v_4$ and $V_{34}=(m_3v_3+m_4v_4)/M$. If $\theta\in\Sph$, define the reverse reaction maps by
	\begin{equation}\label{eq:reverse-map}
		\mathcal V_1^-(v_3,v_4,\theta)=V_{34}+\frac{m_2}{M}\rho_-(|h|)\theta, \quad
		\mathcal V_2^-(v_3,v_4,\theta)=V_{34}-\frac{m_1}{M}\rho_-(|h|)\theta.
	\end{equation}
	If $v_1=\mathcal V_1^-(v_3,v_4,\theta)$ and $v_2=\mathcal V_2^-(v_3,v_4,\theta)$, then
	$v_1-v_2=\rho_-(|h|)\theta$ and $(m_1v_1+m_2v_2)/M=V_{34}$.
	
	The reaction maps give $V_{12}=V_{34}$, and we denote this common center of mass velocity by $V$. It depends on the incoming pair and is unchanged by the reaction. The definitions of $\rho_+$ and $\rho_-$ also give conservation of total kinetic plus internal energy.
	
	For $z\in\R^3\setminus\{0\}$, write $\widehat z=z/|z|$. Let $\C_+=\{(g,\omega)\in\R^3\times\Sph:|g|>g_{\rm th}\}$ and $\C_-=(\R^3\setminus\{0\})\times\Sph$. Define
	\begin{equation}\label{eq:reaction-involution}
		\mathfrak c_+(g,\omega)=\big(\rho_+(|g|)\omega,\widehat g\big), \quad
		\mathfrak c_-(h,\theta)=\big(\rho_-(|h|)\theta,\widehat h\big).
	\end{equation}
	Then $\mathfrak c_-$ is the inverse of $\mathfrak c_+$.
	
	\begin{lemma}[Reactive change of variables]\label{lem:reactive-jacobian}
		Let $(v_3,v_4,\theta)$ be obtained from $(v_1,v_2,\omega)$ by \eqref{eq:forward-map} and $\theta=\widehat{v_1-v_2}$. If $r=|v_1-v_2|$ and $s=|v_3-v_4|=\rho_+(r)$, then
		\begin{equation}\label{eq:reactive-jacobian}
			dv_3\,dv_4\,d\theta=\frac{\mu_{12}}{\mu_{34}}\frac{s}{r}\,dv_1\,dv_2\,d\omega,
		\end{equation}
		where $d\theta$ and $d\omega$ denote the standard surface measures on $\Sph$. The transformation is inverted by \eqref{eq:reverse-map} and $\omega=\widehat{v_3-v_4}$.
	\end{lemma}
	
	\begin{proof}
		Set $g=r\theta$ and $h=s\omega$, where $s=\rho_+(r)$. The common center-of-mass velocity is $V=(m_1v_1+m_2v_2)/M=(m_3v_3+m_4v_4)/M$. Each of the linear transformations $(v_1,v_2)\mapsto(V,g)$ and $(v_3,v_4)\mapsto(V,h)$ has absolute determinant one. In polar coordinates,
		\[
		dv_1\,dv_2\,d\omega=dV\,r^2\,dr\,d\theta\,d\omega,
		\quad dv_3\,dv_4\,d\theta=dV\,s^2\,ds\,d\omega\,d\theta,
		\]
		where $dV$ denotes Lebesgue measure in the center-of-mass velocity variable $V\in\R^3$. Differentiating $s^2=(\mu_{12}/\mu_{34})r^2-2\Delta E/\mu_{34}$ gives $s\,ds=(\mu_{12}/\mu_{34})r\,dr$, which proves \eqref{eq:reactive-jacobian}. The identities $r=\rho_-(s)$, $g=r\theta$, and $\omega=\widehat h$ recover the original incoming pair through \eqref{eq:reverse-map}.
	\end{proof}
	
	\begin{lemma}[Galilean covariance]\label{lem:galilean}
		For every $q\in\R^3$, the forward reaction maps satisfy
		\[
		\mathcal V_i^+(v_1+q,v_2+q,\omega)=\mathcal V_i^+(v_1,v_2,\omega)+q, \quad i=3,4.
		\]
		The reverse reaction maps satisfy
		\[
		\mathcal V_i^-(v_3+q,v_4+q,\theta)=\mathcal V_i^-(v_3,v_4,\theta)+q,\quad i=1,2.
		\]
	\end{lemma}
	
	\begin{proof}
		A common velocity translation leaves $g$ and $h$ unchanged and sends each center-of-mass velocity $V$ to $V+q$. The formulas follow from \eqref{eq:forward-map} and \eqref{eq:reverse-map}.
	\end{proof}
	
	\para{The unknown kernel and reciprocity} The coefficient to be recovered is the forward kernel $K^+$. The reverse kernel $K^-$ is not an independent unknown. A forward collision configuration $(g,\omega)\in\C_+$ is associated with the reverse configuration
	\[
	(h,\theta)=\mathfrak c_+(g,\omega)=\big(\rho_+(|g|)\omega,\widehat g\big)\in\C_-.
	\]
	We impose the following reciprocity relation on ordered collision configurations. In the rotationally invariant case, it is the microscopic reversibility relation used in \cite{RossaniScarfone2003,BondesanTang2024}. For a general anisotropic kernel, the relation below is an explicit assumption of the model.
	
	\begin{definition}[Reciprocal kernels]\label{def:microreversible-kernel}
		Let $K^+:\ol\Omega\times\C_+\to[0,\infty)$ and $K^-:\ol\Omega\times\C_-\to[0,\infty)$ be nonnegative measurable kernels. We call $(K^+,K^-)$ reciprocal if
		\begin{equation}\label{eq:microreversibility}
			|g|K^+(x,g,\omega)=\Big(\frac{\mu_{34}}{\mu_{12}}\Big)^2|h|K^-(x,h,\theta)
		\end{equation}
		whenever $(h,\theta)=\mathfrak c_+(g,\omega)$.
	\end{definition}
	
	The map $\mathfrak c_+$ is a bijection from $\C_+$ onto $\C_-$ with inverse $\mathfrak c_-$. For $(h,\theta)\in\C_-$, substitute $(g,\omega)=\mathfrak c_-(h,\theta)=(\rho_-(|h|)\theta,\widehat h)$ into \eqref{eq:microreversibility}. This determines $K^-$ uniquely as
	\begin{equation}\label{eq:explicit-reverse-kernel}
		K^-(x,h,\theta)=\Big(\frac{\mu_{12}}{\mu_{34}}\Big)^2\frac{\rho_-(|h|)}{|h|}K^+\big(x,\rho_-(|h|)\theta,\widehat h\big),\quad h\ne0.
	\end{equation}
	Formula \eqref{eq:explicit-reverse-kernel} follows from the reciprocity condition. In integrals over all $g\in\R^3$, we extend $K^+$ by zero to $|g|\le g_{\rm th}$. Whenever a forward collision integral is written over all incoming velocities, its integrand is defined to be zero for $|g|\le g_{\rm th}$, without evaluating the reaction maps or radius factors there. The value of $K^-$ at $h=0$ is arbitrary, since this set does not contribute to the collision integrals.
	
	We next express the two event measures in the same variables. Suppose that $(v_3,v_4,\theta)$ is obtained from $(v_1,v_2,\omega)$ by the forward reaction map, with $\theta=\widehat{v_1-v_2}$. Set $r=|v_1-v_2|$ and $s=|v_3-v_4|=\rho_+(r)$. Lemma~\ref{lem:reactive-jacobian} and \eqref{eq:microreversibility} give
	\begin{equation}\label{eq:event-measure-reciprocity}
		\begin{aligned}
			K^-(x,h,\theta)\,dv_3\,dv_4\,d\theta
			=\frac{\mu_{12}}{\mu_{34}}\frac{s}{r}K^-(x,h,\theta)\,dv_1\,dv_2\,d\omega
			=\Big(\frac{\mu_{12}}{\mu_{34}}\Big)^3K^+(x,g,\omega)\,dv_1\,dv_2\,d\omega.
		\end{aligned}
	\end{equation}
	In \eqref{eq:event-measure-reciprocity}, the first equality is the reactive change of variables and the second is the reciprocity relation.
	
	When $K^+(x,g,\omega)=k^+(x,|g|,\widehat g\cdot\omega)$, the reaction involution preserves the scalar scattering angle, and \eqref{eq:microreversibility} reduces to the usual scalar relation. For anisotropic laws, the directions in \eqref{eq:reaction-involution} specify the convention used here; no symmetry of $K^+$ under an interchange of its two directional arguments is assumed. The inverse proof uses the reverse kernel only through its bound on the collision rate.
	
	For the gain in the measured component, set $c=\mu_{12}/\mu_{34}$, $h=v-v_4$, and $v_i^-=\mathcal V_i^-(v,v_4,\theta)$. Formula \eqref{eq:event-measure-reciprocity} gives
	\begin{equation}\label{eq:third-gain-reciprocal-normalization}
		Q_{3,+}^K(f_1,f_2)(x,v)=c^{-3}\int_{\R^3}\int_{\Sph}K^-(x,h,\theta)f_1(v_1^-)f_2(v_2^-)\,d\theta\,dv_4.
	\end{equation}
	The factor $c^{-3}$ comes from using the reverse event variables. The forward representation \eqref{eq:third-gain-weak} has no such factor.
	
	\begin{definition}[Admissible reactive kernels]\label{def:kernel-class}
		Fix bounded open sets
		\begin{equation}\label{eq:collision-windows}
			\mathscr G\Subset\mathscr G^\sharp\Subset\{g\in\R^3:\,|g|>g_{\rm th}\},
		\end{equation}
		and constants $M_0,M_1>0$.
		
		For a nonnegative function $K^+:\ol\Omega\times\C_+\to[0,\infty)$, let $K^-$ be the reciprocal kernel determined by \eqref{eq:explicit-reverse-kernel}. When $K^+$ is used in an integral over all $g\in\R^3$, it is extended by zero to $|g|\le g_{\rm th}$. Define the essential suprema in the spatial variable
		\begin{equation}\label{eq:reactive-spatial-majorants}
			\ol K^+(g,\omega)=\esssup_{x\in\Omega}K^+(x,g,\omega),
			\quad \ol K^-(h,\theta)=\esssup_{x\in\Omega}K^-(x,h,\theta),
		\end{equation}
		and the forward and reverse angularly integrated rates
		\begin{equation}\label{eq:Lambda-reactive}
			\begin{aligned}
				\Lambda_+&:=\esssup_{g\in\R^3}\int_{\Sph}\ol K^+(g,\omega)\,d\omega,\\
				\Lambda_-&:=\esssup_{h\in\R^3}\int_{\Sph}\ol K^-(h,\theta)\,d\theta.
			\end{aligned}
		\end{equation}
		We assume the rate bound
		\begin{equation}\label{eq:kernel-finite-rate-class}
			\Lambda_++\Lambda_-\le M_0.
		\end{equation}
		
		The admissible class is
		\begin{equation}\label{eq:admissible-kernel-class}
			\mathfrak K(\mathscr G^\sharp,M_0,M_1):=\left\{K^+:\ol\Omega\times\C_+\to[0,\infty)\ \middle|\
			\begin{array}{l}
				K^+\ \text{is measurable},\quad \Lambda_++\Lambda_-\le M_0,\\[1mm]
				K^+\big|_{\ol\Omega\times\ol{\mathscr G^\sharp}\times\Sph}\in
				C^1\big(\ol\Omega\times\ol{\mathscr G^\sharp}\times\Sph\big),\\[1mm]
				\|K^+\|_{C^1(\ol\Omega\times\ol{\mathscr G^\sharp}\times\Sph)}\le M_1
			\end{array}
			\right\}.
		\end{equation}
		We write $\mathfrak K=\mathfrak K(\mathscr G^\sharp,M_0,M_1)$. The value of $K^-$ at $h=0$ is arbitrary, since this set does not contribute to the collision integrals. The coefficient to be recovered is denoted by $K=K^+$. The reverse kernel is fixed by microscopic reversibility and is not an independent unknown.
	\end{definition}
	
	\begin{remark}[Threshold compatibility]\label{rem:reverse-threshold}
		When $\Delta E>0$, let $s=|h|$ and $r=\rho_-(s)$. The identity $r^2=g_{\rm th}^2+(\mu_{34}/\mu_{12})s^2$ gives $r\downarrow g_{\rm th}>0$ as $s\downarrow0$. Low reverse relative speeds correspond to forward relative speeds near the threshold. The factor $r/s$ in \eqref{eq:explicit-reverse-kernel} shows that the rate bound for $\ol K^-$ imposes a condition on $K^+$ near the threshold. This condition does not follow from the local $C^1$ bound on $\mathscr G^\sharp$, which is separated from the threshold.
		
		For example, suppose that the forward kernel is rotationally invariant and satisfies
		\begin{equation}\label{eq:threshold-compatible-example}
			K^+(x,g,\omega)\le C\frac{\rho_+(|g|)}{|g|} b_0(\widehat g\cdot\omega),	\quad |g|>g_{\rm th},
		\end{equation}
		where $b_0\ge0$ and $\int_{-1}^1b_0(z)\,dz<\infty$. Let $h\ne0$, set $s=|h|$ and $r=\rho_-(s)$. Since $\rho_+(r)=s$, formula \eqref{eq:explicit-reverse-kernel} gives
		\begin{align*}
			K^-(x,h,\theta)=\Big(\frac{\mu_{12}}{\mu_{34}}\Big)^2\frac{r}{s}K^+(x,r\theta,\widehat h)\le C\Big(\frac{\mu_{12}}{\mu_{34}}\Big)^2\frac{r}{s}\frac{\rho_+(r)}{r}
			b_0(\theta\cdot\widehat h)=C\Big(\frac{\mu_{12}}{\mu_{34}}\Big)^2 b_0(\theta\cdot\widehat h).
		\end{align*}
		Consequently,
		\[
		\int_{\Sph}\ol K^-(h,\theta)\,d\theta \le 2\pi C\Big(\frac{\mu_{12}}{\mu_{34}}\Big)^2\int_{-1}^1b_0(z)\,dz.
		\]
		Moreover,
		\[
		\int_{\Sph}\ol K^+(g,\omega)\,d\omega \le 2\pi C\frac{\rho_+(|g|)}{|g|} \int_{-1}^1b_0(z)\,dz.
		\]
		Since $\rho_+(r)/r\le(\mu_{12}/\mu_{34})^{1/2}$, both angularly integrated rates are uniformly bounded.
		
		The compatibility condition is used only in the reversible forward problem. The inverse argument is unchanged for an irreversible
		$1+2\to3+4$ model, or when the reverse kernel is prescribed independently and satisfies the same bound on the collision rate.
	\end{remark}
	
	Fix an open set $\mathcal O\Subset\Sph$. The collision configurations recovered in the paper form the set
	\begin{equation}\label{eq:accessible-collision-set}
		\mathcal C_{\rm acc}=\mathscr G\times\mathcal O.
	\end{equation}
	The larger set $\mathscr G^\sharp$ is used only to contain the relative velocities of the incident packets for small packet width. The total kernel may be nonzero outside $\mathscr G^\sharp$.
	
	\begin{remark}[Recovery on a prescribed collision set]\label{rem:windowed-unknown}
		Suppose that $K=K_{\rm bg}+q$, where $K_{\rm bg}$ is known and $K\in\mathfrak K$. Theorem~\ref{thm:main} determines $K$ and $q$ on $\Omega\times\mathcal C_{\rm acc}=\Omega\times\mathscr G\times\mathcal O$. The kernel $K$ may be nonzero outside $\mathcal C_{\rm acc}$. If $q$ is known to vanish outside $\mathcal C_{\rm acc}$ in those variables, the same measurements determine $q$ everywhere.
	\end{remark}

	\subsection{Reactive and elastic collision operators}\label{subsec:collision-operators}
	
	As in the Introduction, $Q_i^{\rm r}(F)$, $Q_i^{\rm el}(F)$, and $\Q_{K,i}(F)$ are quadratic operators on $F=(F_1,F_2,F_3,F_4)$. The notation with two scalar arguments, such as $Q_{3,+}^{K}(f_1,f_2)$ or $Q_{ij,+}^{\rm el}(f_i,f_j)$, denotes a bilinear gain for a specified reaction channel or elastic species pair. We define the forward and reverse reactive operators separately in their incoming event variables. No symmetry of the ordered anisotropic kernel is assumed.

	\begin{definition}[Reactive collision operator]\label{def:event-formulation}
		Let $K^+\in\mathfrak K$, and let $K^-$ be its reciprocal kernel. For a bounded test vector $\Phi=(\phi_1,\phi_2,\phi_3,\phi_4)$, define the reactive collision operator by
		\begin{equation}\label{eq:event-weak}
			\begin{split}
				&\sum_{i=1}^4\int_{\R^3}Q_i^{\rm r}(F)(x,v)\phi_i(v)\,dv\\
				&=\int_{\R^3}\int_{\R^3}\int_{\Sph}K^+(x,v_1-v_2,\omega)F_1(x,v_1)F_2(x,v_2)\\
				&\quad\quad\quad \quad \quad \times
				\big[\phi_3\big(\mathcal V_3^+(v_1,v_2,\omega)\big)+\phi_4\big(\mathcal V_4^+(v_1,v_2,\omega)\big)-\phi_1(v_1)-\phi_2(v_2)\big]\,d\omega\,dv_2\,dv_1\\
				&\quad\, +	\int_{\R^3}\int_{\R^3}\int_{\Sph}K^-(x,v_3-v_4,\theta)F_3(x,v_3)F_4(x,v_4)\\
				&\quad\quad \quad \quad \quad \times
				\big[\phi_1\big(\mathcal V_1^-(v_3,v_4,\theta)\big)	+\phi_2\big(\mathcal V_2^-(v_3,v_4,\theta)\big)-\phi_3(v_3)-\phi_4(v_4)\big]	\,d\theta\,dv_4\,dv_3.
			\end{split}
		\end{equation}
		The first integral uses the forward incoming velocities $(v_1,v_2)$ and the product direction $\omega$; the second uses the reverse incoming velocities $(v_3,v_4)$ and the product direction $\theta$. These are independent integrations, rather than an integration over one common quadruple $(v_1,v_2,v_3,v_4)$.
		
		For every bounded test function $\zeta$ on $\R^3$, the forward gain in the third component is characterized by
		\begin{equation}\label{eq:third-gain-weak}
			\begin{split}
				&\int_{\R^3}Q_{3,+}^{K}(F_1,F_2)(x,v)\zeta(v)\,dv\\
				&=\int_{\R^3}\int_{\R^3}\int_{\Sph}K^+(x,v_1-v_2,\omega)F_1(x,v_1)F_2(x,v_2)
				\zeta\big(\mathcal V_3^+(v_1,v_2,\omega)\big)\,d\omega\,dv_2\,dv_1.
			\end{split}
		\end{equation}
	\end{definition}
	
	For almost every $x\in\Omega$, every component of the reactive operator is the signed pushforward of the corresponding event measure. A gain term is obtained by pushing this measure to the velocity of a produced species, whereas a loss term is evaluated at the incoming velocity of a consumed species. Lemma~\ref{lem:reactive-jacobian} and the reciprocity relation \eqref{eq:microreversibility} relate the event formulation to the usual gain--loss representation while preserving the ordering of the variables in an anisotropic kernel.
	
	For $|g|>g_{\rm th}$ and $h\in\R^3$, set
	\begin{equation}\label{eq:reaction-radius-factors}
		\begin{aligned}
			\lambda_3(g)&=\frac{m_4}{M}\rho_+(|g|), & \lambda_4(g)&=\frac{m_3}{M}\rho_+(|g|),\\
			\lambda_1(h)&=\frac{m_2}{M}\rho_-(|h|), & \lambda_2(h)&=\frac{m_1}{M}\rho_-(|h|).
		\end{aligned}
	\end{equation}
	We first write the forward gain in the third component as a density. Let $\zeta\in L^\infty(\R^3)$. With $V=(m_1v_1+m_2v_2)/M$ and $g=v_1-v_2$, the inverse formulas are $v_1=V+m_2g/M$ and $v_2=V-m_1g/M$, and the Jacobian is one. The third product velocity is $V+\lambda_3(g)\omega$, so \eqref{eq:third-gain-weak} becomes
	\begin{equation}\label{eq:third-gain-pushforward}
		\begin{split}
			&\int_{\R^3}Q_{3,+}^{K}(f_1,f_2)(x,v)\zeta(v)\,dv\\
			&=\int_{\R^3}\int_{|g|>g_{\rm th}}\int_{\Sph}K^+(x,g,\omega)f_1\Big(V+\frac{m_2}{M}g\Big)
			f_2\Big(V-\frac{m_1}{M}g\Big)\zeta\big(V+\lambda_3(g)\omega\big)\,d\omega\,dg\,dV\\
			&= \int_{\R^3}\zeta(v)\int_{|g|>g_{\rm th}}\int_{\Sph} K^+(x,g,\omega)f_1\Big(v-\lambda_3(g)\omega+\frac{m_2}{M}g\Big)
			f_2\Big(v-\lambda_3(g)\omega-\frac{m_1}{M}g\Big)\,d\omega\,dg\,dv,
		\end{split}
	\end{equation}
	where the last equality uses $v=V+\lambda_3(g)\omega$ for fixed $(g,\omega)$, with $dV=dv$. Since $\zeta$ is arbitrary, the gain density is
	\begin{equation}\label{eq:third-gain-strong}
		\begin{split}
			&Q_{3,+}^{K}(f_1,f_2)(x,v)\\&=\int_{|g|>g_{\rm th}}\int_{\Sph}K^+(x,g,\omega)f_1\Big(v-\lambda_3(g)\omega+\frac{m_2}{M}g\Big)
			f_2\Big(v-\lambda_3(g)\omega-\frac{m_1}{M}g\Big)\,d\omega\,dg.
		\end{split}
	\end{equation}
	After extending $K^+$ by zero to $|g|\le g_{\rm th}$, the integral in \eqref{eq:third-gain-strong} may be taken over all $g\in\R^3$. For fixed $(g,\omega)$, the observed velocity $v$ determines the center-of-mass velocity as $V=v-\lambda_3(g)\omega$.
	
	\begin{proposition}[Reactive gains in $L^1$]\label{prop:all-reactive-gains}
		Let $K^+\in\mathfrak K$, and let $K^-$ be the reverse kernel determined by microscopic reversibility. Let $f_i\in L^1(\R^3)$ for $1\le i\le4$. For almost every $x\in\Omega$, every gain term determined by \eqref{eq:event-weak} is an $L^1$ function of its output velocity.
		
		In addition to \eqref{eq:third-gain-strong}, the forward gain in the fourth component is
		\begin{equation}\label{eq:fourth-gain-strong}
			\begin{split}
				&Q_{4,+}^{K}(f_1,f_2)(x,v)\\
				&=\int_{|g|>g_{\rm th}}\int_{\Sph}
				K^+(x,g,\omega)
				f_1\Big(v+\lambda_4(g)\omega+\frac{m_2}{M}g\Big)
				f_2\Big(v+\lambda_4(g)\omega-\frac{m_1}{M}g\Big)
				\,d\omega\,dg.
			\end{split}
		\end{equation}
		The reverse gain in the first component is
		\begin{equation}\label{eq:first-reverse-gain}
			\begin{split}
				&Q_{1,+}^{K^-}(f_3,f_4)(x,v)\\
				&=\int_{\R^3}\int_{\Sph}
				K^-(x,h,\theta)
				f_3\Big(v-\lambda_1(h)\theta+\frac{m_4}{M}h\Big)
				f_4\Big(v-\lambda_1(h)\theta-\frac{m_3}{M}h\Big)
				\,d\theta\,dh,
			\end{split}
		\end{equation}
		and the reverse gain in the second component is
		\begin{equation}\label{eq:second-reverse-gain}
			\begin{split}
				&Q_{2,+}^{K^-}(f_3,f_4)(x,v)\\
				&=\int_{\R^3}\int_{\Sph}
				K^-(x,h,\theta)
				f_3\Big(v+\lambda_2(h)\theta+\frac{m_4}{M}h\Big)
				f_4\Big(v+\lambda_2(h)\theta-\frac{m_3}{M}h\Big)
				\,d\theta\,dh.
			\end{split}
		\end{equation}
		Moreover, the following estimates hold for almost every $x\in\Omega$, with constants independent of $x$:
		\begin{equation}\label{eq:reactive-gain-L1-bounds}
			\begin{aligned}
				\|Q_{3,+}^{K}(f_1,f_2)(x,\cdot)\|_{L^1}+\|Q_{4,+}^{K}(f_1,f_2)(x,\cdot)\|_{L^1}&
				\le2\Lambda_+\|f_1\|_{L^1}\|f_2\|_{L^1},\\
				\|Q_{1,+}^{K^-}(f_3,f_4)(x,\cdot)\|_{L^1}+\|Q_{2,+}^{K^-}(f_3,f_4)(x,\cdot)\|_{L^1}
				&\le2\Lambda_-\|f_3\|_{L^1}\|f_4\|_{L^1}.
			\end{aligned}
		\end{equation}
		If the inputs are nonnegative, then the $L^1$ norm of each gain is equal to the corresponding total event rate.
	\end{proposition}
	
	\begin{proof}
		Formula \eqref{eq:third-gain-strong} was derived above. For the fourth component, use the same variables $(V,g)$. The fourth product velocity is $v_4=V-\lambda_4(g)\omega$. For a prescribed output velocity $v$, one has $V=v+\lambda_4(g)\omega$. Substitution into $v_1=V+m_2g/M$ and $v_2=V-m_1g/M$ gives \eqref{eq:fourth-gain-strong}.
		
		For the reverse reaction, introduce $W=\frac{m_3v_3+m_4v_4}{M}$, $h=v_3-v_4$. The inverse formulas are $v_3=W+\frac{m_4}{M}h$, $v_4=W-\frac{m_3}{M}h$, and $dv_3\,dv_4=dW\,dh$. The first reverse product velocity is $v_1=W+\lambda_1(h)\theta$. For a prescribed value $v$ of the output velocity of the first component, one has $W=v-\lambda_1(h)\theta$. Substitution into the formulas for $v_3$ and $v_4$ gives \eqref{eq:first-reverse-gain}. Similarly, the second reverse product velocity is $v_2=W-\lambda_2(h)\theta$. Thus, for a prescribed output velocity of the second component $v$, one has $W=v+\lambda_2(h)\theta$. This gives \eqref{eq:second-reverse-gain}.
		
		We prove the first estimate in \eqref{eq:reactive-gain-L1-bounds}. By Tonelli's theorem and the event representation,
		\begin{align*}
			\|Q_{3,+}^{K}(f_1,f_2)(x,\cdot)\|_{L^1}&\le \int_{\R^3}\int_{\R^3} \bigg[\int_{\Sph}
			K^+(x,v_1-v_2,\omega)\,d\omega\bigg] |f_1(v_1)||f_2(v_2)|\,dv_2\,dv_1\\
			&\le \int_{\R^3}\int_{\R^3} \bigg[\int_{\Sph} \ol K^+(v_1-v_2,\omega)\,d\omega\bigg]
			|f_1(v_1)||f_2(v_2)|\,dv_2\,dv_1\\ &\le\Lambda_+\|f_1\|_{L^1}\|f_2\|_{L^1}.
		\end{align*}
		The same argument applied to the forward event and reverse event measures gives
		\begin{equation}
			\begin{split}
				\|Q_{4,+}^{K}(f_1,f_2)(x,\cdot)\|_{L^1}&\le \Lambda_+\|f_1\|_{L^1}\|f_2\|_{L^1}, \\ \|Q_{1,+}^{K^-}(f_3,f_4)(x,\cdot)\|_{L^1}&\le \Lambda_-\|f_3\|_{L^1}\|f_4\|_{L^1}, \\
				\|Q_{2,+}^{K^-}(f_3,f_4)(x,\cdot)\|_{L^1}&\le \Lambda_-\|f_3\|_{L^1}\|f_4\|_{L^1}.
			\end{split}
		\end{equation}
		These four estimates prove \eqref{eq:reactive-gain-L1-bounds}. For nonnegative inputs, every gain is nonnegative. Integrating in the output velocity, or testing the weak formula with the constant function one, gives its $L^1$ norm as the total rate of the corresponding events.
	\end{proof}
	
	The loss terms are evaluated at the incoming velocity of the species removed by the reaction. Define
	\begin{equation}\label{eq:reactive-loss-frequencies}
		\begin{aligned}
			\nu_1^{\rm fwd}[f_2](x,v)&=\int_{\R^3}\int_{\Sph}K^+(x,v-v_2,\omega)f_2(v_2)
			\,d\omega\,dv_2,\\
			\nu_2^{\rm fwd}[f_1](x,v)&=\int_{\R^3}\int_{\Sph}K^+(x,v_1-v,\omega)f_1(v_1)
			\,d\omega\,dv_1,\\
			\nu_3^{\rm rev}[f_4](x,v)&=\int_{\R^3}\int_{\Sph}K^-(x,v-v_4,\theta)f_4(v_4)
			\,d\theta\,dv_4,\\
			\nu_4^{\rm rev}[f_3](x,v)&=\int_{\R^3}\int_{\Sph}K^-(x,v_3-v,\theta)f_3(v_3)
			\,d\theta\,dv_3.
		\end{aligned}
	\end{equation}
	The componentwise strong gain--loss form of the reactive operator is 
	\begin{equation}\label{eq:third-reactive-gain-loss}
		\begin{aligned}
			Q_1^{\rm r}(F)(x,v)
			&=Q_{1,+}^{K^-}(F_3,F_4)(x,v)
			-\nu_1^{\rm fwd}[F_2](x,v)F_1(x,v),\\
			Q_2^{\rm r}(F)(x,v)
			&=Q_{2,+}^{K^-}(F_3,F_4)(x,v)
			-\nu_2^{\rm fwd}[F_1](x,v)F_2(x,v),\\
			Q_3^{\rm r}(F)(x,v)
			&=Q_{3,+}^{K}(F_1,F_2)(x,v)
			-\nu_3^{\rm rev}[F_4](x,v)F_3(x,v),\\
			Q_4^{\rm r}(F)(x,v)
			&=Q_{4,+}^{K}(F_1,F_2)(x,v)
			-\nu_4^{\rm rev}[F_3](x,v)F_4(x,v).
		\end{aligned}
	\end{equation}
	In particular, the third equation contains the gain created by the forward reaction $1+2\to3+4$ and the loss created by the reverse reaction $3+4\to1+2$. The full third component of \eqref{eq:intro-system} is obtained by adding the known elastic term $Q_3^{\rm el}(F)$.
	
	\begin{remark}[Choice of the detected product channel]\label{rem:detected-product-channel}
		The choice of the third component is made only to fix notation. For a forward reaction with incoming relative velocity $g=v_1-v_2$ and center-of-mass velocity $V=\frac{m_1v_1+m_2v_2}{M}$, the product velocities satisfy 
		\[
		\mathcal V_3^+(v_1,v_2,\omega)-V=\lambda_3(g)\omega,\quad \text{and}\quad 
		\mathcal V_4^+(v_1,v_2,\omega)-V=-\lambda_4(g)\omega.
		\]
		Both product velocities are parametrized by the same $\omega$, through $\omega\mapsto\lambda_3(g)\omega$ and $\omega\mapsto-\lambda_4(g)\omega$. Only the known radius and the sign differ.
		
		For the incident beams used below, the recentered center of mass velocity tends to zero as the packets concentrate. To measure the fourth component, choose its velocity filter to equal $\varphi(\omega)$ at $-\lambda_4(g_0)\omega$, using this parametrization in place of $\lambda_3(g_0)\omega$. The gain \eqref{eq:fourth-gain-strong} then gives the same limiting integral $\int_{\Sph}K(x,g_0,\omega)\varphi(\omega)\,d\omega$ as the gain in the third component. The X-ray identity and recovery proof are unchanged. We use the third component throughout to avoid repeating the notation.
	\end{remark}
	
	\begin{proposition}[Conservation laws]\label{prop:conservation}
		Assume that the nonnegative distributions have sufficient velocity moments so that all the expressions below are absolutely integrable. Then, for almost every $x\in\Omega$, the reactive operator conserves the total number of particles, total mass, momentum, and total kinetic plus internal energy. More precisely,
		\begin{equation}\label{eq:conservation-number}
			\sum_{i=1}^4\int_{\R^3}Q_i^{\rm r}(F)(x,v)\,dv=0,
		\end{equation}
		\begin{equation}\label{eq:conservation-mass-momentum}
			\sum_{i=1}^4\int_{\R^3}m_iQ_i^{\rm r}(F)(x,v)\,dv=0, \quad \sum_{i=1}^4\int_{\R^3}m_ivQ_i^{\rm r}(F)(x,v)\,dv=0,
		\end{equation}
		and
		\begin{equation}\label{eq:conservation-energy}
			\sum_{i=1}^4\int_{\R^3}\Big(\frac12m_i|v|^2+E_i\Big)Q_i^{\rm r}(F)(x,v)\,dv=0.
		\end{equation}
	\end{proposition}
	
	\begin{proof}
		The identities follow by testing the event formulation \eqref{eq:event-weak} with the collision invariants. We first consider a forward event. For the total particle number, take $\phi_i=1$ for every $i$. The event bracket is $1+1-1-1=0$. Thus, every forward reaction removes two particles and creates two particles. The same calculation applies to the reverse event.
		
		For total mass, take $\phi_i=m_i$. The forward event bracket is $m_3+m_4-m_1-m_2=0$ by \eqref{eq:mass-balance}. This proves the first identity in 	\eqref{eq:conservation-mass-momentum}.
		
		For momentum, apply the event formulation componentwise with $\phi_i(v)=m_iv$. By the forward reaction maps,
		\begin{align*}
			m_3v_3+m_4v_4=m_3\Big(V+\frac{m_4}{M}\rho_+(|g|)\omega\Big)+m_4\Big(V-\frac{m_3}{M}\rho_+(|g|)\omega\Big)=(m_3+m_4)V=m_1v_1+m_2v_2.
		\end{align*}
		Hence, the momentum event bracket vanishes. The reverse maps satisfy the same identity with the incoming and product pairs interchanged.
		
		We finally consider the total kinetic plus internal energy. For any two species $i$ and $j$, set
		$\mu_{ij}=m_im_j/(m_i+m_j)$. If $V_{ij}$ is their center-of-mass velocity, then
		\[
		\frac12m_i|v_i|^2+\frac12m_j|v_j|^2=\frac12(m_i+m_j)|V_{ij}|^2+\frac12\mu_{ij}|v_i-v_j|^2.
		\]
		For a forward reaction, the center-of-mass velocity is the same for the two channel pairs, while $|v_1-v_2|=|g|$, $|v_3-v_4|=\rho_+(|g|)$. Consequently, the total energy of the product pair is
		\begin{align*}
			&\frac12m_3|v_3|^2+\frac12m_4|v_4|^2+E_3+E_4=\frac12M|V|^2
			+\frac12\mu_{34}\rho_+(|g|)^2+E_3+E_4.
		\end{align*}
		By the definition \eqref{eq:rho-plus} of $\rho_+$, we have $\frac12\mu_{34}\rho_+(|g|)^2=\frac12\mu_{12}|g|^2-\Delta E$. Since $\Delta E=E_3+E_4-E_1-E_2$, it follows that
		\begin{align*}
			\frac12M|V|^2+\frac12\mu_{34}\rho_+(|g|)^2+E_3+E_4&=\frac12M|V|^2
			+\frac12\mu_{12}|g|^2+E_1+E_2\\
			&=\frac12m_1|v_1|^2+\frac12m_2|v_2|^2+E_1+E_2.
		\end{align*}
		Thus, the energy event bracket vanishes. The reverse identity follows
		either from the same calculation using $\rho_-$ or from the fact
		that the reverse reaction map is the inverse of the forward map.
		
		Inserting these collision invariants into
		\eqref{eq:event-weak} proves
		\eqref{eq:conservation-number}--\eqref{eq:conservation-energy}.
		For the unbounded momentum and energy test functions, one first uses
		compactly supported truncations and then passes to the limit by
		dominated convergence under the stated moment assumptions.
	\end{proof}

	\para{Elastic collisions and the full operator} 
	For each ordered species pair $(i,j)$, let $B_{ij}(x,g,\sigma)\ge0$ be a known angular-cutoff elastic kernel. We assume the interchange symmetry $B_{ji}(x,-g,-\sigma)=B_{ij}(x,g,\sigma)$. Given incoming velocities $v,v_*\in\R^3$, set 
	\[
	M_{ij}=m_i+m_j,\quad V_{ij}=\frac{m_iv+m_jv_*}{M_{ij}}, \quad g=v-v_*.
	\]
	The elastic post-collisional velocities are
	\[
	v'=V_{ij}+\frac{m_j}{M_{ij}}|g|\sigma, \quad v_*'=V_{ij}-\frac{m_i}{M_{ij}}|g|\sigma, \quad \sigma\in\Sph.
	\]
	Thus, the center-of-mass velocity and the relative speed are preserved.
	
	For a bounded test vector $\Phi=(\phi_1,\phi_2,\phi_3,\phi_4)$, define the elastic collision operator by
	\begin{equation}\label{eq:elastic-event}
		\begin{split}
			\sum_{i=1}^4\int_{\R^3}Q_i^{\rm el}(F)(x,v)\phi_i(v)\,dv&=\frac12\sum_{i,j=1}^4	\int_{\R^3}\int_{\R^3}\int_{\Sph}B_{ij}(x,v-v_*,\sigma)F_i(x,v)F_j(x,v_*)\\
			&\quad\times \big[\phi_i(v')+\phi_j(v_*')-\phi_i(v)-\phi_j(v_*)\big]\,d\sigma\,dv_*\,dv.
		\end{split}
	\end{equation}
	
	The full collision operator in \eqref{eq:intro-system} is $\Q_{K,i}(F)=Q_i^{\rm el}(F)+Q_i^{\rm r}(F)$. Since $\Q_K$ is homogeneous quadratic, we define its symmetric bilinear polarization by \begin{equation}\label{eq:polarization}
		\B_K(F,G)=\frac12\big[\Q_K(F+G)-\Q_K(F)-\Q_K(G)\big].
	\end{equation}
	Then $\Q_K(F)=\B_K(F,F)$. The $Q$-notation with two arguments is used only for gains associated with a specified reaction channel or elastic species pair.
	
	For the elastic gain associated with $(i,j)$, write $\lambda_{ij}(g)=m_j|g|/M_{ij}$. The change of variables $(v,v_*)\mapsto(V_{ij},g)$ has Jacobian one, and $v'=V_{ij}+\lambda_{ij}(g)\sigma$. For fixed $(g,\sigma)$ and output velocity $u=v'$, substitute $V_{ij}=u-\lambda_{ij}(g)\sigma$ to obtain
	\begin{equation}\label{eq:elastic-gain-density}
		\begin{split}
			Q_{ij,+}^{\rm el}(f_i,f_j)(x,u)&=\int_{\R^3}\int_{\Sph}B_{ij}(x,g,\sigma)\\
			&\quad\times f_i\Big(u-\lambda_{ij}(g)\sigma+\frac{m_j}{M_{ij}}g\Big)
			f_j\Big(u-\lambda_{ij}(g)\sigma-\frac{m_i}{M_{ij}}g\Big)\,d\sigma\,dg.
		\end{split}
	\end{equation}
	Integrating this expression in $u$ gives the corresponding total elastic event rate. The elastic and reactive gains are all realized as $L^1$ densities under the bounds on the collision rates introduced below.
	
	\subsection{Estimates for bounded collision rates}\label{subsec:cutoff-estimates}
	
	For the known elastic kernels, define 
	\begin{equation}\label{eq:elastic-spatial-majorants}
		\ol B_{ij}(g,\sigma)=\esssup_{x\in\Omega}B_{ij}(x,g,\sigma)
	\end{equation}
	and assume
	\begin{equation}\label{eq:Lambda-el}
		\Lambda_{\rm el}=\frac12\sum_{i,j=1}^4\esssup_{g\in\R^3}
		\int_{\Sph}\ol B_{ij}(g,\sigma)\,d\sigma<\infty.
	\end{equation}
	The reactive rates $\Lambda_+$ and $\Lambda_-$ in \eqref{eq:Lambda-reactive} satisfy $\Lambda_++\Lambda_-\le M_0$ by \eqref{eq:kernel-finite-rate-class}. For an ordered anisotropic kernel, the forward angular integral does not control the reverse rate, because the reaction involution interchanges the incoming and outgoing directions. Both rate bounds are included in the admissible class.
	
	Let $(Z_0,\mathcal Z_0,\mu_0)$ be a $\sigma$-finite measure space, equip $Z=Z_0\times\Omega$ with the product measure, and write $z=(z_0,x)\in Z$. Define
	\begin{equation}\label{eq:XZ}
		\|F\|_{\mathfrak X(Z)}=\sum_{i=1}^4\int_{\R^3}\esssup_{z\in Z}|F_i(z,w)|\,dw.
	\end{equation}
	Thus, $\mathfrak X(Z)=L^1_w(L^\infty_z)^4$. Equipped with this norm, $\mathfrak X(Z)$ is a Banach space. We write $\mathbb L^1=(L^1(\R^3))^4$, equipped with the norm $\|F\|_{\mathbb L^1}=\sum_{i=1}^4\|F_i\|_{L^1}$, and use the symmetric bilinear polarization $\B_K$ defined in \eqref{eq:polarization}. We write the elastic part and the two reaction parts of \eqref{eq:polarization} separately.
	
	\begin{definition}[Bilinear event operators]
		\label{def:bilinear-event-operators}
		Let $F,G\in\mathbb L^1$ and let $\Phi=(\phi_1,\phi_2,\phi_3,\phi_4)$ with $\phi_i\in L^\infty(\R^3)$. For almost every $x\in\Omega$, define the bilinear operator for the forward reaction $\mathcal B_{\mathrm{fwd}}^K(F,G)$ by
		\begin{equation}\label{eq:bilinear-forward-event}
			\begin{split}
				\sum_{i=1}^4\int_{\R^3}\mathcal B_{\mathrm{fwd},i}^K(F,G)(x,v)\phi_i(v)\,dv
				&=\frac12\int_{\R^3}\int_{\R^3}\int_{\Sph}
				K^+(x,v_1-v_2,\omega)
				\big[F_1(v_1)G_2(v_2)+G_1(v_1)F_2(v_2)\big]\\
				&\quad\times \big[\phi_3(v_3)+\phi_4(v_4)
				-\phi_1(v_1)-\phi_2(v_2)\big]
				\,d\omega\,dv_2\,dv_1,
			\end{split}
		\end{equation}
		where $v_3=\mathcal V_3^+(v_1,v_2,\omega)$,	and $v_4=\mathcal V_4^+(v_1,v_2,\omega)$. The bilinear operator for the reverse reaction $\mathcal B_{\mathrm{rev}}^K(F,G)$ is defined by
		\begin{equation}\label{eq:bilinear-reverse-event}
			\begin{split}
				\sum_{i=1}^4\int_{\R^3}\mathcal B_{\mathrm{rev},i}^K(F,G)(x,v)\phi_i(v)\,dv
				&=\frac12\int_{\R^3}\int_{\R^3}\int_{\Sph}K^-(x,v_3-v_4,\theta)\big[F_3(v_3)G_4(v_4)+G_3(v_3)F_4(v_4)\big]\\
				&\quad\times\big[\phi_1(v_1)+\phi_2(v_2)-\phi_3(v_3)-\phi_4(v_4)\big]\,d\theta\,dv_4\,dv_3,
			\end{split}
		\end{equation}
		where $v_1=\mathcal V_1^-(v_3,v_4,\theta)$ and $v_2=\mathcal V_2^-(v_3,v_4,\theta)$. Finally, define the elastic bilinear operator
		$\mathcal B_{\mathrm{el}}(F,G)$ by
		\begin{equation}\label{eq:bilinear-elastic-event}
			\begin{split}
				\sum_{i=1}^4\int_{\R^3}\mathcal B_{\mathrm{el},i}(F,G)(x,v)\phi_i(v)\,dv&=\frac14\sum_{i,j=1}^4
				\int_{\R^3}\int_{\R^3}\int_{\Sph}B_{ij}(x,v-v_*,\sigma)\big[F_i(v)G_j(v_*)+G_i(v)F_j(v_*)\big]\\
				&\quad\times \big[\phi_i(v')+\phi_j(v_*')-\phi_i(v)-\phi_j(v_*)\big]\,d\sigma\,dv_*\,dv.
			\end{split}
		\end{equation}
		Here $(v',v_*')$ are the elastic velocities defined above \eqref{eq:elastic-event}. The operators $\mathcal B_{\mathrm{fwd}}^K$, $\mathcal B_{\mathrm{rev}}^K$, and $\mathcal B_{\mathrm{el}}$ are the symmetric bilinear polarizations of the forward reaction, reverse reaction, and elastic operators in \eqref{eq:event-weak} and \eqref{eq:elastic-event}. Formula \eqref{eq:polarization} gives
		\begin{equation}\label{eq:bilinear-collision-decomposition}
			\B_K(F,G)=\mathcal B_{\mathrm{el}}(F,G)+\mathcal B_{\mathrm{fwd}}^K(F,G)+\mathcal B_{\mathrm{rev}}^K(F,G).
		\end{equation}
	\end{definition}
	
	\begin{lemma}[Bounds for the event rates]\label{lem:event-rate-estimate}
		For every $F,G\in\mathbb L^1$, the bilinear operator for the forward reaction
		satisfies
		\begin{equation}\label{eq:event-rate-forward}
			\|\mathcal B_{\mathrm{fwd}}^K(F,G)\|_{\mathbb L^1}
			\le4\Lambda_+\|F\|_{\mathbb L^1}\|G\|_{\mathbb L^1}.
		\end{equation}
		The bilinear operator for the reverse reaction satisfies
		\begin{equation}\label{eq:event-rate-reverse}
			\|\mathcal B_{\mathrm{rev}}^K(F,G)\|_{\mathbb L^1}
			\le4\Lambda_-\|F\|_{\mathbb L^1}\|G\|_{\mathbb L^1},
		\end{equation}
		and the elastic bilinear operator satisfies
		\begin{equation}\label{eq:event-rate-elastic}
			\|\mathcal B_{\mathrm{el}}(F,G)\|_{\mathbb L^1}
			\le4\Lambda_{\rm el}\|F\|_{\mathbb L^1}\|G\|_{\mathbb L^1}.
		\end{equation}
		All three estimates hold for almost every $x\in\Omega$, with constants independent of $x$.
	\end{lemma}
	
	\begin{proof}
		By Proposition~\ref{prop:all-reactive-gains} and the strong elastic gain representation \eqref{eq:elastic-gain-density}, the three bilinear event operators are vectors of signed $L^1$ densities. For such a vector $H=(H_1,H_2,H_3,H_4)$,
		\[
		\|H\|_{\mathbb L^1}=\sup_{\max_i\|\phi_i\|_{L^\infty}\le1}	\bigg|\sum_{i=1}^4\int_{\R^3}H_i(v)\phi_i(v)\,dv\bigg|.
		\]
		
		We first consider the operator for the forward reaction. If $\max_i\|\phi_i\|_{L^\infty}\le1$, then $\big|\phi_3(v_3)+\phi_4(v_4)-\phi_1(v_1)-\phi_2(v_2)\big|\le4$. Using \eqref{eq:bilinear-forward-event} and the bound $K^+(x,g,\omega)\le \ol K^+(g,\omega)$ for almost every $x\in\Omega$, which follows from \eqref{eq:reactive-spatial-majorants}, and then
		integrating in $\omega$, we obtain
		\begin{equation*}
			\begin{split}
				&\quad \, \bigg|\sum_{i=1}^4\int_{\R^3}\mathcal B_{\mathrm{fwd},i}^K(F,G)(x,v)\phi_i(v)\,dv\bigg|\\
				&\le2\int_{\R^3}\int_{\R^3}\int_{\Sph}\ol K^+(v_1-v_2,\omega)\big[|F_1(v_1)||G_2(v_2)|		+|G_1(v_1)||F_2(v_2)|\big]\,d\omega\,dv_2\,dv_1\\
				&\le2\Lambda_+ \big(\|F_1\|_{L^1}\|G_2\|_{L^1}+\|G_1\|_{L^1}\|F_2\|_{L^1}\big)\le4\Lambda_+\|F\|_{\mathbb L^1}\|G\|_{\mathbb L^1}.
			\end{split}
		\end{equation*}
		Taking the supremum over the test vectors proves \eqref{eq:event-rate-forward}.	
		
		For the operator for the reverse reaction, the event bracket in \eqref{eq:bilinear-reverse-event} also has absolute value at most four. The same argument gives
		\begin{align*}
			\|\mathcal B_{\mathrm{rev}}^K(F,G)\|_{\mathbb L^1}\le2\Lambda_-		\big(\|F_3\|_{L^1}\|G_4\|_{L^1}+\|G_3\|_{L^1}\|F_4\|_{L^1}\big)\le4\Lambda_-\|F\|_{\mathbb L^1}\|G\|_{\mathbb L^1},
		\end{align*}
		which proves \eqref{eq:event-rate-reverse}.
		
		For the elastic operator, set
		\[
		\Lambda_{ij}^{\rm el}=\esssup_{g\in\R^3}\int_{\Sph}\ol B_{ij}(g,\sigma)\,d\sigma.
		\]
		By definition, $\frac12\sum_{i,j=1}^4\Lambda_{ij}^{\rm el}=\Lambda_{\rm el}$. The elastic event bracket has absolute value at most four. The factor $1/4$ in \eqref{eq:bilinear-elastic-event} therefore gives
		\begin{align*}
			\|\mathcal B_{\mathrm{el}}(F,G)\|_{\mathbb L^1}\le\sum_{i,j=1}^4\Lambda_{ij}^{\rm el}
			\big(\|F_i\|_{L^1}\|G_j\|_{L^1}+\|G_i\|_{L^1}\|F_j\|_{L^1}\big)\le4\Lambda_{\rm el}
			\|F\|_{\mathbb L^1}\|G\|_{\mathbb L^1}.
		\end{align*}
		This proves \eqref{eq:event-rate-elastic}.
	\end{proof}
	
	For nonnegative $F\in\mathfrak X(Z)$, let $\Q_{K,i}^+(F)$ denote the sum of all elastic and reactive gain terms in the $i$-th component, and let $\nu_{K,i}[F]$ denote the sum of all corresponding loss frequencies. Thus,
	\begin{equation}\label{eq:gain-loss}
		\Q_{K,i}(F)=\Q_{K,i}^+(F)-\nu_{K,i}[F]F_i, 	\quad 1\le i\le4.
	\end{equation}
	We write $\Q_K^+(F)=\big(\Q_{K,1}^+(F),\ldots,\Q_{K,4}^+(F)\big)$, $\nu_K[F]=\big(\nu_{K,1}[F],\ldots,\nu_{K,4}[F]\big)$, and use
	\[
	\|\nu_K[F]\|_{L^\infty(Z\times\R^3)}=\max_{1\le i\le4}\|\nu_{K,i}[F]\|_{L^\infty(Z\times\R^3)}.
	\] 
	For $q\in\R^3$, set $(\tau_qF)_i(z,v)=F_i(z,v-q)$ and $C_0=4(\Lambda_{\rm el}+\Lambda_++\Lambda_-)$. For $K\in\mathfrak K$, we have $C_0\le4(\Lambda_{\rm el}+M_0)$. We use this upper bound when choosing $\eps$ uniformly over $\mathfrak K$.
	
	\begin{theorem}[Estimates for the collision operator]
		\label{thm:collision-estimates}
		Assume $\Lambda_{\rm el}+\Lambda_++\Lambda_-<\infty$. Then, for
		every $F,G\in\mathfrak X(Z)$,
		\begin{equation}\label{eq:bilinear-bound}
			\|\B_K(F,G)\|_{\mathfrak X(Z)}
			\le C_0
			\|F\|_{\mathfrak X(Z)}
			\|G\|_{\mathfrak X(Z)},
		\end{equation}
		and
		\begin{equation}\label{eq:Q-Lipschitz}
			\|\Q_K(F)-\Q_K(G)\|_{\mathfrak X(Z)}
			\le C_0
			\big(\|F\|_{\mathfrak X(Z)}+\|G\|_{\mathfrak X(Z)}\big)
			\|F-G\|_{\mathfrak X(Z)}.
		\end{equation}
		
		If $F,G\ge0$, then
		\begin{align}
			\|\nu_K[F]\|_{L^\infty(Z\times\R^3)}
			&\le C_0\|F\|_{\mathfrak X(Z)},
			\label{eq:frequency-bound}\\
			\|\nu_K[F]-\nu_K[G]\|_{L^\infty(Z\times\R^3)}
			&\le C_0\|F-G\|_{\mathfrak X(Z)},
			\label{eq:frequency-Lipschitz}\\
			\|\Q_K^+(F)\|_{\mathfrak X(Z)}
			&\le C_0\|F\|_{\mathfrak X(Z)}^2,
			\label{eq:gain-bound}\\
			\|\Q_K^+(F)-\Q_K^+(G)\|_{\mathfrak X(Z)}
			&\le C_0
			\big(\|F\|_{\mathfrak X(Z)}+\|G\|_{\mathfrak X(Z)}\big)
			\|F-G\|_{\mathfrak X(Z)}.
			\label{eq:gain-Lipschitz}
		\end{align}
		If, in addition, $F_3=F_4=0$, then
		\begin{equation}\label{eq:channel-identity}
			\Q_{K,3}(F)=Q_{3,+}^K(F_1,F_2).
		\end{equation}
		Moreover, $\Q_K(\tau_qF)=\tau_q\Q_K(F)$ for every $q\in\R^3$.
	\end{theorem}
	
	\begin{proof}
		We first prove the bilinear estimate. For $1\le i\le4$, define
		\[
		F_i^\sharp(w)=\esssup_{z\in Z}|F_i(z,w)|, \quad G_i^\sharp(w)=\esssup_{z\in Z}|G_i(z,w)|.
		\]
		Then $|F_i(z,w)|\le F_i^\sharp(w)$ and $|G_i(z,w)|\le G_i^\sharp(w)$ for almost every $(z,w)$. Moreover,
		\[
		\|F^\sharp\|_{\mathbb L^1}=\|F\|_{\mathfrak X(Z)}, \quad \|G^\sharp\|_{\mathbb L^1}=\|G\|_{\mathfrak X(Z)}.
		\]
		
		Consider any gain or loss term occurring in $\B_K(F,G)$. In its explicit velocity representation, replace each input factor by the corresponding function $F_i^\sharp$ or $G_i^\sharp$, and use the spatial supremum bounds
		\[
		K^+(x,g,\omega)\le\ol K^+(g,\omega),\quad K^-(x,h,\theta)\le\ol K^-(h,\theta),
		\]
		together with $B_{ij}(x,g,\sigma)\le\ol B_{ij}(g,\sigma)$. After taking the essential supremum over $z$, the absolute value of the original term is bounded by the resulting nonnegative expression, which is
		independent of $z$. Summing these bounds over the four components, integrating in the output velocity, and applying the calculations in Lemma~\ref{lem:event-rate-estimate} with inputs $F^\sharp$ and $G^\sharp$ give
		\[
		\|\B_K(F,G)\|_{\mathfrak X(Z)}\le C_0 \|F^\sharp\|_{\mathbb L^1}\|G^\sharp\|_{\mathbb L^1}
		=C_0\|F\|_{\mathfrak X(Z)}\|G\|_{\mathfrak X(Z)}.
		\]
		This proves \eqref{eq:bilinear-bound}. Since $\Q_K(F)=\B_K(F,F)$ and $\B_K$ is symmetric and bilinear,
		\begin{equation}\label{eq:quadratic-difference}
			\Q_K(F)-\Q_K(G)=\B_K(F-G,F)+\B_K(G,F-G).
		\end{equation}
		Applying \eqref{eq:bilinear-bound} to the two terms on the right-hand side proves \eqref{eq:Q-Lipschitz}.
		
		We next estimate the loss frequencies. For example, the loss frequency in the first component due to the forward reaction is
		\[
		\nu_1^{\rm fwd}[F_2](z,v)=\int_{\R^3}\int_{\Sph}K^+(x,v-v_2,\omega)F_2(z,v_2)\,d\omega\,dv_2,
		\]
		where $x$ is the spatial component of $z$. By the definitions of $\ol K^+$ and $\Lambda_+$,
		\[
		\|\nu_1^{\rm fwd}[F_2]\|_{L^\infty(Z\times\R^3)}\le \Lambda_+\int_{\R^3}F_2^\sharp(v_2)\,dv_2
		\le \Lambda_+\|F\|_{\mathfrak X(Z)}.
		\]
		The remaining forward and reverse reactive loss frequencies in \eqref{eq:reactive-loss-frequencies} are estimated in the same way. The elastic loss frequencies satisfy the analogous estimate with $2\Lambda_{\rm el}$ after summing over the collision partner species. Since $C_0=4(\Lambda_{\rm el}+\Lambda_++\Lambda_-)$, adding all contributions proves \eqref{eq:frequency-bound}.
		
		Each loss frequency depends linearly on the distributions of the collision partners. Applying the same estimates to $F-G$ therefore gives \eqref{eq:frequency-Lipschitz}.
		
		Let $\B_K^+$ denote the symmetric bilinear operator formed by the gain terms of $\B_K$. Repeating the first part of the proof without the loss terms gives
		\[
		\|\B_K^+(F,G)\|_{\mathfrak X(Z)}\le C_0\|F\|_{\mathfrak X(Z)}\|G\|_{\mathfrak X(Z)}.
		\]
		For nonnegative $F$, one has $\Q_K^+(F)=\B_K^+(F,F)$, which proves \eqref{eq:gain-bound}. Similarly,
		\[
		\Q_K^+(F)-\Q_K^+(G)=\B_K^+(F-G,F)+\B_K^+(G,F-G),
		\]
		and \eqref{eq:gain-Lipschitz} follows from the bilinear gain estimate.
		
		Suppose now that $F_3=F_4=0$. Since elastic collisions preserve the species labels, every elastic term in the third component contains a factor $F_3$ and therefore vanishes. The reverse reactive contribution also vanishes, since it contains the product $F_3F_4$. The only remaining term in the third component is the forward gain produced by the reaction $1+2\to3+4$. Hence, $\Q_{K,3}(F)=Q_{3,+}^K(F_1,F_2)$, which proves \eqref{eq:channel-identity}.
		
		Finally, fix $q\in\R^3$. A common translation of all velocity variables leaves every relative velocity unchanged. The reaction maps commute with this translation by Lemma~\ref{lem:galilean}, and the elastic post-collisional maps have the same property. Translating all velocity integration variables
		by $q$ in the event formulas gives
		\[
		\Q_{K,i}(\tau_qF)(z,v)=\Q_{K,i}(F)(z,v-q)=(\tau_q\Q_K(F))_i(z,v).
		\]
		Thus, $\Q_K(\tau_qF)=\tau_q\Q_K(F)$.
	\end{proof}
	
	\begin{remark}\label{rem:finite-rate-class}
		The forward theorem assumes bounded angularly integrated rates. For $\Delta E=0$, this includes bounded angular-cutoff Maxwell laws. For $\Delta E>0$, the reverse kernel must also satisfy the condition near the threshold in Definition~\ref{def:kernel-class}; \eqref{eq:threshold-compatible-example} gives an example. The recovery proof uses only relative velocities in $\mathscr G^\sharp$, while the kernel may be nonzero outside this set. A cutoff law for hard potentials with rate growing like $|g|^\gamma$ requires estimates in a weighted space for the gain and loss terms. Such estimates are not proved here.
	\end{remark}

	\section{The forward problem for large velocities and the boundary flux}\label{sec:scaled-forward}
	
	\subsection{Scaling and well-posedness}\label{subsec:scaling-wp}
	
	Recall that $n(x)$ denotes the outward unit normal to $\p\Omega$ at $x\in\p\Omega$. Throughout this section, $\Omega$ is the bounded $C^3$ strictly convex domain fixed in Theorem~\ref{thm:main}.
	
	For $e\in\Sph$, let $P_{e^\perp}$ be the orthogonal projection onto $e^\perp$, and set $\Omega_e=P_{e^\perp}(\Omega)$. For $y\in\Omega_e$, define
	\[
	I(y,e)=\{\sigma\in\R:\ y+\sigma e\in\Omega\}.
	\]
	Since $\Omega$ is bounded, open, and convex, $I(y,e)$ is a nonempty bounded open interval. We write $I(y,e)=\big(\sigma_-(y,e),\sigma_+(y,e)\big)$, where $\sigma_-(y,e)=\inf I(y,e)$ and
	$\sigma_+(y,e)=\sup I(y,e)$. Equivalently,
	\begin{equation}\label{eq:chord}
		\Omega\cap(y+\R e)=\{y+\sigma e:\ \sigma_-(y,e)<\sigma<\sigma_+(y,e)\}.
	\end{equation}
	The corresponding entry and exit points are $x_\pm(y,e)=y+\sigma_\pm(y,e)e$, and the chord length is
	$L(y,e)=\sigma_+(y,e)-\sigma_-(y,e)$.
	
	\begin{lemma}[Chord geometry]\label{lem:chord-geometry}
		Fix $e\in\Sph$. For every $y\in\Omega_e$,
		\[
		n(x_-(y,e))\cdot e<0<n(x_+(y,e))\cdot e.
		\]
		If $U\Subset\Omega_e$, then the maps $y\mapsto\sigma_\pm(y,e)$, $y\mapsto x_\pm(y,e)$, and
		$y\mapsto L(y,e)$ are $C^3$ on $U$. Moreover, there exists $\kappa_U>0$ such that
		\begin{equation}\label{eq:local-nongrazing}
			-n(x_-(y,e))\cdot e\ge\kappa_U, \quad n(x_+(y,e))\cdot e\ge\kappa_U
		\end{equation}
		for every $y\in\ol U$. On either boundary sheet, the change of variables $y\mapsto x_\pm(y,e)$ satisfies
		\[
		dy=|n(x_\pm(y,e))\cdot e|\, dS_{x_\pm(y,e)}.
		\]
	\end{lemma}
	
	\begin{proof}
		Fix $y\in\Omega_e$ and abbreviate $x_\pm=x_\pm(y,e)$. We first show that the intersections are transversal. Let $H_+=\{z\in\R^3:\ (z-x_+)\cdot n(x_+)=0\}$ be the tangent hyperplane at $x_+$. Since $\Omega$ is convex, $H_+$ is a supporting hyperplane and therefore contains no interior point of $\Omega$. If $n(x_+)\cdot e=0$, then the line $x_++\R e$ is contained in $H_+$. This is impossible because $x_+-te\in\Omega$ for all sufficiently small $t>0$. Hence, $n(x_+)\cdot e\ne0$. The same argument gives $n(x_-)\cdot e\ne0$.
		
		Choose a $C^3$ defining function $\rho$ near $\p\Omega$ such that $\Omega=\{\rho<0\}$, $\p\Omega=\{\rho=0\}$, and $n=\nabla\rho/|\nabla\rho|$ on $\p\Omega$. Since $x_-+te\in\Omega$ and $x_+-te\in\Omega$ for small $t>0$, $\rho(x_-+te)=t\nabla\rho(x_-)\cdot e+o(t)<0$, and $\rho(x_+-te)=
		-t\nabla\rho(x_+)\cdot e+o(t)<0$. Using the nonvanishing established above, we obtain $\nabla\rho(x_-)\cdot e<0<\nabla\rho(x_+)\cdot e$, and hence, the asserted signs for the outward normal.
		
		We next fix $e$ and regard $y$ as a variable in the two-dimensional plane $e^\perp$. Set $\mathcal R(y,\sigma)=\rho(y+\sigma e)$. At $\sigma=\sigma_\pm(y,e)$,
		\[
		\partial_\sigma\mathcal R(y,\sigma_\pm(y,e))=\nabla\rho(x_\pm(y,e))\cdot e\ne0.
		\]
		The implicit function theorem therefore gives local $C^3$ representations of the two roots as functions of $y$. Convexity identifies these roots uniquely with the two endpoints of $I(y,e)$, so the local representations agree on overlaps. Consequently, $\sigma_\pm(\cdot,e)$ are $C^3$ on $U$, and the same is true of $x_\pm(\cdot,e)$ and $L(\cdot,e)$.
		
		The two functions $-n(x_-(y,e))\cdot e$ and $n(x_+(y,e))\cdot e$ are continuous and strictly positive on the compact set $\ol U$. Thus,
		\[
		\kappa_U=\min_{y\in\ol U}\min\big\{-n(x_-(y,e))\cdot e,\, n(x_+(y,e))\cdot e\big\}>0,
		\]
		which proves \eqref{eq:local-nongrazing}.
		
		It remains to prove the identity for the surface measure. Choose an oriented orthonormal basis $b_1,b_2$ of $e^\perp$ such that $b_1\times b_2=e$, write $y=y_1b_1+y_2b_2$, and parametrize the two sheets by
		$X_\pm(y)=y+\sigma_\pm(y,e)e$. Then $\partial_{y_j}X_\pm=b_j+\partial_{y_j}\sigma_\pm(y,e)e$, $j=1,2$. With $N_\pm=\partial_{y_1}X_\pm\times\partial_{y_2}X_\pm$, one has $N_\pm\cdot e=1$. At the outgoing sheet, $N_+=|N_+|n(X_+)$, whereas at the incoming sheet, $N_-=-|N_-|n(X_-)$. Since
		$dS_{X_\pm(y)}=|N_\pm(y)|\,dy$, it follows that 
		\[
		n(X_+(y))\cdot e\,dS_{X_+(y)}=dy=-n(X_-(y))\cdot e\,dS_{X_-(y)}.
		\]
		This is the asserted change of variables.
	\end{proof}
	
	For every $U\Subset\Omega_e$, the entry and exit points of the chords parametrized by $y\in\ol U$ are uniformly transverse to $\p\Omega$.
	
	\para{Physical and scaled initial-boundary value problems} Fix $e\in\Sph$, set $\eps=R^{-1}$, and write $s=Rt$ and $v=Re+w$. Define
	\begin{equation}\label{eq:scaled-H}
		H_i^\eps(s,x,w)=F_i^R(\eps s,x,Re+w),\quad R=\eps^{-1},
	\end{equation}
	and set
	\begin{equation}\label{eq:def_a_eps}
		a_\eps(w)=e+\eps w.
	\end{equation}
	The chain rule gives $(\p_s+a_\eps(w)\cdot\nabla_x)H_i^\eps=\eps\Q_{K,i}(F^R)(\eps s,x,Re+w)$. Translation invariance from Theorem~\ref{thm:collision-estimates} identifies the collision term with $\Q_{K,i}(H^\eps)(s,x,w)$. We obtain
	\begin{equation}\label{eq:scaled-system}
		\big[\p_s+a_\eps(w)\cdot\nabla_x\big]H_i^\eps=\eps\Q_{K,i}(H^\eps),\quad 1\le i\le4,
	\end{equation}
	on the fixed interval $0<s<T$. The same calculation holds for mild solutions by the corresponding change of characteristic variables.
	
	For $i=1,2$, the scaled incoming data are
	\begin{equation}\label{eq:physical-scaled-inflow}
		g_i^\eps(s,x,w)=G_{i,R}(\eps s,x,Re+w).
	\end{equation}
	Since $Re+w=Ra_\eps(w)$, the incoming condition is $n(x)\cdot a_\eps(w)<0$. Write $\Sigma_{-,\eps}=\{(s,x,w):0<s<T,\ x\in\p\Omega,\ n(x)\cdot a_\eps(w)<0\}$. The initial data are zero and the incoming vector is $g^\eps=(g_1^\eps,g_2^\eps,0,0)$.
	
	For $x\in\Omega$ and $a\in\R^3$, define
	\begin{equation}\label{eq:backward-exit-time}
		\tau_-(x,a)=\inf\{r>0:x-ra\notin\Omega\},
	\end{equation}
	where $\inf\varnothing=+\infty$. This time is finite for $a\ne0$, while $\tau_-(x,0)=+\infty$. Set
	\begin{equation}\label{eq:def_tau_minus_eps}
		\tau_{-,\eps}(x,w)=\tau_-(x,a_\eps(w)).
	\end{equation}
	For a scalar source, or componentwise for a vector source, let
	\begin{equation}\label{eq:G-eps}
		(\G_\eps f)(s,x,w)=\int_0^{\min\{s,\tau_{-,\eps}(x,w)\}}f(s-r,x-ra_\eps(w),w)\,dr.
	\end{equation}
	The integral stops at the initial plane or the incoming boundary, whichever is reached first.
	
	We use the Banach space of measurable functions with norm
	\begin{equation}\label{eq:X-T-norm}
		\|F\|_{\mathfrak X_T}=\sum_{i=1}^4\int_{\R^3}\esssup_{(s,x)\in(0,T)\times\Omega}|F_i(s,x,w)|\,dw.
	\end{equation}
	For scalar functions, the same notation omits the species sum. This is the mixed-norm space in \eqref{eq:XZ} with $Z_0=(0,T)$. The corresponding incoming norm is
	\begin{equation}\label{eq:incoming-norm}
		\|g^\eps\|_{\mathfrak X_{-,\eps}}=\sum_{i=1}^4\int_{\R^3}\esssup_{\substack{0<s<T,\ x\in\p\Omega\\n(x)\cdot a_\eps(w)<0}}|g_i^\eps(s,x,w)|\,dw.
	\end{equation}
	The interior and boundary essential suprema use $ds\,dx$ and $ds\,dS_x$, respectively. The supremum over an empty incoming set is zero. This occurs only at $a_\eps(w)=0$, a null set in velocity. Let $U^\eps$ be the free transport solution with these data.
	
	\begin{lemma}[Free transport and Duhamel bounds]\label{lem:free-Duhamel-bounds}
		For every $f\in\mathfrak X_T$, the Duhamel operator in \eqref{eq:G-eps} satisfies
		\begin{equation}\label{eq:G-eps-bound}
			\|\G_\eps f\|_{\mathfrak X_T}\le T\|f\|_{\mathfrak X_T}.
		\end{equation}
		Let $U^\eps$ be the free transport solution with zero initial data and incoming boundary value $g^\eps$. Then
		\begin{equation}\label{eq:free-incoming-bound}
			\|U^\eps\|_{\mathfrak X_T}\le \|g^\eps\|_{\mathfrak X_{-,\eps}}.
		\end{equation}
		If the third and fourth incoming data vanish, then $U_3^\eps=U_4^\eps=0$.
	\end{lemma}
	
	\begin{proof}
		For fixed $r$ and $w$, the map $(s,x)\mapsto(s-r,x-ra_\eps(w))$ is a translation. Fubini's theorem shows that \eqref{eq:G-eps} is well-defined for almost every $(s,x,w)$ and is independent of the representative of $f$. The interval has length at most $T$, so its velocity envelope is bounded by $T$ times that of $f$. This proves \eqref{eq:G-eps-bound}.
		
		Write $\tau=\tau_{-,\eps}(x,w)$. The free solution is
		\begin{equation}\label{eq:free-transport-representation}
			U_i^\eps(s,x,w)=
			\begin{cases}
				g_i^\eps(s-\tau,x-\tau a_\eps(w),w),&s>\tau,\\
				0,&s\le\tau.
			\end{cases}
		\end{equation}
		For $a_\eps(w)\ne0$, the change from an incoming point $b$ and a flight parameter $r$ to $x=b+ra_\eps(w)$ has Jacobian $|n(b)\cdot a_\eps(w)|$. It is positive on the incoming set. Together with the time translation, this shows that the pullback of boundary data is independent of its representative for almost every interior point. Taking velocity envelopes in \eqref{eq:free-transport-representation} proves \eqref{eq:free-incoming-bound}. The last assertion follows from the zero initial and incoming data of the product species.
	\end{proof}
	
	The scaled transport equation \eqref{eq:scaled-system}, together with the zero initial data and the prescribed incoming boundary values, forms the initial-boundary value problem
	\begin{equation}\label{eq:scaled-ibvp}
		\begin{cases}
			\big[\p_s+a_\eps(w)\cdot\nabla_x\big]H^\eps
			=\eps\Q_K(H^\eps)
			&\text{in }(0,T)\times\Omega\times\R^3,\\
			H^\eps(0,x,w)=0
			&\text{for }(x,w)\in\Omega\times\R^3,\\
			H^\eps=g^\eps
			&\text{on }\Sigma_{-,\eps}.
		\end{cases}
	\end{equation}
	The free term $U^\eps$ in \eqref{eq:free-transport-representation} satisfies the initial and incoming boundary conditions in \eqref{eq:scaled-ibvp}. We call $H^\eps\in\mathfrak X_T$ a mild solution of \eqref{eq:scaled-ibvp} if
	\begin{equation}\label{eq:mild-equation}
		H^\eps=U^\eps+\eps\G_\eps\Q_K(H^\eps)
	\end{equation}
	for almost every $(s,x,w)\in(0,T)\times\Omega\times\R^3$. Here $U^\eps$ is the free solution with zero initial data and incoming value $g^\eps$, and $\G_\eps$ is defined in \eqref{eq:G-eps}. Formula \eqref{eq:mild-equation} follows by integrating \eqref{eq:scaled-ibvp} along the backward characteristics up to the initial plane or the incoming boundary.
	
	\begin{theorem}[Well-posedness for incoming data of fixed size]\label{thm:forward-wp}
		Assume the hypotheses of Theorem~\ref{thm:collision-estimates}. Let $g^\eps=(g_1^\eps,g_2^\eps,0,0)\ge0$, and suppose that $\|g^\eps\|_{\mathfrak X_{-,\eps}}\le M_{\rm in}$ uniformly in $\eps$. If $\eps TC_0M_{\rm in}\le1/16$, then the scaled initial-boundary value problem \eqref{eq:scaled-ibvp} admits a unique mild solution $H^\eps\in\mathfrak X_T$. This solution is nonnegative, satisfies $\|H^\eps\|_{\mathfrak X_T}\le2M_{\rm in}$, and obeys
		\begin{equation}\label{eq:H-U}
			\|H^\eps-U^\eps\|_{\mathfrak X_T}\le4\eps TC_0M_{\rm in}^2.
		\end{equation}
		For each fixed finite bound $M_{\rm in}$ on the incoming norm, the theorem applies when $\eps$ is sufficiently small. We use this solution to define the boundary response.
	\end{theorem}
	
	\begin{proof}
		Let $\mathbb B_{2M_{\rm in}}=\{H\in\mathfrak X_T:\,\|H\|_{\mathfrak X_T}\le2M_{\rm in}\}$,
		and $\mathfrak a=\eps TC_0M_{\rm in}$. Using $\Q_K(H)=\B_K(H,H)$ from \eqref{eq:polarization}, define
		\begin{equation}\label{eq:signed-fixed-point-map}
			\mathfrak T_\eps H=U^\eps+\eps\G_\eps\B_K(H,H).
		\end{equation}
		By the definition above, a function $H\in\mathfrak X_T$ is a mild solution of \eqref{eq:scaled-ibvp} if and only if it is a fixed point of $\mathfrak T_\eps$.
		
		If $H\in\mathbb B_{2M_{\rm in}}$, then Theorem~\ref{thm:collision-estimates},
		Lemma~\ref{lem:free-Duhamel-bounds}, and \eqref{eq:free-incoming-bound} give
		\begin{equation}\label{eq:signed-map-ball}
			\begin{aligned}
				\|\mathfrak T_\eps H\|_{\mathfrak X_T}\le \|U^\eps\|_{\mathfrak X_T}+\eps T\|\B_K(H,H)\|_{\mathfrak X_T}\le M_{\rm in}+\eps TC_0\|H\|_{\mathfrak X_T}^2\le M_{\rm in}+4\eps TC_0M_{\rm in}^2=(1+4\mathfrak a)M_{\rm in}.
			\end{aligned}
		\end{equation}
		Since $\mathfrak a\le1/16$, the right-hand side of \eqref{eq:signed-map-ball} is at most $5M_{\rm in}/4$, and hence, $\mathfrak T_\eps$ maps $\mathbb B_{2M_{\rm in}}$ into itself.
		
		For $H,G\in\mathbb B_{2M_{\rm in}}$, the bilinearity and symmetry of $\B_K$ give
		$\B_K(H,H)-\B_K(G,G)=\B_K(H-G,H)+\B_K(G,H-G)$. The bilinear estimate yields
		\begin{equation}\label{eq:signed-map-contraction}
			\begin{aligned}
				\|\mathfrak T_\eps H-\mathfrak T_\eps G\|_{\mathfrak X_T}\le\eps TC_0
				\big(\|H\|_{\mathfrak X_T}+\|G\|_{\mathfrak X_T}\big)\|H-G\|_{\mathfrak X_T}\le4\eps TC_0M_{\rm in}\|H-G\|_{\mathfrak X_T}=4\mathfrak a\|H-G\|_{\mathfrak X_T}.
			\end{aligned}
		\end{equation}
		The assumption $\mathfrak a\le1/16$ makes the contraction factor at most $1/4$. The Banach fixed-point theorem gives a unique $H^\eps\in\mathbb B_{2M_{\rm in}}$ satisfying
		\begin{equation}\label{eq:signed-fixed-point-identity}
			H^\eps=\mathfrak T_\eps H^\eps
			=U^\eps+\eps\G_\eps\B_K(H^\eps,H^\eps).
		\end{equation}
		By \eqref{eq:polarization}, $\B_K(H^\eps,H^\eps)=\Q_K(H^\eps)$, so \eqref{eq:signed-fixed-point-identity} is precisely \eqref{eq:mild-equation}. This makes $H^\eps$ a mild solution of \eqref{eq:scaled-ibvp}.
		
		We next prove uniqueness in $\mathfrak X_T$. Let $H,G\in\mathfrak X_T$ be two mild solutions with the same initial and incoming data, and set $D=H-G$. For $0\le s_1<s_2\le T$, let $\|\cdot\|_{\mathfrak X(s_1,s_2)}$ denote the norm in \eqref{eq:X-T-norm} with the essential supremum restricted to $(s_1,s_2)\times\Omega$.
		
		Suppose that $D=0$ almost everywhere on $(0,s_1)\times\Omega\times\R^3$. Subtracting the two mild identities and using the symmetry and bilinearity of $\B_K$ gives $D=\eps\G_\eps\big[\B_K(D,H)+\B_K(G,D)\big]$. For $0<h\le T-s_1$, the part of each backward characteristic lying in $0<s-r\le s_1$ gives no contribution. Hence,
		\begin{equation}\label{eq:local-mild-uniqueness}
			\|D\|_{\mathfrak X(s_1,s_1+h)}
			\le\eps hC_0\big(\|H\|_{\mathfrak X_T}+\|G\|_{\mathfrak X_T}\big)\|D\|_{\mathfrak X(s_1,s_1+h)}.
		\end{equation}
		Choose $h>0$ so that the coefficient on the right-hand side of \eqref{eq:local-mild-uniqueness} is strictly smaller than one. Starting with $s_1=0$, we obtain $D=0$ on $(0,h)$. Repeating the same argument on consecutive intervals of length at most $h$ proves that $D=0$ on $(0,T)$. Thus, the mild solution in $\mathfrak X_T$ is unique.
		
		Subtracting $U^\eps$ from \eqref{eq:signed-fixed-point-identity}, and then applying 	\eqref{eq:G-eps-bound} and \eqref{eq:bilinear-bound}, we obtain
		\begin{align*}
			\|H^\eps-U^\eps\|_{\mathfrak X_T}=\eps\big\|\G_\eps\B_K(H^\eps,H^\eps)\big\|_{\mathfrak X_T}\le\eps T\|\B_K(H^\eps,H^\eps)\|_{\mathfrak X_T}\le\eps TC_0\|H^\eps\|_{\mathfrak X_T}^2
			\le4\eps TC_0M_{\rm in}^2.
		\end{align*}
		This proves \eqref{eq:H-U}.
		
		It remains to prove that $H^\eps$ is nonnegative. Since $g^\eps\ge0$ and the initial data vanish, the free transport representation \eqref{eq:free-transport-representation} gives $U^\eps\ge0$. For a 	nonnegative $H\in\mathbb B_{2M_{\rm in}}$, define $r_b(s,x,w)=\min\{s,\tau_{-,\eps}(x,w)\}$, where
		$\tau_{-,\eps}(x,w)$ is given by \eqref{eq:def_tau_minus_eps}. For $0\le r\le r_b(s,x,w)$, set
		\begin{equation}\label{eq:attenuation}
			E_{i,H}(r;s,x,w)=\exp\bigg[-\eps\int_0^r \nu_{K,i}[H]\big(s-\ell,x-\ell a_\eps(w),w\big)\,d\ell\bigg].
		\end{equation}
		Because $\nu_{K,i}[H]\ge0$, one has $0<E_{i,H}\le1$.
		
		Define the nonnegative map $\mathfrak P_\eps$ componentwise by
		\begin{equation}\label{eq:positive-map}
			\begin{split}
				(\mathfrak P_\eps H)_i(s,x,w)&=E_{i,H}\big(r_b(s,x,w);s,x,w\big)U_i^\eps(s,x,w)\\
				&\quad\, 
				+\eps\int_0^{r_b(s,x,w)}
				E_{i,H}(r;s,x,w)
				\Q_{K,i}^+(H)\big(s-r,x-ra_\eps(w),w\big)\,dr.
			\end{split}
		\end{equation}
		The first term is the attenuated initial or incoming contribution, while the second term is the accumulated gain along the backward characteristic.
		
		By \eqref{eq:gain-bound}, for every nonnegative $H\in\mathbb B_{2M_{\rm in}}$,
		\begin{equation}\label{eq:positive-map-ball-detailed}
			\begin{split}
				\|\mathfrak P_\eps H\|_{\mathfrak X_T}\le\|U^\eps\|_{\mathfrak X_T}+\eps T\|\Q_K^+(H)\|_{\mathfrak X_T}\le M_{\rm in}+\eps TC_0\|H\|_{\mathfrak X_T}^2\le(1+4\mathfrak a)M_{\rm in}.
			\end{split}
		\end{equation}
		Estimate \eqref{eq:positive-map-ball-detailed} shows that $\mathfrak P_\eps$ maps the nonnegative part of $\mathbb B_{2M_{\rm in}}$ into itself whenever $\mathfrak a\le1/4$, and in particular under the present assumption $\mathfrak a\le1/16$.
		
		To prove that $\mathfrak P_\eps$ is a contraction, introduce the characteristic parameter set
		\[
		\mathscr C_\eps=\big\{(r,s,x,w):\,0<s<T,\ x\in\Omega,\ w\in\R^3,\ 0\le r\le r_b(s,x,w)\big\}.
		\]
		For nonnegative $H,G\in\mathbb B_{2M_{\rm in}}$, the frequency estimate \eqref{eq:frequency-Lipschitz} and the inequality $|e^{-u}-e^{-v}|\le|u-v|$ for $u,v\ge0$ give
		\begin{equation}\label{eq:attenuation-difference}
			\max_{1\le i\le4}\|E_{i,H}-E_{i,G}\|_{L^\infty(\mathscr C_\eps)}\le\eps TC_0\|H-G\|_{\mathfrak X_T}.
		\end{equation}
		Indeed, the length of every characteristic segment in $\mathscr C_\eps$ is at most $T$.
		
		Subtracting the two formulas in \eqref{eq:positive-map}, we obtain, for every $1\le i\le4$,
		\begin{equation*}
			\begin{split}
				&\quad\, (\mathfrak P_\eps H-\mathfrak P_\eps G)_i(s,x,w)\\
				&=\big[E_{i,H}\big(r_b(s,x,w);s,x,w\big)-E_{i,G}\big(r_b(s,x,w);s,x,w\big)\big]U_i^\eps(s,x,w)\\
				&\quad\, +\eps\int_0^{r_b(s,x,w)}E_{i,H}(r;s,x,w)\big[\Q_{K,i}^+(H)-\Q_{K,i}^+(G)\big]
				\big(s-r,x-ra_\eps(w),w\big)\,dr\\
				&\quad\, +\eps\int_0^{r_b(s,x,w)}\big[E_{i,H}(r;s,x,w)-E_{i,G}(r;s,x,w)\big]\Q_{K,i}^+(G)		\big(s-r,x-ra_\eps(w),w\big)\,dr.
			\end{split}
		\end{equation*}
		
		We estimate the three terms separately. For the first term, \eqref{eq:attenuation-difference} and 
		\eqref{eq:free-incoming-bound} give
		\begin{align*}
			&\sum_{i=1}^4\int_{\R^3}\esssup_{(s,x)\in(0,T)\times\Omega}
			\Big|E_{i,H}\big(r_b(s,x,w);s,x,w\big)-E_{i,G}\big(r_b(s,x,w);s,x,w\big)\Big|
			|U_i^\eps(s,x,w)|\,dw\\
			&\le \eps TC_0\|H-G\|_{\mathfrak X_T}\|U^\eps\|_{\mathfrak X_T}\le\mathfrak a\|H-G\|_{\mathfrak X_T}.
		\end{align*}
		
		Since $0<E_{i,H}\le1$, the second term is bounded, by \eqref{eq:gain-Lipschitz}, as follows:
		\begin{align*}
			&\quad \, \eps\sum_{i=1}^4\int_{\R^3}\esssup_{(s,x)\in(0,T)\times\Omega}
			\int_0^{r_b(s,x,w)}E_{i,H}(r;s,x,w)\\
			&\hspace{35mm}\times\big|\Q_{K,i}^+(H)-\Q_{K,i}^+(G)\big|
			\big(s-r,x-ra_\eps(w),w\big)\,dr\,dw\\
			&\le \eps T\|\Q_K^+(H)-\Q_K^+(G)\|_{\mathfrak X_T}\\
			&\le \eps TC_0\big(\|H\|_{\mathfrak X_T}+\|G\|_{\mathfrak X_T}\big)\|H-G\|_{\mathfrak X_T}\\
			&\le4\mathfrak a\|H-G\|_{\mathfrak X_T}.
		\end{align*}
		
		For the third term, using \eqref{eq:attenuation-difference} and \eqref{eq:gain-bound}, we obtain
		\begin{equation}\label{eq:equation_for_Z}
			\begin{split}
				&\quad \, \eps\sum_{i=1}^4\int_{\R^3}\esssup_{(s,x)\in(0,T)\times\Omega}\int_0^{r_b(s,x,w)}
				\big|E_{i,H}(r;s,x,w)-E_{i,G}(r;s,x,w)\big|\\
				&\hspace{35mm}\times \Q_{K,i}^+(G)\big(s-r,x-ra_\eps(w),w\big)\,dr\,dw\\
				&\le \eps T	\max_{1\le i\le4}\|E_{i,H}-E_{i,G}\|_{L^\infty(\mathscr C_\eps)}
				\|\Q_K^+(G)\|_{\mathfrak X_T}\\
				&\le \eps T \big(\eps TC_0\|H-G\|_{\mathfrak X_T}\big)C_0\|G\|_{\mathfrak X_T}^2\\
				&\le4\mathfrak a^2\|H-G\|_{\mathfrak X_T},
			\end{split}
		\end{equation}
		where we use $\|G\|_{\mathfrak X_T}\le2M_{\rm in}$ and $\mathfrak a=\eps TC_0M_{\rm in}$ in the last inequality. Combining the preceding three estimates gives
		\begin{equation}\label{eq:positive-contraction}
			\|\mathfrak P_\eps H-\mathfrak P_\eps G\|_{\mathfrak X_T}\le \big(5\mathfrak a+4\mathfrak a^2\big)
			\|H-G\|_{\mathfrak X_T}.
		\end{equation}
		
		Let $\mathbb B_{2M_{\rm in}}^+=\{H\in\mathbb B_{2M_{\rm in}}:\,H\ge0\}$; this is a closed subset
		of $\mathfrak X_T$. If $\mathfrak a\le1/16$, then $5\mathfrak a+4\mathfrak a^2
		\le\frac{5}{16}+\frac{1}{64}=\frac{21}{64}<1$. Therefore, $\mathfrak P_\eps$ is a contraction from $\mathbb B_{2M_{\rm in}}^+$ into itself. By the Banach fixed-point theorem, it has a unique fixed point in $\mathbb B_{2M_{\rm in}}^+$.
		
		For fixed nonnegative $H\in\mathbb B_{2M_{\rm in}}^+$, formula \eqref{eq:positive-map} solves the linear transport equation with loss coefficient $\eps\nu_K[H]$, source $\eps\Q_K^+(H)$, zero initial data, and incoming value $g^\eps$, by variation of constants along characteristics. At a fixed point, both the loss frequency and the gain are evaluated at the solution itself. The identity \eqref{eq:gain-loss} then gives $H=U^\eps+\eps\G_\eps\Q_K(H)$, so $H$ is a mild solution of \eqref{eq:scaled-ibvp}. Uniqueness in $\mathfrak X_T$ gives $H=H^\eps\ge0$.
	\end{proof}
	
	For $R=\eps^{-1}$, set $F_i^R(t,x,v)=H_i^\eps(Rt,x,v-Re)$. Reversing the change of variables shows that $F^R$ solves the physical initial-boundary value problem on $0<t<T/R$ in the mild sense, with the prescribed incoming data. This defines the boundary response $\A_{K,R}$.

	We next compute the leading term of the third component. Since $\G_\eps$ in \eqref{eq:G-eps} acts componentwise, we write $\G_{\eps,i}$ for the same operator acting in the $i$-th equation.
	
	\begin{theorem}[Expansion of the product population]\label{thm:one-reaction}
		Under the assumptions of Theorem~\ref{thm:forward-wp}, there exists a scalar remainder
		$\mathcal E_3^\eps\in L^1_w(L^\infty_{s,x})$ such that
		\begin{equation}\label{eq:one-reaction}
			H_3^\eps=\eps\G_{\eps,3}Q_{3,+}^K(U_1^\eps,U_2^\eps)+\mathcal E_3^\eps.
		\end{equation}
		Moreover,
		\begin{equation}\label{eq:one-reaction-error}
			\|\mathcal E_3^\eps\|_{\mathfrak X_T}\le12\eps^2T^2C_0^2M_{\rm in}^3.
		\end{equation}
		The numerical coefficient $12$ is universal. In particular, the estimate is uniform for $K$ in every admissible family for which $C_0$ is uniformly bounded.
	\end{theorem}
	
	\begin{proof}
		Taking the third component of the mild identity \eqref{eq:mild-equation} gives
		\begin{equation}\label{eq:third-component-mild}
			H_3^\eps=U_3^\eps+\eps\G_{\eps,3}\Q_{K,3}(H^\eps).
		\end{equation}
		The third incoming and initial data vanish. Therefore, Lemma~\ref{lem:free-Duhamel-bounds} gives $U_3^\eps=0$, and \eqref{eq:third-component-mild} reduces to $H_3^\eps=\eps\G_{\eps,3}\Q_{K,3}(H^\eps)$. We add and subtract the collision source evaluated at the free solution:
		\begin{align*}
			H_3^\eps=\eps\G_{\eps,3}\Q_{K,3}(U^\eps)+\eps\G_{\eps,3}\big[\Q_{K,3}(H^\eps)-\Q_{K,3}(U^\eps)\big].
		\end{align*}
		By Lemma~\ref{lem:free-Duhamel-bounds}, one also has $U_4^\eps=0$. Hence, $U_3^\eps=U_4^\eps=0$, and the channel identity \eqref{eq:channel-identity} yields $\Q_{K,3}(U^\eps)=Q_{3,+}^K(U_1^\eps,U_2^\eps)$. We define
		\begin{equation}\label{eq:one-reaction-remainder-definition}
			\mathcal E_3^\eps:=\eps\G_{\eps,3}\big[\Q_{K,3}(H^\eps)-\Q_{K,3}(U^\eps)\big].
		\end{equation}
		Substitution gives the decomposition \eqref{eq:one-reaction}.
		
		It remains to estimate the remainder \eqref{eq:one-reaction-remainder-definition}. The componentwise
		version of the Duhamel estimate \eqref{eq:G-eps-bound} gives
		\begin{equation}\label{eq:one-reaction-error-first-bound}
			\begin{split}
				\|\mathcal E_3^\eps\|_{\mathfrak X_T}\le \eps T\|\Q_{K,3}(H^\eps)-\Q_{K,3}(U^\eps)\|_{\mathfrak X_T}\le \eps T\|\Q_K(H^\eps)-\Q_K(U^\eps)\|_{\mathfrak X_T},
			\end{split}
		\end{equation}
		where the second inequality holds because the norm of one component is bounded by the norm of the full vector of four components.
		
		Using $\Q_K(F)=\B_K(F,F)$ and the symmetry and bilinearity of $\B_K$, identity \eqref{eq:quadratic-difference} gives $\Q_K(H^\eps)-\Q_K(U^\eps)=\B_K(H^\eps-U^\eps,H^\eps)+\B_K(U^\eps,H^\eps-U^\eps)$. Moreover, the bilinear estimate \eqref{eq:bilinear-bound} implies
		\begin{equation}\label{eq:one-reaction-source-difference}
			\|\Q_K(H^\eps)-\Q_K(U^\eps)\|_{\mathfrak X_T}\le C_0
			\big(\|H^\eps\|_{\mathfrak X_T}+\|U^\eps\|_{\mathfrak X_T}\big)
			\|H^\eps-U^\eps\|_{\mathfrak X_T}.
		\end{equation}
		Combining \eqref{eq:one-reaction-error-first-bound} and \eqref{eq:one-reaction-source-difference}, we obtain
		\begin{equation}\label{eq:one-reaction-error_final}
			\|\mathcal E_3^\eps\|_{\mathfrak X_T}\le \eps TC_0 \big(\|H^\eps\|_{\mathfrak X_T}
			+\|U^\eps\|_{\mathfrak X_T}\big)\|H^\eps-U^\eps\|_{\mathfrak X_T}.
		\end{equation}
		By Theorem~\ref{thm:forward-wp} and \eqref{eq:free-incoming-bound}, $\|H^\eps\|_{\mathfrak X_T}\le2M_{\rm in}$, $\|U^\eps\|_{\mathfrak X_T}\le M_{\rm in}$, while \eqref{eq:H-U} gives $\|H^\eps-U^\eps\|_{\mathfrak X_T} \le4\eps TC_0M_{\rm in}^2$. Inserting these inequalities into
		\eqref{eq:one-reaction-error_final} gives \eqref{eq:one-reaction-error}, which proves the assertion.
	\end{proof}
	
	The first term in \eqref{eq:one-reaction} corresponds to free transport of the two incoming populations, one forward reaction, and free transport of the resulting particle of species $3$. Each remainder term contains $H^\eps-U^\eps$ and comes from at least one additional collision. The estimate $\|H^\eps-U^\eps\|_{\mathfrak X_T}=O(\eps)$ and the factor $\eps$ in the Duhamel integral give an $O(\eps^2)$ remainder.
	
	\subsection{Boundary coordinates and outgoing flux}\label{subsec:boundary-flux}
	
	Membership in $\mathfrak X_T$ does not by itself give a pointwise boundary trace. The mild equation defines an outgoing representative along almost every characteristic. We use coordinates along $a_\eps(w)$ to construct this representative and estimate its integrals against detector profiles.
	
	\para{Oblique characteristic coordinates} Fix $e\in\Sph$, $U\Subset\Omega_e$, and a bounded velocity set $W\Subset\R^3$. Recall that $a_\eps(w)=e+\eps w$. After decreasing $\eps_0>0$, one has $a_\eps(w)\cdot e\ge1/2$ for $0\le\eps\le\eps_0$ and $w\in W$. For such $(\eps,w)$, define the projection along $a_\eps(w)$ onto $e^\perp$ by
	\begin{equation}\label{eq:oblique-projection}
		\pi_{\eps,w}^e(x)=x-\frac{x\cdot e}{a_\eps(w)\cdot e}a_\eps(w).
	\end{equation}
	Then $\pi_{\eps,w}^e(x)\in e^\perp$, and the projection is constant on every affine line parallel to $a_\eps(w)$.
	
	\begin{lemma}[Oblique boundary coordinates]\label{lem:characteristic-maps}
		There is $\eps_0>0$ such that, for $0\le\eps\le\eps_0$, $y\in U$, and $w\in W$, there exist unique numbers $\sigma_-^\eps(y,w)<\sigma_+^\eps(y,w)$ satisfying
		\begin{equation}\label{eq:oblique-chord}
			\Omega\cap\big(y+\R a_\eps(w)\big)=\left\{y+\sigma a_\eps(w):\, \sigma_-^\eps(y,w)<\sigma<\sigma_+^\eps(y,w)\right\}.
		\end{equation}
		The line $y+\R a_\eps(w)$ meets $\p\Omega$ at exactly the two points
		\begin{equation}\label{eq:oblique-endpoints}
			x_\pm^\eps(y,w)=y+\sigma_\pm^\eps(y,w)a_\eps(w).
		\end{equation} The functions $\sigma_\pm^\eps$ and $x_\pm^\eps$ are $C^1$ in $(\eps,y,w)$, and
		\begin{equation}\label{eq:oblique-Oeps}
			|\sigma_\pm^\eps(y,w)-\sigma_\pm(y,e)|+|x_\pm^\eps(y,w)-x_\pm(y,e)|\le C\eps.
		\end{equation}
		Moreover,
		\begin{equation}\label{eq:oblique-signs}
			n(x_-^\eps(y,w))\cdot a_\eps(w)<0<n(x_+^\eps(y,w))\cdot a_\eps(w),
		\end{equation}
		and the boundary area formulas are
		\begin{align}
			-n(x_-^\eps(y,w))\cdot a_\eps(w)\,dS_{x_-^\eps}
			&=(a_\eps(w)\cdot e)\,dy,\label{eq:incoming-flux-jacobian}\\
			n(x_+^\eps(y,w))\cdot a_\eps(w)\,dS_{x_+^\eps}
			&=(a_\eps(w)\cdot e)\,dy.\label{eq:outgoing-flux-jacobian}
		\end{align}
		All estimates are uniform on $U\times W$.
	\end{lemma}
	
	\begin{proof}
		Choose an open neighborhood $\mathcal N$ of $\p\Omega$ and a $C^3$ function $\varrho:\mathcal N\to\R$ such that $\Omega\cap\mathcal N=\{\varrho<0\}$, $\p\Omega=\{\varrho=0\}$, and $\nabla\varrho\ne0$ on $\p\Omega$. Thus, $\varrho$ is a boundary defining function for $\Omega$. With this sign convention, $n(x)=\frac{\nabla\varrho(x)}{|\nabla\varrho(x)|}$, $x\in\p\Omega$.
		
		Near the two boundary intersections, set
		\begin{equation}\label{eq:boundary-root-equation}
			\mathcal F(\eps,y,w,\sigma):=\varrho\big(y+\sigma a_\eps(w)\big).
		\end{equation}
		A value of $\sigma$ determines a boundary intersection precisely when $\mathcal F(\eps,y,w,\sigma)=0$. At $\eps=0$, one has $a_\eps(w)|_{\eps=0}=a_0(w)=e$, independently of $w$, and hence, $\mathcal F\big(0,y,w,\sigma_\pm(y,e)\big)=0$. Moreover,
		\[
		\p_\sigma\mathcal F\big(0,y,w,\sigma_\pm(y,e)\big)=\nabla\varrho\big(x_\pm(y,e)\big)\cdot e
		=|\nabla\varrho\big(x_\pm(y,e)\big)|n\big(x_\pm(y,e)\big)\cdot e.
		\]
		By Lemma~\ref{lem:chord-geometry}, the factors $|n(x_\pm(y,e))\cdot e|$ are bounded from below on $\ol U$. The function $|\nabla\varrho|$ is also bounded from below on the compact boundary sheets $\{x_\pm(y,e):y\in\ol U\}$. Consequently, there exists $c_U>0$ such that
		\[
		\big|\p_\sigma\mathcal F\big(0,y,w,\sigma_\pm(y,e)\big)\big|\ge c_U
		\]
		for every $y\in\ol U$ and $w\in\ol W$.
		
		The implicit function theorem, followed by compactness of $\ol U\times\ol W$, gives $\eps_0>0$ and two uniquely determined roots $\sigma=\sigma_\pm^\eps(y,w)$ for $0\le\eps\le\eps_0$, $y\in U$, and $w\in W$. These roots are $C^1$ functions of $(\eps,y,w)$ and satisfy $\sigma_\pm^0(y,w)=\sigma_\pm(y,e)$. Since $\Omega$ is convex, its intersection with an affine line is an interval. Thus, the two roots are the only boundary intersections of the line $y+\R a_\eps(w)$.
		
		The $C^1$ dependence on $\eps$, together with compactness of $\ol U\times\ol W$, gives $|\sigma_\pm^\eps(y,w)-\sigma_\pm(y,e)|\le C\eps$. Furthermore,
		\[
		x_\pm^\eps(y,w)-x_\pm(y,e)=\big(\sigma_\pm^\eps(y,w)-\sigma_\pm(y,e)\big)a_\eps(w)
		+\sigma_\pm(y,e)\big(a_\eps(w)-e\big).
		\]
		Since $a_\eps(w)-e=\eps w$ and both $W$ and the functions $\sigma_\pm(\cdot,e)$ are bounded on the relevant compact sets, this proves \eqref{eq:oblique-Oeps}. Finally, the continuity of $n$, $a_\eps$, and $x_\pm^\eps$, together with $n(x_-(y,e))\cdot e<0<n(x_+(y,e))\cdot e$ gives \eqref{eq:oblique-signs}, after decreasing $\eps_0$ if necessary.
		
		We prove \eqref{eq:outgoing-flux-jacobian}; the incoming identity is the same with the opposite orientation. Choose an oriented orthonormal basis $b_1,b_2$ of $e^\perp$ such that $b_1\times b_2=e$, and write $y=y_1b_1+y_2b_2$. For fixed $(\eps,w)$, the outgoing sheet is parametrized by $X(y)=y+\sigma_+^\eps(y,w)a_\eps(w)$. Hence,
		\begin{equation*}
			\p_{y_j}X=b_j+(\p_{y_j}\sigma_+^\eps)a_\eps(w), \quad j=1,2.
		\end{equation*}
		Taking the scalar product of $\p_{y_1}X\times\p_{y_2}X$ with $a_\eps(w)$ removes the two terms containing a repeated factor $a_\eps(w)$ and gives
		\begin{equation*}
			a_\eps(w)\cdot(\p_{y_1}X\times\p_{y_2}X)=a_\eps(w)\cdot e.
		\end{equation*}
		The orientation is positive at the outgoing sheet by \eqref{eq:oblique-signs}. Since $n\,dS$ is the oriented surface element, this proves \eqref{eq:outgoing-flux-jacobian}. The same calculation at the incoming sheet gives \eqref{eq:incoming-flux-jacobian}.
	\end{proof}
	
	We write 
	\begin{equation}\label{eq:def_L_eps}
		L_\eps(y,w):=\sigma_+^\eps(y,w)-\sigma_-^\eps(y,w).
	\end{equation}
	By Lemma~\ref{lem:characteristic-maps}, one has $L_\eps(y,w)=L(y,e)+O(\eps)$ uniformly for $(y,w)\in U\times W$. For every open set $V\Subset\Omega_e$, define
	\begin{equation}\label{eq:characteristic-tube}
		\mathscr T_V:=\left\{y+\sigma e:\, y\in\ol V,\  \sigma_-(y,e)\le \sigma\le \sigma_+(y,e)\right\}.
	\end{equation}
	The parameter set $\{(y,\sigma):y\in\ol V,\ \sigma_-(y,e)\le\sigma\le\sigma_+(y,e)\}$ is compact because $\ol V$ is compact and $\sigma_\pm(\cdot,e)$ are continuous. Its image $\mathscr T_V$ is the union of the corresponding closed chords and is a compact subset of $\ol\Omega$. By \eqref{eq:local-nongrazing}, their endpoints meet $\p\Omega$ uniformly transversally.
	
	\para{Exact incoming pullback} Let $g_i\in C_c^\infty((0,T)\times U\times W)$ and extend it by zero to $\R\times e^\perp\times\R^3$. Define the scaled incoming data by
	\begin{equation}\label{eq:incoming-pullback}
		g_i^\eps(s,x_-^\eps(y,w),w)=g_i(s,y,w),
	\end{equation}
	and set them equal to zero on the rest of $\Sigma_{-,\eps}$. The map $(y,w)\mapsto(x_-^\eps(y,w),w)$ is injective. The flux Jacobian in Lemma~\ref{lem:characteristic-maps} is nonzero, so
	\begin{equation}\label{eq:incoming-uniform}
		\|g^\eps\|_{\mathfrak X_{-,\eps}}\le\sum_{i=1}^4\int_W\esssup_{(s,y)\in(0,T)\times U}|g_i(s,y,w)|\,dw.
	\end{equation}
	In physical variables, \eqref{eq:physical-scaled-inflow} gives
	\begin{equation}\label{eq:physical-incoming-pullback}
		G_{i,R}(\eps s,x_-^\eps(y,w),Re+w)=g_i(s,y,w),\quad R=\eps^{-1},\quad i=1,2.
	\end{equation}
	These points belong to $\Gamma_-$ because $n(x_-^\eps)\cdot(Re+w)=R n(x_-^\eps)\cdot a_\eps(w)<0$.
	
	For $x\in\Omega$ and $w\in W$, let $b_{-,\eps}(x,w)=x-\tau_{-,\eps}(x,w)a_\eps(w)$. Its incoming line coordinate is
	\begin{equation}\label{eq:Y-minus-eps}
		Y_{-,\eps}(x,w)=\pi_{\eps,w}^e(b_{-,\eps}(x,w))=\pi_{\eps,w}^e(x).
	\end{equation}
	The second equality follows because $\pi_{\eps,w}^e$ is constant along $a_\eps(w)$-lines. If $Y=Y_{-,\eps}(x,w)\in U$, then
	\[
	\tau_{-,\eps}(x,w)=\frac{x\cdot e}{a_\eps(w)\cdot e}-\sigma_-^\eps(Y,w),
	\quad b_{-,\eps}(x,w)=x_-^\eps(Y,w).
	\]
	Using the zero extension in time and in $y$, \eqref{eq:free-transport-representation} becomes
	\begin{equation}\label{eq:free-solution-formula}
		U_i^\eps(s,x,w)=g_i(s-\tau_{-,\eps}(x,w),Y_{-,\eps}(x,w),w).
	\end{equation}
	For $Y_{-,\eps}(x,w)\notin U$, both the prescribed incoming value and the right-hand side vanish. If $s\le\tau_{-,\eps}(x,w)$, they vanish by the initial condition and the zero extension in time.
	
	For each component, set 
	\begin{equation}\label{eq:incoming-C1-component-norm}
		\|g_i\|_{\mathfrak X_-^1}:=\int_{\R^3}\esssup_{(s,y)\in\R\times e^\perp}\big(|g_i|+|\p_s g_i|+|\nabla_y g_i|\big)(s,y,w)\,dw,
	\end{equation}
	and write $\|g\|_{\mathfrak X_-^1}=\sum_{i=1}^4\|g_i\|_{\mathfrak X_-^1}$.
	
	\begin{lemma}[Stability of the free characteristics]\label{lem:free-convergence}
		Assume $\|g\|_{\mathfrak X_-^1}<\infty$. For $0\le\eps\le\eps_0$ and $i=1,2$, one has
		\begin{equation}\label{eq:free-convergence}
			\|U_i^\eps-U_i^0\|_{L^1_w(L^\infty_{s,x}((0,T)\times\Omega))}
			\le C\eps\|g_i\|_{\mathfrak X_-^1}.
		\end{equation}
		The constant $C$ depends only on $e$, $W$, the geometry of $\Omega$, and a compact set $U_g\Subset U$ containing all $y\in U$ for which $g_1(s,y,w)$ or $g_2(s,y,w)$ is nonzero for some $(s,w)\in\R\times W$. In particular, $C$ is independent of the width of an $L^1$-normalized velocity packet whose $(s,y)$ profile is fixed.
	\end{lemma}
	
	\begin{proof}
		Choose open sets $V_0,V_1$ such that $\ol U_g\subset V_0\Subset V_1\Subset U$. For bounded $x\in\Omega$ and $w\in W$, formula \eqref{eq:oblique-projection} gives $|Y_{-,\eps}(x,w)-P_{e^\perp}x|\le C\eps$. If either free field is nonzero, one of these two line coordinates belongs to $\ol U_g$; both belong to $V_0$ for sufficiently small $\eps$.
		
		On this set the incoming root formula following \eqref{eq:Y-minus-eps} gives
		\[
		\tau_{-,\eps}(x,w)=\frac{x\cdot e}{a_\eps(w)\cdot e}-\sigma_-^\eps(Y_{-,\eps}(x,w),w),
		\quad \tau_{-,0}(x,w)=x\cdot e-\sigma_-(P_{e^\perp}x,e).
		\]
		The uniform root estimate \eqref{eq:oblique-Oeps}, the bounded derivatives of $\sigma_-(\cdot,e)$ on $V_1$, and the projection estimate imply
		\begin{equation}\label{eq:footpoint-convergence}
			|\tau_{-,\eps}(x,w)-\tau_{-,0}(x,w)|+|Y_{-,\eps}(x,w)-Y_{-,0}(x,w)|\le C\eps.
		\end{equation}
		Apply the mean value theorem to the smooth zero extension of $g_i$ in its $(s,y)$ variables. If at least one free field is nonzero, \eqref{eq:free-solution-formula} and \eqref{eq:footpoint-convergence} give
		\[
		|U_i^\eps-U_i^0|(s,x,w)\le C\eps\sup_{(t,y)\in\R\times e^\perp}(|\p_tg_i|+|\nabla_yg_i|)(t,y,w).
		\]
		The same inequality is trivial where both fields vanish. Taking the essential supremum in $(s,x)$ and integrating in $w$ proves \eqref{eq:free-convergence}. Only derivatives of the fixed temporal and spatial profiles occur for the packet family in Definition~\ref{def:beam-data}; the velocity packets have unit $L^1$ mass.
	\end{proof}
	
	\para{Outgoing flux as a weak boundary measurement} 
	Let $W_3\Subset\R^3$ be a bounded open set. Applying Lemma~\ref{lem:characteristic-maps} with $W_3$ in place of $W$, and decreasing $\eps_0$ if necessary, the quantities $x_\pm^\eps(y,w)$ and $L_\eps(y,w)$ are defined uniformly for $(y,w)\in U\times W_3$.
	
	For a scalar source $f$ with $\|f\|_{\mathfrak X_T}<\infty$, define its outgoing Duhamel representative by
	\begin{align}
		(\mathcal T_{+,\eps}f)(s,y,w)
		&=\int_0^{\min\{s,L_\eps(y,w)\}}
		f\big(s-r,x_+^\eps(y,w)-ra_\eps(w),w\big)\,dr.
		\label{eq:outgoing-characteristic-trace}
	\end{align}
	This quantity is defined for almost every $(s,y,w)$. It is the outgoing value of the solution with zero initial and incoming data and source $f$.
	
	\begin{lemma}[Outgoing Duhamel estimate]\label{lem:trace-bound}
		Let $f$ be a scalar source with $\|f\|_{\mathfrak X_T}<\infty$. Then
		\begin{equation}\label{eq:trace-bound}
			\int_{W_3}\esssup_{(s,y)\in(0,T)\times U}|(\mathcal T_{+,\eps}f)(s,y,w)|\,dw\le T\|f\|_{\mathfrak X_T}.
		\end{equation}
		The estimate is uniform for $0\le\eps\le\eps_0$.
	\end{lemma}
	
	\begin{proof}
		Fix $w\in W_3$. On $0<r<\min\{s,L_\eps(y,w)\}$, consider
		\[
		(s,y,r)\mapsto(t,x)=(s-r,x_+^\eps(y,w)-ra_\eps(w)).
		\]
		Its absolute Jacobian is $a_\eps(w)\cdot e\ge1/2$. Indeed, the time coordinate has derivative one in $s$, and the spatial Jacobian is the outgoing flux Jacobian \eqref{eq:outgoing-flux-jacobian}. Fubini's theorem now shows that changing $f$ on a null set changes \eqref{eq:outgoing-characteristic-trace} only on a null set in $(s,y,w)$. It also gives, for almost every $(s,y,w)$,
		\[
		|(\mathcal T_{+,\eps}f)(s,y,w)|\le T\esssup_{(t,x)\in(0,T)\times\Omega}|f(t,x,w)|.
		\]
		Taking the essential supremum in $(s,y)$ and integrating in $w$ proves \eqref{eq:trace-bound}.
	\end{proof}
	
	Since the third component has zero initial and incoming data, define its outgoing representative on the coordinate patch by
	\begin{equation}\label{eq:third-outgoing-trace-definition}
		(\gamma_{+,\eps}H_3^\eps)(s,y,w)
		:=\eps\big(\mathcal T_{+,\eps}\Q_{K,3}(H^\eps)\big)(s,y,w).
	\end{equation}
	By Lemma~\ref{lem:trace-bound}, the definition does not depend on the representative of $\Q_{K,3}(H^\eps)$. Along almost every characteristic, the mild equation gives an absolutely continuous solution whose outgoing endpoint value is \eqref{eq:third-outgoing-trace-definition}. Applying the variation of constants formula from the proof of Theorem~\ref{thm:forward-wp} up to this endpoint shows that the trace is nonnegative. For a classical solution, it agrees with the usual outgoing trace.
	
	For $(x,v)\in\Gamma_+$, let $\tau_{\Omega,-}(x,v)$ be the largest number $\tau>0$ such that $x-\theta v\in\Omega$ for every $0<\theta<\tau$. Boundedness and strict convexity give $0<\tau_{\Omega,-}(x,v)<\infty$. For fixed $v$, the map from an outgoing boundary point and a flight time to an interior point has Jacobian $|n(x)\cdot v|>0$. Applying the Fubini argument of Lemma~\ref{lem:trace-bound} in local boundary charts gives the following definition for almost every $(t,x,v)\in(0,T/R)\times\Gamma_+$, independently of the representative of the collision source:
	\begin{equation}\label{eq:global-physical-outgoing-trace}
		\gamma_+F_3^R(t,x,v)
		:=\int_0^{\min\{t,\tau_{\Omega,-}(x,v)\}}\Q_{K,3}(F^R)(t-\theta,x-\theta v,v)\,d\theta.
	\end{equation}
	This is the outgoing characteristic representative used in the physical flux $J_{3,R}$ from \eqref{eq:intro-product-flux}. On the boundary patch parametrized by $y\in U$ and $w\in W_3$, one has $\tau_{\Omega,-}(x_+^\eps(y,w),Re+w)=L_\eps(y,w)/R$, and therefore,
	\begin{equation}\label{eq:physical-outgoing-trace-definition}
		\begin{split}
			&\gamma_+F_3^R\big(t,x_+^\eps(y,w),Re+w\big)\\
			&\quad=\int_0^{\min\{t,L_\eps(y,w)/R\}}\Q_{K,3}(F^R)\big(t-\theta,x_+^\eps(y,w)-\theta(Re+w),Re+w\big)\,d\theta.
		\end{split}
	\end{equation}
	
	For $y\in U$ and $w\in W_3$, define
	\begin{equation}\label{eq:projected-flux}
		\mathscr J_{3,\eps}^e(s,y,w)=\frac{a_\eps(w)\cdot e}{\eps}(\gamma_{+,\eps}H_3^\eps)(s,y,w).
	\end{equation}
	The factor $\eps^{-1}$ accounts for the $O(\eps)$ size of the first reaction term, and $a_\eps(w)\cdot e$ comes from the flux Jacobian in \eqref{eq:outgoing-flux-jacobian}. For a detector profile $\Phi=\Phi(s,y,w)$ supported in $(0,T)\times U\times W_3$, set
	\begin{equation}\label{eq:detector-pairing}
		\langle\mathscr A_{K,\eps}^e[g^\eps],\Phi\rangle=\int_U\int_0^T\int_{W_3}
		\mathscr J_{3,\eps}^e(s,y,w)\Phi(s,y,w)\,dw\,ds\,dy.
	\end{equation}
	For $w\in W_3$, set $\Gamma_{+,\eps}^e(U,w)=\{x_+^\eps(y,w):y\in U\}$.
	
	\begin{proposition}[Physical representation of the detector]\label{prop:exact-detector-identity}
		For every detector $\Phi$ as above,
		\begin{equation}\label{eq:physical-detector-identity}
			\begin{split}
				\langle\mathscr A_{K,\eps}^e[g^\eps],\Phi\rangle =R\int_{W_3}\int_0^{T/R}\int_{\Gamma_{+,\eps}^e(U,w)}		J_{3,R}(t,x,Re+w)\Phi\big(Rt,\pi_{\eps,w}^e(x),w\big)\,dS_x\,dt\,dw.
			\end{split}
		\end{equation}
		Identity \eqref{eq:physical-detector-identity} shows that the functional on the left is determined by the measured outgoing particle flux.
	\end{proposition}
	
	\begin{proof}
		Fix $y\in U$ and $w\in W_3$, and set $s=Rt$. In \eqref{eq:physical-outgoing-trace-definition}, make the change of variables $r=R\theta$. Since $R=\eps^{-1}$ and $Re+w=Ra_\eps(w)$, one has $d\theta=\eps\,dr$, $t-\theta=\eps(s-r)$, and $x_+^\eps(y,w)-\theta(Re+w)=x_+^\eps(y,w)-ra_\eps(w)$. The identity for the collision operator used in the derivation of \eqref{eq:scaled-system} gives
		\[
		\Q_{K,3}(F^R)\big(\eps(s-r),z,Re+w\big)=\Q_{K,3}(H^\eps)(s-r,z,w)
		\]
		for almost every $(s-r,z,w)\in(0,T)\times\Omega\times\R^3$. The change of variables in Lemma~\ref{lem:trace-bound} gives, for almost every $(s,y,w)\in(0,T)\times U\times W_3$,
		\begin{equation}\label{eq:physical-scaled-trace}
			\begin{split}
				\gamma_+F_3^R\big(t,x_+^\eps(y,w),Re+w\big)&=\eps\int_0^{\min\{s,L_\eps(y,w)\}}
				\Q_{K,3}(H^\eps)\big(s-r,x_+^\eps(y,w)-ra_\eps(w),w\big)\,dr\\
				&=(\gamma_{+,\eps}H_3^\eps)(s,y,w).
			\end{split}
		\end{equation}
		
		We now evaluate the right-hand side of \eqref{eq:physical-detector-identity}. On the outgoing sheet,
		\eqref{eq:oblique-signs} gives $\big(n(x_+^\eps(y,w))\cdot(Re+w)\big)_+=Rn(x_+^\eps(y,w))\cdot a_\eps(w)$. Using the parametrization $x=x_+^\eps(y,w)$, the trace identity \eqref{eq:physical-scaled-trace}, the change of variables $s=Rt$, and the flux Jacobian \eqref{eq:outgoing-flux-jacobian}, we obtain
		\begin{align*}
			&\quad\, R\int_{W_3}\int_0^{T/R}\int_{\Gamma_{+,\eps}^e(U,w)}J_{3,R}(t,x,Re+w)
			\Phi\big(Rt,\pi_{\eps,w}^e(x),w\big)\,dS_x\,dt\,dw\\
			&=R\int_{W_3}\int_0^T\int_U(a_\eps(w)\cdot e)(\gamma_{+,\eps}H_3^\eps)(s,y,w)		\Phi(s,y,w)\,dy\,ds\,dw\\
			&=\int_{W_3}\int_0^T\int_U\frac{a_\eps(w)\cdot e}{\eps}(\gamma_{+,\eps}H_3^\eps)(s,y,w)
			\Phi(s,y,w)\,dy\,ds\,dw\\
			&=\langle \mathscr A_{K,\eps}^e[g^\eps],\Phi\rangle,
		\end{align*}
		where $R=\eps^{-1}$ was used in the penultimate equality. This proves 	\eqref{eq:physical-detector-identity}.
	\end{proof}
	
	The detector functional is the normalized outgoing particle count tested against a prescribed function of time, boundary position, and velocity. The time dependence used in the inverse experiment is specified in Definition~\ref{def:measured-functional}. A signed smooth detector profile is the difference of two nonnegative smooth profiles, so its value can be obtained by subtracting two nonnegative readings. Let $\mathfrak D(U,W_3)=L^\infty_w(W_3;L^1_{s,y}((0,T)\times U))$, with norm
	\begin{equation}\label{eq:detector-norm}
		\|\Phi\|_{\mathfrak D(U,W_3)}
		=\esssup_{w\in W_3}\int_U\int_0^T
		|\Phi(s,y,w)|\,ds\,dy.
	\end{equation}
	
	\begin{lemma}[Regularity of the free reaction source]\label{lem:source-regularity}
		Let $K\in\mathfrak K$, and let $W_1,W_2\Subset W$ satisfy $W_1-W_2\Subset\mathscr G^\sharp$. Let $U_1^0,U_2^0$ be the free transport solutions generated by controls with finite $\mathfrak X_-^1$ norms and velocity supports contained in $W_1$ and $W_2$, respectively. Then the velocity support of $S^0$ is contained in the compact set
		\begin{equation}\label{eq:product-output-support}
			\mathcal V_3(W_1,W_2)=\left\{\frac{m_1w+m_2w_*}{M}+\frac{m_4}{M}\rho_+(|w-w_*|)\omega:\,
			w\in\ol W_1,\ w_*\in\ol W_2,\ \omega\in\Sph \right\},
		\end{equation}
		where 
		\begin{equation}\label{eq:S0-definition}
			S^0:=Q_{3,+}^K(U_1^0,U_2^0).
		\end{equation}
		Let $V$ be any open set satisfying $\ol U\subset V\Subset\Omega_e$, and let $W_3\Subset\R^3$ be a bounded open set containing $\mathcal V_3(W_1,W_2)$. Then
		\begin{equation}\label{eq:source-C1-bound}
			\begin{split}
				\int_{W_3}\esssup_{\substack{0<s<T\\ x\in\mathscr T_V\cap\Omega}}\big(|S^0|+|\p_sS^0|+|\nabla_xS^0|\big)(s,x,w)\,dw\le
				C\|g_1\|_{\mathfrak X_-^1}\|g_2\|_{\mathfrak X_-^1}.
			\end{split}
		\end{equation}
		The constant depends on $V$, the velocity sets, the geometry of $\Omega$, and the local $C^1$ bound for $K$ on $\ol\Omega\times\ol{\mathscr G^\sharp}\times\Sph$. It is independent of the width of the normalized packets in Definition~\ref{def:beam-data}.
	\end{lemma}
	
	\begin{proof}
		We first estimate the free fields. Let $x\in\mathscr T_V\cap\Omega$, and write 	$y=P_{e^\perp}x$ and $\sigma=x\cdot e$. Since $e\in\Sph$, one has $x=y+\sigma e$. At $\eps=0$, the transport direction is $a_0(w)=e$, so the backward characteristic is
		$x-re=y+(\sigma-r)e$. By \eqref{eq:chord}, $\Omega\cap(y+\R e)=\{y+\lambda e:\,\sigma_-(y,e)<\lambda<\sigma_+(y,e)\}$. Since $x\in\Omega$, one has 	$\sigma_-(y,e)<\sigma<\sigma_+(y,e)$. Along the backward characteristic, the line parameter $\sigma-r$ decreases as $r$ increases, and the first boundary point is reached when $\sigma-r=\sigma_-(y,e)$. Hence, $\tau_{-,0}(x,w)=\sigma-\sigma_-(y,e)$ and
		\[
		b_{-,0}(x,w)=x-\tau_{-,0}(x,w)e=y+\sigma_-(y,e)e=x_-(y,e).
		\]
		Moreover, \eqref{eq:oblique-projection} gives $\pi_{0,w}^e=P_{e^\perp}$, because $a_0(w)=e$ and $|e|=1$. We have 
		\[
		Y_{-,0}(x,w)=\pi_{0,w}^e\big(b_{-,0}(x,w)\big)=P_{e^\perp}\big(y+\sigma_-(y,e)e\big)=y.
		\]
		Thus, the free characteristic formula becomes $U_i^0(s,x,w)=g_i\big(s-\sigma+\sigma_-(y,e),y,w\big)$. The functions $\sigma_-(\cdot,e)$ and $P_{e^\perp}$ have bounded first derivatives on the relevant compact set. Applying the chain rule and using the smooth zero extension of the controls gives
		\begin{equation}\label{eq:free-C1-bound}
			\begin{split}
				&\int_{W_i}\esssup_{\substack{0<s<T\\x\in\mathscr T_V\cap\Omega}}
				\big(|U_i^0|+|\p_sU_i^0|+|\nabla_xU_i^0|\big)(s,x,w)\,dw
				\le C\|g_i\|_{\mathfrak X_-^1},\quad i=1,2.
			\end{split}
		\end{equation}
		
		We next consider the gain. By the strong representation \eqref{eq:third-gain-strong},
		\[
		\begin{split}
			S^0(s,x,v_3)=\int_{\R^3}\int_{\Sph}K(x,g,\omega)U_1^0\Big(s,x,	v_3-\lambda_3(g)\omega+\frac{m_2}{M}g\Big)U_2^0\Big(s,x,	v_3-\lambda_3(g)\omega-\frac{m_1}{M}g\Big)\,d\omega\,dg.
		\end{split}
		\]
		If the two input factors are nonzero, their velocity arguments belong to $W_1$ and $W_2$. Thus, the relative velocity belongs to $W_1-W_2$, and the output velocity belongs to $\mathcal V_3(W_1,W_2)$. This proves the support inclusion. Compactness follows from continuity of the reaction map on the compact set $\overline{W_1}\times\overline{W_2}\times\Sph$.
		
		A time derivative falls on one of the two free fields:
		\[
		\p_sS^0=Q_{3,+}^{K}(\p_sU_1^0,U_2^0)+Q_{3,+}^{K}(U_1^0,\p_sU_2^0).
		\]
		To justify differentiation under the integral, bound each input and its $(s,x)$ derivatives by the velocity envelopes in \eqref{eq:free-C1-bound}, and bound $K$ and $\nabla_xK$ by their local majorants. The resulting nonnegative gain expressions are integrable in the output velocity by the estimate for the event rate. For almost every output velocity, this gives an integrable majorant for the integrands and their first derivatives, uniformly on the spatial tube. Differentiation follows by dominated convergence. In particular, for each $x_j$,
		\[
		\begin{split}
			\p_{x_j}S^0=Q_{3,+}^{\p_{x_j}K}(U_1^0,U_2^0)+Q_{3,+}^{K}(\p_{x_j}U_1^0,U_2^0)+
			Q_{3,+}^{K}(U_1^0,\p_{x_j}U_2^0),
		\end{split}
		\]
		where $Q_{3,+}^{\p_{x_j}K}$ denotes the same bilinear gain with $K$ replaced by $\p_{x_j}K$.
		
		Since $W_1-W_2\Subset\mathscr G^\sharp$, the angular integrals of $K$ and $\nabla_xK$ are uniformly bounded for all relative velocities in these formulas. Take the essential supremum in $(s,x)$, bound the inputs by their velocity envelopes, and apply the scalar gain estimate for the forward reaction to obtain
		\[
		\begin{split}
			&\int_{W_3}\esssup_{\substack{0<s<T\\x\in\mathscr T_V\cap\Omega}}
			\big(|S^0|+|\p_sS^0|+|\nabla_xS^0|\big)(s,x,w)\,dw\\
			&\quad\le C\prod_{i=1}^2\bigg[\int_{W_i}\esssup_{\substack{0<s<T\\x\in\mathscr T_V\cap\Omega}}
			\big(|U_i^0|+|\p_sU_i^0|+|\nabla_xU_i^0|\big)(s,x,w)\,dw \bigg].
		\end{split}
		\]
		Estimate \eqref{eq:source-C1-bound} now follows from \eqref{eq:free-C1-bound}.
		
		For the controls in Definition~\ref{def:beam-data}, all derivatives in $(s,y)$ act only on the fixed temporal and spatial profiles, whereas the velocity packets have unit $L^1$ mass. Hence, the
		constant is independent of the packet width.
	\end{proof}
	
	\begin{lemma}[Convergence of the leading reaction term]\label{lem:coefficient-convergence}
		Assume the hypotheses of Lemmas~\ref{lem:free-convergence} and \ref{lem:source-regularity}. Then
		\begin{equation}\label{eq:coefficient-convergence}
			\begin{split}
				\int_{W_3}\esssup_{(s,y)\in(0,T)\times U}\big|\mathcal T_{+,\eps}Q_{3,+}^K(U_1^\eps,U_2^\eps)
				-\mathcal T_{+,0}Q_{3,+}^K(U_1^0,U_2^0)\big|(s,y,w)\,dw\le C\eps.
			\end{split}
		\end{equation}
	\end{lemma}
	
	\begin{proof}
		Set $S^\eps=Q_{3,+}^K(U_1^\eps,U_2^\eps)$ and $S^0=Q_{3,+}^K(U_1^0,U_2^0)$. Split
		\begin{equation}\label{eq:coefficient-splitting}
			\mathcal T_{+,\eps}S^\eps-\mathcal T_{+,0}S^0=\mathcal T_{+,\eps}(S^\eps-S^0)
			+(\mathcal T_{+,\eps}-\mathcal T_{+,0})S^0.
		\end{equation}
		First, the mixed-norm bilinear gain estimate and Lemma~\ref{lem:free-convergence} give
		\begin{equation}\label{eq:source-eps-minus-zero}
			\begin{split}
				\|S^\eps-S^0\|_{\mathfrak X_T}\le C\|U_1^\eps-U_1^0\|_{\mathfrak X_T}
				\|U_2^\eps\|_{\mathfrak X_T}+C\|U_1^0\|_{\mathfrak X_T}\|U_2^\eps-U_2^0\|_{\mathfrak X_T}\le C\eps.
			\end{split}
		\end{equation}
		Lemma~\ref{lem:trace-bound} now bounds the first term on the right-hand side of \eqref{eq:coefficient-splitting} by $C\eps$ in $L^1_wL^\infty_{s,y}$.
		
		For the second term in \eqref{eq:coefficient-splitting}, set $\ell_\eps(s,y,w)=\min\{s,L_\eps(y,w)\}$, $\ell_0(s,y)=\min\{s,L(y,e)\}$, $p_\eps(r)=x_+^\eps(y,w)-ra_\eps(w)$, and $p_0(r)=x_+(y,e)-re$. Then
		\[
		(\mathcal T_{+,\eps}S^0)(s,y,w)=\int_0^{\ell_\eps}S^0(s-r,p_\eps(r),w)\,dr,
		\quad
		(\mathcal T_{+,0}S^0)(s,y,w)=\int_0^{\ell_0}S^0(s-r,p_0(r),w)\,dr.
		\]
		Lemma~\ref{lem:characteristic-maps} gives $|\ell_\eps-\ell_0|+|x_+^\eps(y,w)-x_+(y,e)|+|a_\eps(w)-e|\le C\eps$ uniformly on $(0,T)\times U\times W_3$. Let $m=\min\{\ell_\eps,\ell_0\}$. Then
		\begin{align*}
			\big|(\mathcal T_{+,\eps}-\mathcal T_{+,0})S^0(s,y,w)\big|
			&\le\int_0^m\big|S^0(s-r,p_\eps(r),w)-S^0(s-r,p_0(r),w)\big|\,dr\\
			&\quad\, +\int_m^{\ell_\eps}|S^0(s-r,p_\eps(r),w)|\,dr
			+\int_m^{\ell_0}|S^0(s-r,p_0(r),w)|\,dr.
		\end{align*}
		At most one of the last two integrals is nonzero, and its interval has length at most $C\eps$. Moreover, $|p_\eps(r)-p_0(r)|\le C\eps$ for $0\le r\le m$.
		
		Choose an open set $V$ such that $\ol U\subset V\Subset\Omega_e$. For sufficiently small $\eps$, the points $p_\eps(r)$, $p_0(r)$, and the segment joining them belong to $\mathscr T_V\cap\ol\Omega$ for $0<r<m$. The boundary endpoints form a null set in the $r$-integration. The mean value theorem and \eqref{eq:source-C1-bound} therefore give
		\[
		\big|(\mathcal T_{+,\eps}-\mathcal T_{+,0})S^0(s,y,w)\big|
		\le C\eps\esssup_{\substack{0<\widetilde s<T\\ x\in\mathscr T_V\cap\Omega}}
		\big(|S^0|+|\nabla_xS^0|\big)(\widetilde s,x,w).
		\]
		Taking the essential supremum in $(s,y)$ and integrating in $w$ yields
		\begin{equation}\label{eq:trace-operator-convergence}
			\|(\mathcal T_{+,\eps}-\mathcal T_{+,0})S^0\|_{L^1_wL^\infty_{s,y}}\le C\eps.
		\end{equation}
		Combining \eqref{eq:source-eps-minus-zero}, Lemma~\ref{lem:trace-bound}, and \eqref{eq:trace-operator-convergence} proves \eqref{eq:coefficient-convergence}.
	\end{proof}
	
	\begin{theorem}[Asymptotics of the measured outgoing flux]\label{thm:flux-functional}
		Let $K\in\mathfrak K$, with $\mathfrak K$ defined by \eqref{eq:admissible-kernel-class}. Let $0\le g_i\in C_c^\infty((0,T)\times U\times W_i)$, $i=1,2$, have finite $\mathfrak X_-^1$ norms, where $W_1,W_2\Subset W$ and $W_1-W_2\Subset\mathscr G^\sharp$. Let $g_i^\eps$ be the exact incoming pullbacks defined by \eqref{eq:incoming-pullback}, set $g^\eps=(g_1^\eps,g_2^\eps,0,0)$, and let $H^\eps$ be the corresponding mild solution. Let $W_3\Subset\R^3$ be a bounded open set containing the velocity support of the products generated by $W_1$ and $W_2$. Then there exists $\eps_1>0$ such that, for every $0<\eps\le\eps_1$ and every $\Phi\in C_c^1((0,T)\times U\times W_3)$,
		\begin{equation}\label{eq:weak-flux-asymptotic}
			\begin{split}
				&\bigg|\langle\mathscr A_{K,\eps}^e[g^\eps],\Phi\rangle-\int_U\int_0^T\int_{W_3}\big(\mathcal T_{+,0}Q_{3,+}^K(U_1^0,U_2^0)\big)(s,y,w)\Phi(s,y,w)\,dw\,ds\,dy\bigg|\\
				&\le C\eps\|\Phi\|_{\mathfrak D(U,W_3)}.
			\end{split}
		\end{equation}
		The number $\eps_1$ and the constant $C$ depend on the fixed geometry, the velocity supports, the bounds on the collision rates, the local $C^1$ bound $M_1$ in \eqref{eq:admissible-kernel-class}, and the $\mathfrak X_-^1$ norms of the controls. They may be chosen uniformly for the normalized packet families introduced in Definition~\ref{def:beam-data}.
	\end{theorem}
	
	\begin{proof}
		The outgoing formula \eqref{eq:third-outgoing-trace-definition} expresses $\eps^{-1}\gamma_{+,\eps}H_3^\eps$ as $\mathcal T_{+,\eps}\Q_{K,3}(H^\eps)$. The free product components vanish, so \eqref{eq:channel-identity} identifies $\Q_{K,3}(U^\eps)$ with $Q_{3,+}^K(U_1^\eps,U_2^\eps)$. Applying Lemma~\ref{lem:trace-bound}, \eqref{eq:Q-Lipschitz}, and \eqref{eq:H-U}, we obtain
		\begin{equation}\label{eq:normalized-trace-remainder}
			\begin{split}
				\big\|\eps^{-1}\gamma_{+,\eps}H_3^\eps-\mathcal T_{+,\eps}Q_{3,+}^K(U_1^\eps,U_2^\eps)\big\|_{L^1_wL^\infty_{s,y}}
				\le T\|\Q_K(H^\eps)-\Q_K(U^\eps)\|_{\mathfrak X_T}\le C\eps.
			\end{split}
		\end{equation}
		Lemma~\ref{lem:coefficient-convergence} replaces the coefficient above by $\mathcal T_{+,0}Q_{3,+}^K(U_1^0,U_2^0)$ with another error $C\eps$.
		
		Finally, $a_\eps(w)\cdot e=1+O(\eps)$ uniformly for $w\in W_3$. The limiting coefficient is uniformly bounded in $L^1_wL^\infty_{s,y}$, so multiplication by $a_\eps(w)\cdot e$ introduces only another $O(\eps)$ error. Pairing with $\Phi$ and using the definition \eqref{eq:detector-norm} proves
		\eqref{eq:weak-flux-asymptotic}.
	\end{proof}
	
	By \eqref{eq:physical-detector-identity}, Theorem~\ref{thm:flux-functional} gives the asymptotic behavior of the measured outgoing flux. The argument uses the characteristic representative defined above and does not require a pointwise kinetic trace theorem.

	\section{Product measurements and the X-ray transform}\label{sec:measurements}
	
	\subsection{Synchronized incident pulses}\label{subsec:incident-pulses}
	
	Fix $e\in\Sph$ and a relatively compact open set $U\Subset\Omega_e$. Choose an open set $U^\sharp$ such that $U\Subset U^\sharp\Subset\Omega_e$, and choose $\zeta\in C_c^\infty(U^\sharp)$ satisfying $0\le\zeta\le1$ and $\zeta=1$ on a neighborhood of $\ol U$. Since $T>\diam(\Omega)$, we may choose $\tau_*>0$ and $s_0>2\tau_*$ such that $s_0+\diam(\Omega)+2\tau_*<T$. Choose $\alpha\in C_c^\infty(\R)$ such that $0\le\alpha\le1$, $\alpha=1$ on $[-2\tau_*,2\tau_*]$, and $s_0+\supp\alpha\Subset(0,T)$. Then $s_0+L(y,e)+[-\tau_*,\tau_*]\Subset(0,T)$ for every $y\in U$. Fix $g_0\in\mathscr G$. We choose
	\begin{equation}\label{eq:offsets}
		w_1=\frac{m_2}{M}g_0, \quad w_2=-\frac{m_1}{M}g_0.
	\end{equation}
	Then $w_1-w_2=g_0$ and $m_1w_1+m_2w_2=0$. Let $\psi\in C_c^\infty(\R^3)$ be nonnegative with integral one, and set $\psi_{i,\delta}(w)=\delta^{-3}\psi((w-w_i)/\delta)$ for $i=1,2$. Since $\supp\psi_{i,\delta}=w_i+\delta\supp\psi$, these supports shrink to $w_i$ as $\delta\to0$. Since $g_0\in\mathscr G\Subset\mathscr G^\sharp$, we may choose $\delta_0>0$ and bounded open sets $W_1,W_2,W\Subset\R^3$ such that $w_i\in W_i\Subset W$, $\supp\psi_{i,\delta}\subset W_i$ for $0<\delta\le\delta_0$, and $W_1-W_2\Subset\mathscr G^\sharp$. If $g_0$ varies in a compact subset $\mathscr G_0\Subset\mathscr G$, we may choose $\delta_0$ and $W$ independently of $g_0$. We take $W_i(g_0)$ to be balls of a fixed sufficiently small radius about $w_i(g_0)$, so that all the sets $\ol{W_1(g_0)-W_2(g_0)}$ lie in a fixed compact subset of $\mathscr G^\sharp$.
	
	After decreasing $\eps_0>0$ if necessary, we apply Lemma~\ref{lem:characteristic-maps} with $U^\sharp$ in place of $U$ and with the fixed velocity set $W$.
	
	\begin{definition}[Restricted beam data]\label{def:beam-data}
		For $i=1,2$, define
		\begin{equation}\label{eq:beam-control}
			g_{i,\delta}(s,y,w)=\zeta(y)\alpha(s-s_0)\psi_{i,\delta}(w),\quad (s,y,w)\in(0,T)\times U^\sharp\times W.
		\end{equation}
		For $0\le\eps\le\eps_0$, let $g_{i,\delta}^\eps$ be the exact incoming pullback of $g_{i,\delta}$, namely,
		\begin{equation}\label{eq:beam-incoming-pullback}
			g_{i,\delta}^\eps\big(s,x_-^\eps(y,w),w\big)=g_{i,\delta}(s,y,w),\quad (s,y,w)\in(0,T)\times U^\sharp\times W,
		\end{equation}
		and set $g_{i,\delta}^\eps=0$ on the rest of $\Sigma_{-,\eps}$. The third and fourth incoming components vanish,
		and we set $g^{\eps,\delta}=\big(g_{1,\delta}^\eps,g_{2,\delta}^\eps,0,0\big)$.
	\end{definition}
	
	As in the exact incoming pullback construction, we extend $g_{i,\delta}$ by zero to a smooth function on $\R\times e^\perp\times\R^3$.
	
	For $0<\eps\le\eps_0$ and $R=\eps^{-1}$, define the physical incoming data, for $i=1,2$, by
	\begin{equation}\label{eq:physical-beam-data-definition}
		G_{i,R}^{\eps,\delta}(t,x,v):=g_{i,\delta}^{\eps}\big(Rt,x,v-Re\big),\quad (t,x,v)\in(0,T/R)\times\Gamma_-.
	\end{equation}
	If $w=v-Re$, then $a_\eps(w)=e+\eps(v-Re)=\eps v$, since $\eps R=1$. In particular, $n(x)\cdot v<0$ if and only if $n(x)\cdot a_\eps(w)<0$, so the change of variables preserves the incoming condition. We may also write
	\begin{equation}\label{eq:physical-scaled-beam-relation}
		G_{i,R}^{\eps,\delta}(t,x,Re+w)=g_{i,\delta}^{\eps}(Rt,x,w),\quad i=1,2.
	\end{equation}
	
	On the incoming boundary patch parametrized by $(s,y,w)\in(0,T)\times U^\sharp\times W$, relations \eqref{eq:beam-incoming-pullback} and \eqref{eq:physical-scaled-beam-relation} give
	\begin{equation}\label{eq:physical-beam-incoming-data}
		\begin{aligned}
			G_{i,R}^{\eps,\delta}\big(\eps s,x_-^\eps(y,w),Re+w\big)=g_{i,\delta}^{\eps}		\big(s,x_-^\eps(y,w),w\big)=g_{i,\delta}(s,y,w),\quad i=1,2.
		\end{aligned}
	\end{equation}
	Since $g_{i,\delta}^{\eps}=0$ on the rest of $\Sigma_{-,\eps}$, definition \eqref{eq:physical-beam-data-definition} gives $G_{i,R}^{\eps,\delta}=0$ outside the corresponding incoming boundary patch. The third and fourth species have zero incoming data: $G_{3,R}^{\eps,\delta}=G_{4,R}^{\eps,\delta}=0$.
	
	Let $U_\delta^\eps=(U_{1,\delta}^\eps,\ldots,U_{4,\delta}^\eps)$ denote the free transport solution with zero initial data and incoming boundary value $g^{\eps,\delta}$. We next verify that the incoming norms are uniform in $\eps$ and $\delta$. By \eqref{eq:incoming-uniform}, the bounds $0\le\zeta,\alpha\le1$, and the normalization of the velocity packets,
	\begin{align*}
		\|g^{\eps,\delta}\|_{\mathfrak X_{-,\eps}}\le \sum_{i=1}^2\int_W\esssup_{(s,y)\in(0,T)\times U^\sharp}
		\big|\zeta(y)\alpha(s-s_0)\psi_{i,\delta}(w)\big|\,dw\le\sum_{i=1}^2\int_W\psi_{i,\delta}(w)\,dw=2.
	\end{align*}
	Here we used $\supp\psi_{i,\delta}\subset W$ and $\int_{\R^3}\psi_{i,\delta}(w)\,dw=1$. Thus, the incoming norm is bounded independently of $\eps$ and $\delta$.
	
	The $\mathfrak X_-^1$ norms needed for the estimates along free characteristics are also bounded independently of $\delta$, because the derivatives in $s$ and $y$ act only on the fixed profiles $\alpha$ and $\zeta$. In particular,
	\[
	\|g_{i,\delta}\|_{\mathfrak X_-^1}\le C_{\zeta,\alpha}\int_{\R^3}\psi_{i,\delta}(w)\,dw=C_{\zeta,\alpha},
	\quad i=1,2, 
	\]
	where $C_{\zeta,\alpha}=\|\zeta\|_{L^\infty}(\|\alpha\|_{L^\infty}+\|\alpha'\|_{L^\infty})+\|\nabla_y\zeta\|_{L^\infty}\|\alpha\|_{L^\infty}$. Thus, $\|g_{1,\delta}\|_{\mathfrak X_-^1}+\|g_{2,\delta}\|_{\mathfrak X_-^1}$ is bounded independently of $\delta$. After decreasing $\eps_0$ once more, assume that $2\eps_0TC_0\le1/16$. Since $\|g^{\eps,\delta}\|_{\mathfrak X_{-,\eps}}\le2$, Theorem~\ref{thm:forward-wp} applies for every $0<\eps\le\eps_0$ and $0<\delta\le\delta_0$. Let $H_\delta^\eps$ denote the resulting nonnegative mild solution of \eqref{eq:scaled-ibvp}. The outgoing trace and projected flux associated with $H_\delta^\eps$ are denoted as in Subsection~\ref{subsec:boundary-flux}. From this point on, their dependence on $\delta$ is suppressed from the notation.
	
	At $\eps=0$, if $y\in U^\sharp$ and $\sigma_-(y,e)<\sigma<\sigma_+(y,e)$, then $\tau_{-,0}(y+\sigma e,w)=\sigma-\sigma_-(y,e)$ and $Y_{-,0}(y+\sigma e,w)=y$. Formula \eqref{eq:free-solution-formula} gives
	\begin{equation}\label{eq:free-beams}
		\begin{aligned}
			U_{i,\delta}^0(s,y+\sigma e,w)=g_{i,\delta}\big(s-\sigma+\sigma_-(y,e),y,w\big)=\zeta(y)\alpha\big(s-s_0-(\sigma-\sigma_-(y,e))\big)\psi_{i,\delta}(w),\quad i=1,2.
		\end{aligned}
	\end{equation}

	\begin{lemma}[Constancy of the total flight time]\label{lem:focusing}
		Fix $y\in U$ and $\sigma_-(y,e)<\sigma<\sigma_+(y,e)$. At $\eps=0$, consider the contribution of a forward reaction at $y+\sigma e$ to the third component observed at $x_+(y,e)$ and at the scaled exit time $s_{\rm out}=s_0+L(y,e)+r$. The corresponding reaction time is $s_{\rm col}=s_{\rm out}-\big(\sigma_+(y,e)-\sigma\big)$. At this time, the two incoming free fields satisfy
		\[
		\begin{aligned}
			U_{1,\delta}^0(s_{\rm col},y+\sigma e,w)&=\zeta(y)\alpha(r)\psi_{1,\delta}(w),\\
			U_{2,\delta}^0(s_{\rm col},y+\sigma e,w_*)&=\zeta(y)\alpha(r)\psi_{2,\delta}(w_*).
		\end{aligned}
		\]
		Furthermore, the temporal factor in the product $U_{1,\delta}^0U_{2,\delta}^0$ is $\alpha(r)^2$ and is independent of the reaction position $\sigma$. In particular, if $|r|\le\tau_*$, then $\alpha(r)^2=1$.
	\end{lemma}
	
	\begin{proof}
		At $\eps=0$, we have $a_0(\widetilde w)=e$ for every recentered velocity $\widetilde w\in\R^3$. Since $|e|=1$, the scaled travel time from $y+\sigma e$ to $x_+(y,e)=y+\sigma_+(y,e)e$ is $\sigma_+(y,e)-\sigma$. A particle observed at time $s_{\rm out}$ was created at time $s_{\rm col}=s_0+L(y,e)+r-\big(\sigma_+(y,e)-\sigma\big)$.

		By \eqref{eq:free-beams}, the argument of the temporal pulse in either incoming free field at $(s_{\rm col},y+\sigma e)$ is
		\[
		\begin{aligned}
			s_{\rm col}-s_0-\big(\sigma-\sigma_-(y,e)\big)=L(y,e)+r-\big(\sigma_+(y,e)-\sigma\big)
			-\big(\sigma-\sigma_-(y,e)\big)=r,
		\end{aligned}
		\]
		because $L(y,e)=\sigma_+(y,e)-\sigma_-(y,e)$. Substitution into \eqref{eq:free-beams} gives the two formulas, whose temporal factors are independent of $\sigma$. Since $\alpha=1$ on $[-2\tau_*,2\tau_*]$, we have $\alpha(r)^2=1$ for $|r|\le\tau_*$.
	\end{proof}
	
	Recall that $w=v-Re$ is the recentered velocity relative to the common physical velocity $Re$. By \eqref{eq:offsets}, the central incoming pair $Re+w_1$, $Re+w_2$ has relative velocity $g_0$ and center-of-mass velocity $Re$. Substitution into the forward reaction map \eqref{eq:forward-map} gives
	\[
	\mathcal V_3^+(Re+w_1,Re+w_2,\omega)=Re+\lambda_3(g_0)\omega.
	\]
	We set
	\begin{equation}\label{eq:reaction-sphere}
		\lambda_0:=\lambda_3(g_0)=\frac{m_4}{M}\rho_+(|g_0|)>0,
		\quad \Sigma_3(g_0):=\{\lambda_0\omega:\ \omega\in\Sph\}.
	\end{equation}
	For the central incoming velocities, the recentered velocities of the third species lie on $\Sigma_3(g_0)$, and their physical velocities lie on $Re+\Sigma_3(g_0)$. The radius $\lambda_0$ is determined by $g_0$, the masses, and the reaction energy through $\rho_+$. Packets of width $\delta>0$ produce velocities in a neighborhood of this sphere.
	
	\subsection{Detection of product velocities}\label{subsec:velocity-detection}
	
	Let $\mathcal O\Subset\Sph$ be an accessible set of outgoing directions, and let $\varphi\in C_c^\infty(\mathcal O)$. We use the same notation for its smooth zero extension to $\Sph$. To realize this angular test by a detector acting on product velocities in $\R^3$, we extend $\varphi$ to a smooth function of the recentered velocity.
	
	Write $z\in\R^3$ for the recentered velocity of the third species, so its physical velocity is $Re+z$. Choose $0<d_0<\lambda_0/4$ and $\vartheta\in C_c^\infty((-2d_0,2d_0))$ such that $0\le\vartheta\le1$ and $\vartheta=1$ on $[-d_0,d_0]$. If $g_0$ varies in a compact set $\mathscr G_0\Subset\mathscr G$, choose $d_0$ and $\vartheta$ independently of $g_0$, with $4d_0<\inf_{g\in\mathscr G_0}\lambda_3(g)$. Define
	\begin{equation}\label{eq:tubular-filter}
		\chi_{g_0,\varphi}(z)=
		\begin{cases}
			\vartheta(|z|-\lambda_0)\varphi\big(\frac{z}{|z|}\big),&\text{ if }z\neq0,\\
			0,&\text{ if }z=0.
		\end{cases}
	\end{equation}
	The radial factor restricts the support of $\chi_{g_0,\varphi}$ to $\{z\in\R^3:\, \big||z|-\lambda_0\big|<2d_0\big\}$. Since $2d_0<\lambda_0/2$, this set is disjoint from $\{|z|\le\lambda_0/2\}$, so the angular factor is smooth wherever the radial factor is nonzero. It follows that $\chi_{g_0,\varphi}\in C_c^\infty(\R^3)$. On $\Sigma_3(g_0)$,
	\begin{equation}\label{eq:filter-on-sphere}
		\chi_{g_0,\varphi}(\lambda_0\omega)=\vartheta(0)\varphi(\omega)=\varphi(\omega), \quad \omega\in\Sph.
	\end{equation}
	For a physical product velocity $v=Re+\lambda_0\omega$, the filter $\chi_{g_0,\varphi}(v-Re)$ equals $\varphi(\omega)$.
	
	\begin{lemma}[Velocity localization]\label{lem:velocity-localization}
		Let $K\in\mathfrak K$, with $\mathfrak K$ defined by \eqref{eq:admissible-kernel-class}, and define
		\begin{equation}\label{eq:W3}
			\mathcal W_3(w,w_*,\omega)=\frac{m_1w+m_2w_*}{M}+\frac{m_4}{M}\rho_+(|w-w_*|)\omega.
		\end{equation}
		For $0<\delta\le\delta_0$, set
		\begin{equation}\label{eq:A-delta-velocity-localization}
			\begin{split}
				A_\delta(x)&=\int_{\R^3}\int_{\R^3}\int_{\Sph}K(x,w-w_*,\omega)\psi_{1,\delta}(w)\psi_{2,\delta}(w_*)\chi_{g_0,\varphi}\big(\mathcal W_3(w,w_*,\omega)\big)\,d\omega\,dw_*\,dw,\\
				A_0(x)&=\int_{\Sph}K(x,g_0,\omega)\varphi(\omega)\,d\omega.
			\end{split}
		\end{equation}
		For every compact set $\mathscr G_0\Subset\mathscr G$, there exist $C>0$ and $0<\delta_1\le\delta_0$ such that
		\begin{equation}\label{eq:velocity-localization-error}
			|A_\delta(x)-A_0(x)|\le C\delta\|\varphi\|_{C^1(\Sph)}
		\end{equation}
		for every $x\in\ol\Omega$, $g_0\in\mathscr G_0$, and $0<\delta\le\delta_1$, with a constant independent of $x$, $g_0$, $\delta$, and $\varphi$. In particular,
		\begin{equation}\label{eq:velocity-localization}
			\lim_{\delta\to0}A_\delta(x)=\int_{\Sph}K(x,g_0,\omega)\varphi(\omega)\,d\omega
		\end{equation}
		uniformly in $x$, in $g_0$ on compact subsets of $\mathscr G$, and in $\varphi$ on bounded subsets of $C^1(\Sph)$.
	\end{lemma}
	
	\begin{proof}
		Fix a compact set $\mathscr G_0\Subset\mathscr G$. All constants below may depend on $\mathscr G_0$, $\psi$, $\vartheta$, the masses, and $M_1$, but not on $x$, $g_0$, $\delta$, or $\varphi$. Since $\psi$ has compact support, there exists $C_\psi>0$ such that $|u|\le C_\psi$ for every $u\in\supp\psi$. Set $d_{\mathscr G}=\dist(\mathscr G_0,\R^3\setminus\mathscr G^\sharp)>0$, and choose $0<\delta_1\le\delta_0$ such that $2C_\psi\delta_1<d_{\mathscr G}$.
		
		If $\psi_{1,\delta}(w)\psi_{2,\delta}(w_*)\ne0$, then $w=w_1+\delta u$ and $w_*=w_2+\delta u_*$ for some $u,u_*\in\supp\psi$. Set $\widetilde g=w-w_*$. By \eqref{eq:offsets},
		\[
		|\widetilde g-g_0|\le2C_\psi\delta,
		\quad
		\bigg|\frac{m_1w+m_2w_*}{M}\bigg|\le C_\psi\delta.
		\]
		For $0<\delta\le\delta_1$, the segment joining $g_0$ and $\widetilde g$ is contained in $\mathscr G^\sharp$. Since $\lambda_3$ is $C^1$ on $\ol{\mathscr G^\sharp}$ and $\lambda_0=\lambda_3(g_0)$,
		\[
		|\lambda_3(\widetilde g)-\lambda_0|\le C\delta,
		\quad
		|\mathcal W_3(w,w_*,\omega)-\lambda_0\omega|\le C\delta.
		\]
		
		Let $\lambda_*=\inf_{g\in\mathscr G_0}\lambda_3(g)>0$. The choice of $d_0$ before \eqref{eq:tubular-filter} gives $4d_0<\lambda_*$, so $|z|\ge\lambda_*-2d_0>\lambda_*/2$ on the support of $\chi_{g_0,\varphi}$. Since $z\mapsto z/|z|$ has derivative $(I-\widehat z\otimes\widehat z)/|z|$ away from the origin, differentiation of \eqref{eq:tubular-filter} gives
		\begin{equation}\label{eq:velocity-filter-C1-bound}
			\|\chi_{g_0,\varphi}\|_{C^1(\R^3)}\le C\|\varphi\|_{C^1(\Sph)}
		\end{equation}
		uniformly for $g_0\in\mathscr G_0$.
		
		By \eqref{eq:admissible-kernel-class}, the derivatives of $K$ with respect to $g$ are bounded on $\ol\Omega\times\ol{\mathscr G^\sharp}\times\Sph$. Since the segment from $g_0$ to $\widetilde g$ lies in $\mathscr G^\sharp$, we have $|K(x,\widetilde g,\omega)-K(x,g_0,\omega)|\le C\delta$. Combining this bound with \eqref{eq:velocity-filter-C1-bound}, \eqref{eq:filter-on-sphere}, and $|K(x,g_0,\omega)|\le M_1$ gives
		\begin{align*}
			&\big|K(x,\widetilde g,\omega)\chi_{g_0,\varphi}\big(\mathcal W_3(w,w_*,\omega)\big)-K(x,g_0,\omega)\varphi(\omega)\big|\\
			&\le|K(x,\widetilde g,\omega)-K(x,g_0,\omega)|\big|\chi_{g_0,\varphi}\big(\mathcal W_3(w,w_*,\omega)\big)\big| +|K(x,g_0,\omega)|\big|\chi_{g_0,\varphi}\big(\mathcal W_3(w,w_*,\omega)\big)-\chi_{g_0,\varphi}(\lambda_0\omega)\big|\\
			&\le C\delta\|\varphi\|_{C^1(\Sph)}.
		\end{align*}
		Since the packets are nonnegative and satisfy $\int_{\R^3}\psi_{i,\delta}(w)\,dw=1$, integration in $(w,w_*,\omega)$ gives \eqref{eq:velocity-localization-error}. The constant absorbs the surface measure of $\Sph$ and is uniform for $x\in\ol\Omega$ and $g_0\in\mathscr G_0$. Formula \eqref{eq:velocity-localization} follows immediately.
	\end{proof}
	
	The quantity $A_\delta$ in \eqref{eq:A-delta-velocity-localization} tests the product distribution against a function of the recentered velocity $z\in\R^3$, and \eqref{eq:velocity-localization} gives its limit as $\delta\to0$. Under the parametrization $z=\lambda_0\omega$ of $\Sigma_3(g_0)$, the surface measures satisfy $dS_z=\lambda_0^2\,d\omega$. For every $\chi\in C_c^\infty(\R^3)$,
	\begin{equation}\label{eq:reaction-sphere-change-of-measure}
		\int_{\Sph}K(x,g_0,\omega)\chi(\lambda_0\omega)\,d\omega
		=\int_{\Sigma_3(g_0)}\lambda_0^{-2}K\Big(x,g_0,\frac{z}{\lambda_0}\Big)\chi(z)\,dS_z.
	\end{equation}
	The weak gain formula uses the integral on the left, with the angular measure $d\omega$. Taking $\chi=\chi_{g_0,\varphi}$ and using \eqref{eq:filter-on-sphere} gives $\int_{\Sph}K(x,g_0,\omega)\chi_{g_0,\varphi}(\lambda_0\omega)\,d\omega
	=\int_{\Sph}K(x,g_0,\omega)\varphi(\omega)\,d\omega$. No factor involving $\lambda_0$ appears in this angular formulation. The factor $\lambda_0^{-2}$ appears only when the same measure is written as a density with respect to the surface measure $dS_z$ on $\Sigma_3(g_0)$.
	
	For $0<\delta\le\delta_0$, the product velocities generated by the two packet supports lie in a fixed compact subset of $\R^3$. Choose a bounded open set $W_3\Subset\R^3$ containing this set and $\supp\chi_{g_0,\varphi}$ as compact subsets. We use $W_3$ in the outgoing flux functional. After decreasing $\eps_0$ if necessary, Lemma~\ref{lem:characteristic-maps} applies with $U^\sharp$ and $W_3$, and the maps $x_\pm^\eps(y,w)$ and $L_\eps(y,w)$ are defined on $U^\sharp\times W_3$ for $0\le\eps\le\eps_0$.
	
	Choose $0\le\beta\in C_c^\infty((-\tau_*,\tau_*))$ with $\int_{\R}\beta(r)\,dr=1$. For $y\in U$, the detector uses the temporal weight $\beta(s-s_0-L(y,e))$. Its response is supported where $|s-s_0-L(y,e)|<\tau_*$. By Lemma~\ref{lem:focusing}, at $\eps=0$ the total travel time from entry to exit is independent of the position of a single forward reaction along the chord. The temporal test includes contributions from every point of the chord.
	
	On the support of $\beta$, we have $|r|<\tau_*$ and $\alpha(r)=1$. The normalization gives $\int_{\R}\beta(r)\alpha(r)^2\,dr=1$, which is the factor obtained by integrating in exit time in the proof of Theorem~\ref{thm:boundary-to-xray}.
	
	Since $s=Rt$, the detector response is supported where $|t-(s_0+L(y,e))/R|<\tau_*/R$ in physical time. Choose $\eta\in C_c^\infty(U)$ to specify the spatial weight assigned to the chords.
	
	\begin{definition}[Restricted measurement in scaled coordinates]\label{def:measured-functional}
		For the packet incoming data $g^{\eps,\delta}$ in Definition~\ref{def:beam-data}, define the scaled detector profile 
		\begin{equation}\label{eq:scaled-detector-profile}
			\Phi^{e,g_0;\eta,\varphi}(s,y,w):=\eta(y)\beta\big(s-s_0-L(y,e)\big)\chi_{g_0,\varphi}(w),
		\end{equation}
		for $(s,y,w)\in(0,T)\times U\times W_3$.
		
		The corresponding measurement is obtained by testing the scaled outgoing flux against this profile:
		\begin{equation}\label{eq:measured-functional}
			\begin{aligned}
				\mathscr M_K^{\eps,\delta}(e,g_0;\eta,\varphi)&:=\big\langle\mathscr A_{K,\eps}^e[g^{\eps,\delta}],
				\Phi^{e,g_0;\eta,\varphi}\big\rangle\\
				&\ =\int_U\int_0^T\int_{W_3}\eta(y)\beta\big(s-s_0-L(y,e)\big)\chi_{g_0,\varphi}(w)\mathscr J_{3,\eps}^e(s,y,w)\,dw\,ds\,dy.
			\end{aligned}
		\end{equation}
		Here $\mathscr J_{3,\eps}^e$ is the projected outgoing flux generated by the packet incoming data $g^{\eps,\delta}$. Its dependence on $\delta$ is suppressed in the notation.
	\end{definition}
	
	The factors $\eta$, $\beta$, and $\chi_{g_0,\varphi}$ in \eqref{eq:scaled-detector-profile} act on the chord parameter, exit time, and product velocity, respectively.
	
	We next relate the scaled definition \eqref{eq:measured-functional} to the physical detector reading in \eqref{eq:intro-measurement-symbol}. Let $R=\eps^{-1}$. For $(t,x,v)\in(0,T/R)\times\Gamma_+$, write $w=v-Re$ and define
	\begin{equation}\label{eq:physical-restricted-detector}
		\Psi_\eps^{e,g_0;\eta,\varphi}(t,x,v)=
		\begin{cases}
			\Phi^{e,g_0;\eta,\varphi}\big(Rt,\pi_{\eps,w}^e(x),w\big),
			& w\in W_3\ \text{and}\ \pi_{\eps,w}^e(x)\in U,\\
			0,&\text{otherwise}.
		\end{cases}
	\end{equation}
	The supports of $\eta$, $\beta$, and $\chi_{g_0,\varphi}$ are compactly contained in their respective coordinate domains. The zero extension in \eqref{eq:physical-restricted-detector} is a compactly supported $C^1$ response on the outgoing phase boundary. Its temporal support is contained in $(0,T/R)$ by the choice of $s_0$, $T$, and $\tau_*$. Lemma~\ref{lem:characteristic-maps} gives a uniform positive lower bound for $n(x)\cdot a_\eps(w)$ on the support of this response.
	
	The change of variables $v=Re+w$ has unit Jacobian. Applying Proposition~\ref{prop:exact-detector-identity} with $\Phi=\Phi^{e,g_0;\eta,\varphi}$ gives
	\begin{equation}\label{eq:physical-scaled-measurement}
		\mathscr M_K^{\eps,\delta}(e,g_0;\eta,\varphi)
		=\big\langle\mathscr A_{K,\eps}^e[g^{\eps,\delta}],\Phi^{e,g_0;\eta,\varphi}\big\rangle
		=R\big\langle\A_{K,R}(G_{1,R}^{\eps,\delta},G_{2,R}^{\eps,\delta}),\Psi_\eps^{e,g_0;\eta,\varphi}\big\rangle,
	\end{equation}
	where $\A_{K,R}$ gives the outgoing product flux in the original variables, $\mathscr A_{K,\eps}^e$ is its normalized form in scaled boundary coordinates, and $\mathscr M_K^{\eps,\delta}$ is the resulting detector measurement.

	\subsection{Recovery of the spatial X-ray transform}\label{subsec:boundary-xray}
	
	\begin{theorem}[Tomographic limit of the restricted measurements]\label{thm:boundary-to-xray}
		Let $K\in\mathfrak K$ with $\mathfrak K$ defined by \eqref{eq:admissible-kernel-class}. For the beam and detector families above,
		\begin{equation}\label{eq:boundary-to-xray}
			\begin{split}
				\lim_{\delta\to0}\lim_{\eps\to0}
				\mathscr M_{K}^{\eps,\delta}(e,g_0;\eta,\varphi)=\int_{e^\perp}\eta(y)\int_{\Sph}\varphi(\omega)
				\bigg[\int_{\R}K(y+\sigma e,g_0,\omega)\,d\sigma\bigg]
				d\omega\,dy.
			\end{split}
		\end{equation}
		For fixed $e$ and $U$, the convergence is uniform when $g_0$ ranges over a compact subset of $\mathscr G$ and when $\eta$ and $\varphi$ range over bounded smooth families whose supports remain in fixed compact subsets of $U$ and $\mathcal O$.
	\end{theorem}
	
	For fixed detector tests, \eqref{eq:finite-resolution-error} also gives the joint limit as $(\eps,\delta)\to(0,0)$, without a relation between the two scales.
	
	\begin{proof}[Proof of Theorem \ref{thm:boundary-to-xray}]
		Let $\Phi=\Phi^{e,g_0;\eta,\varphi}$ be the scaled detector profile defined in \eqref{eq:scaled-detector-profile}. Since $\supp\eta\Subset U\Subset U^\sharp$, we extend $\eta$, and hence $\Phi$, by zero in the $y$-variable from $U$ to $U^\sharp$. The extension belongs to $C_c^1((0,T)\times U^\sharp\times W_3)$ and satisfies $\|\Phi\|_{\mathfrak D(U^\sharp,W_3)}=\|\Phi\|_{\mathfrak D(U,W_3)}$. Fix $\delta>0$. The packet controls satisfy the hypotheses of Theorem~\ref{thm:flux-functional}, and their free transport solutions are denoted by $U_{1,\delta}^\eps$ and $U_{2,\delta}^\eps$. Set $S_\delta^0:=Q_{3,+}^K(U_{1,\delta}^0,U_{2,\delta}^0)$. By Definition~\ref{def:measured-functional}, $\mathscr M_K^{\eps,\delta}(e,g_0;\eta,\varphi)=\big\langle\mathscr A_{K,\eps}^e[g^{\eps,\delta}],\Phi^{e,g_0;\eta,\varphi}\big\rangle$. We apply Theorem~\ref{thm:flux-functional} with the spatial coordinate set $U^\sharp$. Since $\Phi$ vanishes on $U^\sharp\setminus U$, the theorem gives
		\begin{equation}\label{eq:tomographic-first-limit-error}
			\begin{aligned}
				\bigg|\mathscr M_K^{\eps,\delta}(e,g_0;\eta,\varphi)-\int_U\int_0^T\int_{W_3}(\mathcal T_{+,0}S_\delta^0)(s,y,w_3)\Phi^{e,g_0;\eta,\varphi}(s,y,w_3)\,dw_3\,ds\,dy\bigg| \le C\eps \|\Phi^{e,g_0;\eta,\varphi}\|_{\mathfrak D(U,W_3)}.
			\end{aligned}
		\end{equation}
		The detector profile does not depend on $\eps$, and its norm is finite. Indeed,  $\|\Phi^{e,g_0;\eta,\varphi}\|_{\mathfrak D(U,W_3)}\le\|\eta\|_{L^1(U)}\|\beta\|_{L^1(\R)}
		\|\chi_{g_0,\varphi}\|_{L^\infty(W_3)}$. The integral in \eqref{eq:tomographic-first-limit-error} is also independent of $\eps$. Letting $\eps\to0$ in the estimate \eqref{eq:tomographic-first-limit-error} gives
		\begin{align}\label{eq:tomographic-proof-first-limit}
			\lim_{\eps\to0}\mathscr M_K^{\eps,\delta}(e,g_0;\eta,\varphi)=\int_U\int_0^T\int_{W_3}
			(\mathcal T_{+,0}S_\delta^0)(s,y,w_3)\Phi^{e,g_0;\eta,\varphi}(s,y,w_3)\,dw_3\,ds\,dy.
		\end{align}
		
		For $s=s_0+L(y,e)+r$ with $\beta(r)\ne0$, one has $|r|<\tau_*$. Since $s_0>\tau_*$, $s-L(y,e)=s_0+r>0$, so the upper limit in the definition of $\mathcal T_{+,0}$ is $L(y,e)$. At $\eps=0$, one has $a_0(w_3)=e$, $x_+^0(y,w_3)=x_+(y,e)$, and $L_0(y,w_3)=L(y,e)$, where $L_\eps(y,w)$ is given by \eqref{eq:def_L_eps}. Formula \eqref{eq:outgoing-characteristic-trace} gives
		\[
		(\mathcal T_{+,0}S_\delta^0)(s,y,w_3)
		=\int_0^{L(y,e)}S_\delta^0\big(s-\ell,x_+(y,e)-\ell e,w_3\big)\,d\ell.
		\]
		Set $\sigma=\sigma_+(y,e)-\ell$. Since $x_+(y,e)=y+\sigma_+(y,e)e$, one has $x_+(y,e)-\ell e=y+\sigma e$ and $s-\ell=s-\sigma_+(y,e)+\sigma$. Reversing the integration limits gives
		\begin{align}\label{eq:full-chord-Duhamel}
			(\mathcal T_{+,0}S_\delta^0)(s,y,w_3)
			=\int_{\sigma_-(y,e)}^{\sigma_+(y,e)}
			S_\delta^0\big(s-\sigma_+(y,e)+\sigma,y+\sigma e,w_3\big)\,d\sigma.
		\end{align}
		
		Set $s_\sigma=s-\sigma_+(y,e)+\sigma$. Then $s_\sigma-s_0-(\sigma-\sigma_-(y,e))=s-s_0-L(y,e)=r$. Formula \eqref{eq:free-beams} gives $U_{i,\delta}^0(s_\sigma,y+\sigma e,z)=\zeta(y)\alpha(r)\psi_{i,\delta}(z)$ for $i=1,2$.
		
		All integrals below are absolutely convergent: the packet relative velocities lie in $\mathscr G^\sharp$, the kernel is bounded on the resulting compact collision set, and the other profiles have compact support. We may apply Fubini's theorem. Substituting \eqref{eq:full-chord-Duhamel} into \eqref{eq:tomographic-proof-first-limit}, we extend the $w_3$ integral from $W_3$ to $\R^3$ because $\supp\chi_{g_0,\varphi}\Subset W_3$. The weak gain formula \eqref{eq:third-gain-weak}, with test function $\chi_{g_0,\varphi}$, then expresses this integral in terms of $w$, $w_*$, and $\omega$, with recentered product velocity $\mathcal W_3(w,w_*,\omega)$ as in \eqref{eq:W3}.
		
		The product of the two input fields in the weak gain formula is
		$\zeta(y)^2\alpha(r)^2\psi_{1,\delta}(w)\psi_{2,\delta}(w_*)$. On the support of $\eta$, one has $\zeta(y)=1$. On the support of $\beta$, one has $\alpha(r)=1$. The choice
		$s_0+L(y,e)+[-\tau_*,\tau_*]\Subset(0,T)$ ensures that the support of $\beta(r)$ is contained in the transformed $s$-integration interval. With the change of variable $r=s-s_0-L(y,e)$, the exit time integration contributes $\int_{\R}\beta(r)\alpha(r)^2\,dr=1$. Extending $\eta$ by zero from $U$ to $e^\perp$ and using the spatial zero extension of $K$ outside $\Omega$ gives
		\begin{equation}\label{eq:packet-averaged-line-integral}
			\begin{aligned}
				\lim_{\eps\to0}\mathscr M_K^{\eps,\delta}(e,g_0;\eta,\varphi)
				&=\int_{e^\perp}\eta(y)\int_{\R}\int_{\R^3}\int_{\R^3}\int_{\Sph}
				K(y+\sigma e,w-w_*,\omega)\psi_{1,\delta}(w)\psi_{2,\delta}(w_*)\\
				&\quad\times\chi_{g_0,\varphi}\big(\mathcal W_3(w,w_*,\omega)\big)
				\,d\omega\,dw_*\,dw\,d\sigma\,dy.
			\end{aligned}
		\end{equation}
		
		Extend $A_\delta$ and $A_0$ by zero from $\ol\Omega$ to $\R^3$. For each $y\in e^\perp$ and $\sigma\in\R$, the inner $(w,w_*,\omega)$-integral in \eqref{eq:packet-averaged-line-integral} is precisely $A_\delta(y+\sigma e)$, while $A_0(y+\sigma e)$ is defined by \eqref{eq:A-delta-velocity-localization}. By \eqref{eq:velocity-localization-error}, $|A_\delta(x)-A_0(x)|
		\le C\delta\|\varphi\|_{C^1(\Sph)}$ uniformly for $x\in\ol\Omega$, and the same estimate holds trivially
		for $x\notin\ol\Omega$. The absolute difference between \eqref{eq:packet-averaged-line-integral} and $\int_{e^\perp}\eta(y)\int_{\R}A_0(y+\sigma e)\,d\sigma\,dy$ is bounded by $C\delta\|\varphi\|_{C^1(\Sph)}\int_U|\eta(y)|L(y,e)\,dy$, which tends to zero as $\delta\to0$. The constants in the preceding estimates are uniform under the compactness assumptions in the theorem. If $g_0$ ranges over a compact subset of $\mathscr G$, the packet centers remain bounded, the relative velocities occurring in the packet integrals stay in a fixed compact subset of $\mathscr G^\sharp$, and the associated product velocities and filter supports lie in a common bounded set. Fixed supports and bounded smooth norms for $\eta$ and $\varphi$ give the remaining uniform bounds. We conclude that \eqref{eq:packet-averaged-line-integral} converges to $\int_{e^\perp}\eta(y)\int_{\R}\int_{\Sph}K(y+\sigma e,g_0,\omega)\varphi(\omega)\,d\omega\,d\sigma\,dy$, which is the right-hand side of \eqref{eq:boundary-to-xray}.
	\end{proof}
	
	For fixed $e\in\Sph$ and $g_0\in\mathscr G$, define $D(y,\omega)=X_xK(y,e;g_0,\omega)$ on $\Omega_e\times\mathcal O$. Theorem~\ref{thm:boundary-to-xray} determines the pairing of $D$ with every test function of the form $\eta(y)\varphi(\omega)$ supported in $\Omega_e\times\mathcal O$. By concentrating both test functions, we recover the value of $D$ at each point.
	
	\begin{corollary}[Pointwise recovery of the spatial X-ray transform]\label{cor:pointwise-line-data}
		Fix $e\in\Sph$, $g_0\in\mathscr G$, $y_0\in\Omega_e$, and $\omega_0\in\mathcal O$. Choose open neighborhoods $U_0\Subset\Omega_e$ of $y_0$ and $\mathcal O_0\Subset\mathcal O$ of $\omega_0$. For all sufficiently small $\rho,\tau>0$, let $\eta_\rho\in C_c^\infty(U_0)$ and $\varphi_\tau\in C_c^\infty(\mathcal O_0)$ be nonnegative and satisfy $\int_{e^\perp}\eta_\rho(y)\,dy=\int_{\Sph}\varphi_\tau(\omega)\,d\omega=1$, $\supp\eta_\rho\subset B_{e^\perp}(y_0,\rho)$, and $\supp\varphi_\tau\subset B_{\Sph}(\omega_0,\tau)$, where $B_{\Sph}(\omega_0,\tau)$ is the geodesic ball on $\Sph$. Then
		\begin{equation}\label{eq:pointwise-line-data}
			\lim_{(\rho,\tau)\to(0,0)}\big[\lim_{\delta\to0}\lim_{\eps\to0}\mathscr M_K^{\eps,\delta}(e,g_0;\eta_\rho,\varphi_\tau)\big]
			=X_xK(y_0,e;g_0,\omega_0).
		\end{equation}
	\end{corollary}
	
	\begin{proof}
		Set $D(y,\omega)=X_xK(y,e;g_0,\omega)$. For $y\in U_0$, the change of variables $\sigma=\sigma_-(y,e)+rL(y,e)$ gives
		\begin{equation}\label{eq:line-data-fixed-interval}
			D(y,\omega)=L(y,e)\int_0^1K\big(y+[\sigma_-(y,e)+rL(y,e)]e,g_0,\omega\big)\,dr.
		\end{equation}
		By Lemma~\ref{lem:chord-geometry}, the functions $\sigma_-(\cdot,e)$ and $L(\cdot,e)$ are continuous on $U_0$. The integrand in \eqref{eq:line-data-fixed-interval} is jointly continuous on $\ol U_0\times[0,1]\times\ol{\mathcal O_0}$. Hence, $D$ is continuous on $U_0\times\mathcal O_0$.
		
		For fixed $\rho$ and $\tau$, Theorem~\ref{thm:boundary-to-xray} gives
		\begin{equation}\label{eq:averaged-pointwise-line-data}
			\lim_{\delta\to0}\lim_{\eps\to0}\mathscr M_K^{\eps,\delta}(e,g_0;\eta_\rho,\varphi_\tau)
			=\int_{e^\perp}\eta_\rho(y)\int_{\Sph}\varphi_\tau(\omega)D(y,\omega)\,d\omega\,dy.
		\end{equation}
		The nonnegativity and the two normalization identities imply
		\begin{align}
			\bigg|\int_{e^\perp}\eta_\rho(y)\int_{\Sph}\varphi_\tau(\omega)D(y,\omega)\,d\omega\,dy-D(y_0,\omega_0)\bigg|\le\sup_{\substack{y\in\supp\eta_\rho\\ \omega\in\supp\varphi_\tau}}|D(y,\omega)-D(y_0,\omega_0)|.
			\label{eq:approximate-identity-line-data-error}
		\end{align}
		The support assumptions imply that the last supremum tends to zero as $(\rho,\tau)\to(0,0)$. Combining this convergence with \eqref{eq:averaged-pointwise-line-data} proves \eqref{eq:pointwise-line-data}.
	\end{proof}

	\section{Recovery of the collision kernel}\label{sec:recovery}
	
	\subsection{Uniqueness and reconstruction}\label{subsec:uniqueness-reconstruction}
	
	\para{Fourier reconstruction} For fixed $(g,\omega)\in\mathscr G\times\mathcal O$, let $\widetilde K_{g,\omega}$ denote the zero extension of $K(\cdot,g,\omega)$ to $\R^3$. 
	
	\begin{lemma}[Fourier slice identity]\label{lem:fourier-slice}
		Let $f\in L^1(\R^3)$. Define $Xf(y,e)=\int_\R f(y+\sigma e)\,d\sigma$ and $\widehat f(\xi)=\int_{\R^3}e^{-\mathsf i x\cdot\xi}f(x)\,dx$. If $\xi\in e^\perp$, then
		\begin{equation}\label{eq:fourier-slice}
			\int_{e^\perp}e^{-\mathsf i y\cdot\xi}Xf(y,e)\,dy=\widehat f(\xi).
		\end{equation}
	\end{lemma}
	
	\begin{proof}
		Write $x=y+\sigma e$. Since $\xi\cdot e=0$, one has $x\cdot\xi=y\cdot\xi$. The identity \eqref{eq:fourier-slice} follows from Fubini's theorem.
	\end{proof}
	
	\begin{corollary}[Reconstruction formula]\label{cor:reconstruction}
		The restricted boundary measurements determine $\widehat{\widetilde K}_{g,\omega}(\xi)$ for every $\xi\in\R^3$. For $\xi\ne0$, choose any $e_\xi\in\Sph$ satisfying $e_\xi\cdot\xi=0$. Then
		\begin{equation}\label{eq:fourier-reconstruction}
			\widehat{\widetilde K}_{g,\omega}(\xi)=\int_{e_\xi^\perp}e^{-\mathsf i y\cdot\xi}X_xK(y,e_\xi;g,\omega)\,dy.
		\end{equation}
		At $\xi=0$, one has $\widehat{\widetilde K}_{g,\omega}(0)=\int_{e^\perp}X_xK(y,e;g,\omega)\,dy$ for any $e\in\Sph$. Moreover,
		\begin{equation}\label{eq:inverse-fourier}
			K(x,g,\omega)=\lim_{\kappa\downarrow0}(2\pi)^{-3}\int_{\R^3}e^{\mathsf i x\cdot\xi}e^{-\kappa|\xi|^2}\widehat{\widetilde K}_{g,\omega}(\xi)\,d\xi,
			\quad x\in\Omega.
		\end{equation}
		The convergence is locally uniform in $x\in\Omega$.
	\end{corollary}
	
	\begin{proof}
		Fix $(g,\omega)\in\mathscr G\times\mathcal O$ and set $f=\widetilde K_{g,\omega}$. Since $\Omega$ is bounded and $K(\cdot,g,\omega)$ is bounded, $f\in L^1(\R^3)$. For every $e\in\Sph$, Theorem~\ref{thm:boundary-to-xray} and Corollary~\ref{cor:pointwise-line-data} determine $Xf(y,e)=X_xK(y,e;g,\omega)$ for every $y\in\Omega_e$.
		
		If $y\notin\ol{\Omega_e}$, then $y+\R e$ does not intersect $\ol\Omega$, so $Xf(y,e)=0$. If $y\in\p\Omega_e$, then $(y+\R e)\cap\ol\Omega$ contains at most one point. Indeed, strict convexity would place the open segment between two distinct points of this intersection in $\Omega$. Its projection would then belong to $\Omega_e$, contradicting $y\in\p\Omega_e$. The intersection has one-dimensional measure zero, so $Xf(y,e)=0$ also in this case. The measurements determine $Xf(\cdot,e)$ on the whole plane $e^\perp$.
		
		Let $\xi\ne0$ and choose $e_\xi\cdot\xi=0$. Since $\xi\in e_\xi^\perp$, Lemma~\ref{lem:fourier-slice} gives \eqref{eq:fourier-reconstruction}. Every admissible choice of $e_\xi$ gives the same value $\widehat f(\xi)$, proving independence of this choice. At $\xi=0$, Fubini's theorem gives $\int_{e^\perp}Xf(y,e)\,dy=\int_{\R^3}f(x)\,dx=\widehat f(0)$.
		
		For $\kappa>0$, let $G_\kappa(x)=(4\pi\kappa)^{-3/2}\exp(-|x|^2/(4\kappa))$. With the Fourier convention used here, the right-hand side of \eqref{eq:inverse-fourier} equals $(G_\kappa*f)(x)$. Every $x\in\Omega$ is a continuity point of $f$, and therefore $(G_\kappa*f)(x)\to K(x,g,\omega)$ as $\kappa\downarrow0$.
		
		Let $E\Subset\Omega$ and choose $r>0$ such that the closed $r$-neighborhood of $E$ is contained in $\Omega$. The function $K(\cdot,g,\omega)$ is uniformly continuous on this neighborhood. Splitting the convolution into $|z|<r$ and $|z|\ge r$, using uniform continuity in the first part and $\int_{|z|\ge r}G_\kappa(z)\,dz\to0$ in the second, gives $\sup_{x\in E}|(G_\kappa*f)(x)-K(x,g,\omega)|\to0$. This proves the asserted local uniform convergence.
	\end{proof}

	\begin{proof}[Proof of Theorem~\ref{thm:main}]
		Fix $e\in\Sph$ and $g_0\in\mathscr G$, and set
		\[
		D_e(y,\omega)=X_x(K_1-K_2)(y,e;g_0,\omega),
		\quad (y,\omega)\in\Omega_e\times\mathcal O.
		\]
		The continuity argument in the proof of Corollary~\ref{cor:pointwise-line-data} shows that $D_e$ is continuous on every relatively compact subset of $\Omega_e\times\mathcal O$.
		
		Let $U\Subset\Omega_e$. Taking first $\eps\to0$ and then $\delta\to0$ in \eqref{eq:intro-measurement-equality}, and applying Theorem~\ref{thm:boundary-to-xray}, gives
		\begin{equation}\label{eq:test-zero}
			\int_U\eta(y)\int_{\Sph}\varphi(\omega)D_e(y,\omega)\,d\omega\,dy=0
		\end{equation}
		for every $\eta\in C_c^\infty(U)$ and $\varphi\in C_c^\infty(\mathcal O)$.
		
		Fix $\varphi\in C_c^\infty(\mathcal O)$ and define $D_{e,\varphi}(y)=\int_{\Sph}\varphi(\omega)D_e(y,\omega)\,d\omega$. Equation \eqref{eq:test-zero} shows that $D_{e,\varphi}=0$ as a distribution on $U$. Since $D_{e,\varphi}$ is continuous, it vanishes pointwise on $U$. As $U\Subset\Omega_e$ is arbitrary, $\int_{\Sph}\varphi(\omega)D_e(y,\omega)\,d\omega=0$ for every $y\in\Omega_e$ and every $\varphi\in C_c^\infty(\mathcal O)$. For each fixed $y$, the function $D_e(y,\cdot)$ therefore vanishes as a distribution on $\mathcal O$. Its continuity in $\omega$ gives
		\[
		X_x(K_1-K_2)(y,e;g_0,\omega)=0
		\]
		for every $y\in\Omega_e$ and $\omega\in\mathcal O$. Since $e\in\Sph$ was arbitrary, this identity holds for every direction.
		
		Fix $\omega\in\mathcal O$, and define
		\[
		f(x)=
		\begin{cases}
			(K_1-K_2)(x,g_0,\omega),&x\in\Omega,\\
			0,&x\notin\Omega.
		\end{cases}
		\]
		For every $e\in\Sph$, one has $Xf(y,e)=0$ for $y\in\Omega_e$. It also vanishes for $y\notin\ol{\Omega_e}$. If $y\in\p\Omega_e$, strict convexity implies that $(y+\R e)\cap\ol\Omega$ contains at most one point, so $Xf(y,e)=0$. Hence, $Xf(\cdot,e)=0$ on all of $e^\perp$.
		
		Let $\xi\ne0$ and choose $e\in\Sph$ with $e\cdot\xi=0$. Lemma~\ref{lem:fourier-slice} gives
		\begin{equation}\label{eq:global-fourier-slice}
			\widehat f(\xi)=\int_{e^\perp}e^{-\mathsf i y\cdot\xi}Xf(y,e)\,dy=0.
		\end{equation}
		Since $f\in L^1(\R^3)$, its Fourier transform is continuous, and therefore $\widehat f(0)=0$ as well. Uniqueness of the Fourier transform yields $f=0$ almost everywhere. The function $K_1-K_2$ is continuous on $\Omega\times\mathscr G\times\mathcal O$, so the equality holds pointwise. Since $g_0\in\mathscr G$ and $\omega\in\mathcal O$ were arbitrary, this proves \eqref{eq:intro-main-conclusion}.
		
		The determination of $X_xK$, the Fourier formula, and the regularized inverse Fourier reconstruction follow from Corollary~\ref{cor:reconstruction}. This completes the proof.
	\end{proof}
	
	\subsection{Stability and finite detector resolution}\label{subsec:stability-resolution}
	
	We first prove stability identities for the spatial X-ray transforms recovered in the limit. Corollary~\ref{cor:raw-flux-noise} gives a separate estimate for errors in the measured boundary flux. Let $de$ be the surface measure on $\Sph$ and $dy$ the Euclidean measure on $e^\perp$. For a compact collision set $\mathcal C\Subset\mathscr G\times\mathcal O$, define
	\begin{equation}\label{eq:asymptotic-data-norm}
		\begin{split}
			\|D\|_{\mathcal Y_{-1/2}(\mathcal C)}^2
			&:=\int_{\mathcal C}\int_{\Sph}\int_{e^\perp}|D(y,e;g,\omega)|^2\,dy\,de\,dg\,d\omega,\\
			\|D\|_{\mathcal Y_0(\mathcal C)}^2
			&:=\int_{\mathcal C}\int_{\Sph}\int_{e^\perp}\big|(-\Delta_y)^{1/4}D(y,e;g,\omega)\big|^2\,dy\,de\,dg\,d\omega.
		\end{split}
	\end{equation}
	Here $(-\Delta_y)^{1/4}$ is the fractional Laplacian defined by the Fourier multiplier $|\xi|^{1/2}$ on the plane $e^\perp$.
	
	\begin{proposition}[Plancherel identities for the spatial X-ray transform]\label{prop:Xray-Plancherel}
		Let $f\in C_c^\infty(\R^3)$, then
		\begin{equation}\label{eq:Xray-Hminus-identity}
			\int_{\Sph}\int_{e^\perp}|Xf(y,e)|^2\,dy\,de=4\pi^2\|f\|_{\dot H^{-1/2}(\R^3)}^2,
		\end{equation}
		where $\|f\|_{\dot H^{-1/2}}^2=(2\pi)^{-3}\int_{\R^3}|\xi|^{-1}|\widehat f(\xi)|^2\,d\xi$.
		Moreover,
		\begin{equation}\label{eq:Xray-L2-filtered}
			\int_{\Sph}\int_{e^\perp}\big|(-\Delta_y)^{1/4}Xf(y,e)
			\big|^2
			\,dy\,de
			=
			4\pi^2\|f\|_{L^2(\R^3)}^2.
		\end{equation}
	\end{proposition}
	
	\begin{proof}
		Plancherel's theorem on $e^\perp$ and Lemma~\ref{lem:fourier-slice} give $\int_{e^\perp}|Xf(y,e)|^2\,dy=(2\pi)^{-2}\int_{e^\perp}|\widehat f(\xi)|^2\,d\xi$. For every nonnegative integrable function $A$,
		\begin{equation}\label{eq:plane-coarea}
			\int_{\Sph}\int_{e^\perp}A(\xi)\,d\xi\,de=2\pi\int_{\R^3}\frac{A(\xi)}{|\xi|}\,d\xi.
		\end{equation}
		To verify \eqref{eq:plane-coarea}, write $\xi=ru$ in each plane $e^\perp$, where $r>0$ and $u\in\Sph\cap e^\perp$. Fubini's theorem on the incidence set $\{(e,u)\in\Sph\times\Sph:\ e\cdot u=0\}$ shows that each fixed $u\in\Sph$ is paired with the great circle $\{e\in\Sph:\ e\cdot u=0\}$, whose length is $2\pi$. The radial measure is $r\,dr$, while $d\xi=r^2\,dr\,du$ in $\R^3$. This gives \eqref{eq:plane-coarea}.
		
		Applying \eqref{eq:plane-coarea} with $A(\xi)=|\widehat f(\xi)|^2$ and using planar Plancherel gives
		\[
		\int_{\Sph}\int_{e^\perp}|Xf(y,e)|^2\,dy\,de
		=(2\pi)^{-2}2\pi\int_{\R^3}\frac{|\widehat f(\xi)|^2}{|\xi|}\,d\xi
		=4\pi^2\|f\|_{\dot H^{-1/2}(\R^3)}^2.
		\]
		For the filtered identity, the Fourier multiplier of $(-\Delta_y)^{1/4}$ on $e^\perp$ is $|\xi|^{1/2}$. 	Planar Plancherel and \eqref{eq:plane-coarea}, now applied with $A(\xi)=|\xi||\widehat f(\xi)|^2$, give
		\[
		\int_{\Sph}\int_{e^\perp}\big|(-\Delta_y)^{1/4}Xf(y,e)\big|^2\,dy\,de
		=(2\pi)^{-2}2\pi\int_{\R^3}|\widehat f(\xi)|^2\,d\xi
		=4\pi^2\|f\|_{L^2(\R^3)}^2,
		\]
		where the last equality follows from the three-dimensional Plancherel identity under the Fourier transform convention used above.
	\end{proof}
	
	\begin{corollary}[Stability from the spatial X-ray transforms]\label{cor:Xray-stability}
		Let $K_1,K_2\in\mathfrak K$, and define
		\[
		D_j(y,e;g,\omega)=X_xK_j(y,e;g,\omega), \quad j=1,2.
		\]
		Let $\widetilde K_j(x,g,\omega)$ denote the zero extension in the spatial variable. Then
		\begin{equation}\label{eq:kernel-Hminus-stability}
			\|\widetilde K_1-\widetilde K_2\|_{L^2(\mathcal C;\dot H_x^{-1/2})}
			=\frac1{2\pi}\|D_1-D_2\|_{\mathcal Y_{-1/2}(\mathcal C)},
		\end{equation}
		and
		\begin{equation}\label{eq:kernel-L2-stability}
			\|\widetilde K_1-\widetilde K_2\|_{L^2(\R^3\times\mathcal C)}
			=\frac1{2\pi}\|D_1-D_2\|_{\mathcal Y_0(\mathcal C)}.
		\end{equation}
	\end{corollary}
	
	\begin{proof}
		Fix $(g,\omega)\in\mathcal C$ and set $f=(\widetilde K_1-\widetilde K_2)(\cdot,g,\omega)$. Then $f\in L^1(\R^3)\cap L^2(\R^3)$ and has compact support. Choose $f_n\in C_c^\infty(\R^3)$ such that $f_n\to f$ in $L^1(\R^3)\cap L^2(\R^3)$.
		
		For $h\in L^1(\R^3)\cap L^2(\R^3)$, the estimate $|\widehat h(\xi)|\le\|h\|_{L^1}$ on $|\xi|\le1$, the inequality $|\xi|^{-1}\le1$ on $|\xi|>1$, and Plancherel's theorem give
		\[
		\|h\|_{\dot H^{-1/2}(\R^3)}^2
		\le C\big(\|h\|_{L^1(\R^3)}^2+\|h\|_{L^2(\R^3)}^2\big).
		\]
		Hence, $f_n\to f$ in $\dot H^{-1/2}(\R^3)$.
		
		Proposition~\ref{prop:Xray-Plancherel}, applied to $f_n-f_m$, shows that $Xf_n$ is Cauchy in $L^2(\{(e,y):e\in\Sph,\ y\in e^\perp\})$. On the other hand,
		\[
		\int_{\Sph}\int_{e^\perp}|X(f_n-f)(y,e)|\,dy\,de
		\le4\pi\|f_n-f\|_{L^1(\R^3)},
		\]
		so $Xf_n\to Xf$ in the corresponding $L^1$ space and in distributions. Its $L^2$ limit is $Xf$. Passing to the limit in \eqref{eq:Xray-Hminus-identity} gives
		\[
		\int_{\Sph}\int_{e^\perp}|Xf(y,e)|^2\,dy\,de
		=4\pi^2\|f\|_{\dot H^{-1/2}(\R^3)}^2.
		\]
		
		Similarly, \eqref{eq:Xray-L2-filtered} shows that $(-\Delta_y)^{1/4}Xf_n$ is Cauchy in the corresponding $L^2$ space. To identify its limit, take a smooth test function compactly supported in a local chart of the line space. Its planar fractional Laplacian is bounded and decays in $y$, uniformly on the compact direction set. Since $Xf_n\to Xf$ in $L^1$ on the line space, self-adjointness of the planar Fourier multiplier gives convergence of the pairings with this test function. A partition of unity in $e$ identifies the limit with $(-\Delta_y)^{1/4}Xf$. Passing to the limit gives
		\[
		\int_{\Sph}\int_{e^\perp}|(-\Delta_y)^{1/4}Xf(y,e)|^2\,dy\,de	=4\pi^2\|f\|_{L^2(\R^3)}^2.
		\]
		Applying these identities for every $(g,\omega)\in\mathcal C$ and integrating with respect to $dg\,d\omega$ proves \eqref{eq:kernel-Hminus-stability} and \eqref{eq:kernel-L2-stability}.
	\end{proof}
	
	\para{Measurements with finite resolution} The proof of Theorem~\ref{thm:boundary-to-xray} also gives an error estimate for finite incident speed and packet width.
	
	\begin{proposition}[Error for finite incident speed and packet width]\label{prop:finite-resolution}
		Let $K\in\mathfrak K$ and fix $e\in\Sph$. Let $g_0$ range over a compact subset of $\mathscr G$, and let the supports of $\eta\in C_c^\infty(U)$ and $\varphi\in C_c^\infty(\mathcal O)$ remain in fixed compact subsets of $U$ and $\mathcal O$, respectively. After decreasing $\eps_0$ and $\delta_0$ uniformly if necessary, one has
		\begin{equation}\label{eq:finite-resolution-error}
			\begin{split}
				&\bigg|\mathscr M_K^{\eps,\delta}(e,g_0;\eta,\varphi)-\int_{e^\perp}\eta(y)\int_{\Sph}\varphi(\omega)\bigg[\int_\R K(y+\sigma e,g_0,\omega)\,d\sigma\bigg]\,d\omega\,dy\bigg|\\
				&\le C\|\eta\|_{L^1(e^\perp)}\big(\eps\|\varphi\|_{L^\infty(\Sph)}+\delta\|\varphi\|_{C^1(\Sph)}\big)
			\end{split}
		\end{equation}
		for $0<\eps\le\eps_0$ and $0<\delta\le\delta_0$. The constant depends on the fixed beam profiles, the compact support sets, the geometry, and the admissible kernel class, but not on $\eps$, $\delta$, $g_0$, $\eta$, or $\varphi$ within these families. In particular, for fixed $\eta$ and $\varphi$, the right-hand side is bounded by $C(\eps+\delta)$.
	\end{proposition}
	
	\begin{proof}
		Extend $A_\delta$ and $A_0$ by zero outside $\ol\Omega$, and set $I_\delta=\int_{e^\perp}\eta(y)\int_\R A_\delta(y+\sigma e)\,d\sigma\,dy$ and $I_0=\int_{e^\perp}\eta(y)\int_\R A_0(y+\sigma e)\,d\sigma\,dy$. The calculation in the proof of Theorem~\ref{thm:boundary-to-xray}, using the weak gain formula and the changes of variables in space and time, shows that the integral in \eqref{eq:tomographic-first-limit-error} equals $I_\delta$. This integral is independent of $\eps$. Since $\|\chi_{g_0,\varphi}\|_{L^\infty}\le C\|\varphi\|_{L^\infty(\Sph)}$, that estimate gives
		\[
		|\mathscr M_K^{\eps,\delta}-I_\delta| \le C\eps\|\eta\|_{L^1(e^\perp)}\|\varphi\|_{L^\infty(\Sph)}.
		\]
		By \eqref{eq:velocity-localization-error} and $L(y,e)\le\diam(\Omega)$,
		\[
		|I_\delta-I_0| \le C\delta\|\eta\|_{L^1(e^\perp)}\|\varphi\|_{C^1(\Sph)}.
		\]
		The definition of $A_0$ shows that $I_0$ is the spatial X-ray pairing in \eqref{eq:finite-resolution-error}. Combining the two estimates proves the result.
	\end{proof}
	
	\begin{corollary}[Finite detector resolution]\label{cor:finite-detector-resolution}
		Fix $e\in\Sph$, $g_0\in\mathscr G$, $y_0\in\Omega_e$, and $\omega_0\in\mathcal O$. Choose the spatial beam profile $\zeta$ and the temporal profiles $\alpha,\beta$ once on fixed neighborhoods of this configuration, independently of $\rho$ and $\tau$. For all sufficiently small $\rho,\tau>0$, let $\eta_\rho\in C_c^\infty(\Omega_e)$ and $\varphi_\tau\in C_c^\infty(\mathcal O)$ be nonnegative and satisfy $\int_{e^\perp}\eta_\rho(y)\,dy=1$ and $\int_{\Sph}\varphi_\tau(\omega)\,d\omega=1$. Assume that, for a fixed constant $c_*>0$, one has $\supp\eta_\rho\subset B_{e^\perp}(y_0,c_*\rho)$ and $\supp\varphi_\tau\subset B_{\Sph}(\omega_0,c_*\tau)$, and that these supports remain in fixed compact subsets of $\Omega_e$ and $\mathcal O$ for all sufficiently small $\rho$ and $\tau$. Then
		\begin{equation}\label{eq:fully-resolved-error}
			\begin{split}
				\big|\mathscr M_K^{\eps,\delta}(e,g_0;\eta_\rho,\varphi_\tau)-X_xK(y_0,e;g_0,\omega_0)\big|
				\le C\big(\eps\|\varphi_\tau\|_{L^\infty(\Sph)}+\delta\|\varphi_\tau\|_{C^1(\Sph)}+\rho+\tau\big).
			\end{split}
		\end{equation}
		For a standard approximate identity on $\Sph$, the bounds $\|\varphi_\tau\|_{L^\infty}\le C\tau^{-2}$ and $\|\varphi_\tau\|_{C^1}\le C\tau^{-3}$ give
		\begin{equation}\label{eq:fully-resolved-scaled-error}
			\big|\mathscr M_K^{\eps,\delta}(e,g_0;\eta_\rho,\varphi_\tau)-X_xK(y_0,e;g_0,\omega_0)\big|
			\le C(\eps\tau^{-2}+\delta\tau^{-3}+\rho+\tau).
		\end{equation}
	\end{corollary}
	
	\begin{proof}
		Proposition~\ref{prop:finite-resolution} gives the first two terms on the right-hand side of \eqref{eq:fully-resolved-error}. Lemma~\ref{lem:chord-geometry} and the $C^1$ regularity of $K$ on $\ol\Omega\times\ol{\mathscr G^\sharp}\times\Sph$ imply that $(y,\omega)\mapsto X_xK(y,e;g_0,\omega)$ is $C^1$ on a fixed compact neighborhood of $(y_0,\omega_0)$. The support and normalization assumptions give
		\[
		\bigg|\int_{e^\perp}\eta_\rho(y)\int_{\Sph}\varphi_\tau(\omega)X_xK(y,e;g_0,\omega)\,d\omega\,dy-X_xK(y_0,e;g_0,\omega_0)\bigg|
		\le C(\rho+\tau).
		\]
		Combining these estimates proves \eqref{eq:fully-resolved-error}. The standard scaling of a normalized bump function on a two-dimensional manifold gives $\|\varphi_\tau\|_{L^\infty}\le C\tau^{-2}$ and $\|\varphi_\tau\|_{C^1}\le C\tau^{-3}$, which yields \eqref{eq:fully-resolved-scaled-error}.
	\end{proof}
	
	For fixed beam profiles, \eqref{eq:fully-resolved-scaled-error} gives pointwise recovery of the line integrals whenever $\rho,\tau\to0$, $\eps\tau^{-2}\to0$, and $\delta\tau^{-3}\to0$. We next allow an error in the measured flux before normalization.
	
	\begin{corollary}[Error in the measured product flux]\label{cor:raw-flux-noise}
		Fix a configuration as in Corollary~\ref{cor:finite-detector-resolution}. Use standard nonnegative approximate identities with $\|\eta_\rho\|_{L^\infty}\le C\rho^{-2}$, $\|\varphi_\tau\|_{L^\infty}\le C\tau^{-2}$, and $\|\varphi_\tau\|_{C^1}\le C\tau^{-3}$. Keep the beam profiles fixed as in that corollary. For $R=\eps^{-1}$, let $J_{3,R}$ be the exact flux for the chosen packet width $\delta$, and suppose the observed flux $J_{3,R}^{\rm obs}$ satisfies
		\begin{equation}\label{eq:raw-flux-error-assumption}
			\int_0^{T/R}\int_{\Gamma_+}|J_{3,R}^{\rm obs}(t,x,v)-J_{3,R}(t,x,v)|\,dS_x\,dv\,dt\le\nu_{\rm obs}.
		\end{equation}
		It suffices to impose this bound on the support of the detector. Define $\mathscr M_{\rm obs}^{\eps,\delta}$ by \eqref{eq:physical-scaled-measurement}, with $J_{3,R}^{\rm obs}$ in place of $J_{3,R}$. Then
		\begin{equation}\label{eq:raw-flux-line-error}
			\begin{split}
				|\mathscr M_{\rm obs}^{\eps,\delta}(e,g_0;\eta_\rho,\varphi_\tau)-X_xK(y_0,e;g_0,\omega_0)|
				\le C\big(\eps\tau^{-2}+\delta\tau^{-3}+\rho+\tau+\nu_{\rm obs}\eps^{-1}\rho^{-2}\tau^{-2}\big).
			\end{split}
		\end{equation}
		The constant is independent of the four resolution parameters and of $\nu_{\rm obs}$, for the fixed configuration and admissible kernel class. If observations with the bound \eqref{eq:raw-flux-error-assumption} are available at $\rho=\tau=\nu_{\rm obs}^{1/8}$, $\eps=\nu_{\rm obs}^{3/8}$, and $\delta=\nu_{\rm obs}^{1/2}$, then, for sufficiently small $\nu_{\rm obs}>0$,
		\begin{equation}\label{eq:raw-flux-holder-rate}
			|\mathscr M_{\rm obs}^{\eps,\delta}(e,g_0;\eta_\rho,\varphi_\tau)-X_xK(y_0,e;g_0,\omega_0)|\le C\nu_{\rm obs}^{1/8}.
		\end{equation}
	\end{corollary}
	
	\begin{proof}
		The detector in \eqref{eq:physical-restricted-detector} satisfies $\|\Psi_\eps^{e,g_0;\eta_\rho,\varphi_\tau}\|_{L^\infty}\le C\rho^{-2}\tau^{-2}$, because $\beta$ and the radial cutoff are fixed. The normalization of the physical measurement gives
		\[
		|\mathscr M_{\rm obs}^{\eps,\delta}-\mathscr M_K^{\eps,\delta}|
		\le R\|J_{3,R}^{\rm obs}-J_{3,R}\|_{L^1}\|\Psi_\eps^{e,g_0;\eta_\rho,\varphi_\tau}\|_{L^\infty}
		\le C\nu_{\rm obs}\eps^{-1}\rho^{-2}\tau^{-2}.
		\]
		Combining this with \eqref{eq:fully-resolved-scaled-error} proves \eqref{eq:raw-flux-line-error}. Each term has order $\nu_{\rm obs}^{1/8}$ for the stated choice of parameters.
	\end{proof}
	
	Under \eqref{eq:raw-flux-error-assumption}, the last estimate controls the error in the integral over a fixed chord that meets the boundary transversally. The exponent has not been shown to be optimal. An estimate uniform up to grazing chords, or an estimate for the full kernel in terms of errors in the boundary measurements, requires further analysis.
	
	\subsection{Physical interpretation and scope}\label{subsec:physical-scope}
	
	For the homogeneous differential cross section introduced in the Introduction, the density $b_{12\to34}$ satisfies
	\begin{equation}\label{eq:kernel-cross-section}
		K(g,\omega)=|g|b_{12\to34}(g,\omega).
	\end{equation}
	Let $R\to\infty$ and then $\delta\to0$. For the chord parametrized by $(y,e)$, the resulting product measure is $L(y,e)K(g_0,\omega)\,d\omega$. By \eqref{eq:reaction-sphere-change-of-measure}, its density on $\Sigma_3(g_0)$ is
	\begin{equation}\label{eq:surface-product-density}
		\frac{L(y,e)|g_0|}{\lambda_3(g_0)^2}b_{12\to34}\Big(g_0,\frac{z}{\lambda_3(g_0)}\Big),\quad z\in\Sigma_3(g_0).
	\end{equation}
	All factors outside $b_{12\to34}$ are known. Measurements of product velocities are used to determine differential cross sections in \cite{MichaelsenEtAl2017,VogelsEtAl2014}. Here we prescribe two synchronized incoming packets whose common velocity tends to infinity, together with tests in boundary position and time. For an inhomogeneous kernel, these measurements give the spatial X-ray transform.
	
	\para{A charge transfer channel} The reaction $\mathrm{Ar}^++\mathrm{H}_2\to\mathrm{Ar}+\mathrm{H}_2^+$ studied in \cite{MichaelsenEtAl2017} is an example in which the velocities of product ions are measured. To describe it by a model with four components, one must specify the incoming and product states and include their internal energies in $E_i$. If species $3$ is the measured ion, the products are ordered as $(\mathrm{H}_2^+,\mathrm{Ar})$. The sign of $\Delta E$ must be checked for the selected states. An exothermic forward channel does not satisfy $\Delta E\ge0$ with that ordering; interchanging the two channel pairs also changes which species must be injected. This experiment illustrates the detection geometry. Applying the theorem requires selecting a resolved channel and checking its hypotheses.
	
	\para{Homogeneous differential collision laws}
	\begin{corollary}[Homogeneous kernel]\label{cor:homogeneous}
		If $K\in\mathfrak K$ is independent of $x$, then any nondegenerate chord gives $K(g,\omega)=\frac{X_xK(y,e;g,\omega)}{L(y,e)}$. The measurements determine the differential collision law on $\mathscr G\times\mathcal O$ without a spatial inversion.
	\end{corollary}
	By \eqref{eq:kernel-cross-section}, division by $|g|>g_{\rm th}$ recovers $b_{12\to34}$ on $\mathscr G\times\mathcal O$.
	
	\para{Spatially modulated collision laws}
	\begin{corollary}[Recovery of a spatial parameter]\label{cor:field}
		Let $\mathcal R$ be a known response law and assume that $K(x,g,\omega)=\mathcal R(g,\omega;\Theta(x))$ belongs to $\mathfrak K$. If, for one accessible pair $(g_*,\omega_*)$, the map $\theta\mapsto\mathcal R(g_*,\omega_*;\theta)$ is injective on the admissible range of $\Theta$, then the measurements determine $\Theta$ in $\Omega$.
	\end{corollary}
	\begin{proof}
		Theorem~\ref{thm:main} determines $\mathcal R(g_*,\omega_*;\Theta(x))$ for every $x\in\Omega$. Injectivity determines $\Theta(x)$.
	\end{proof}
	Control of reactive or inelastic rates by external fields is studied in \cite{MatsudaEtAl2020,ParazzoliEtAl2011}. The corollary applies to an effective local response law satisfying the collision assumptions of this paper. Since the transport equation has no force term, it does not cover a field that changes the leading trajectories.
	
	\para{Size of the measurements} For fixed packet width, the recentered incoming profile is independent of $R$. Its instantaneous particle flux is of order $R$ and its duration is of order $R^{-1}$, so the incident fluence is of order one. The product yield is of order $R^{-1}$. This explains both the normalization in \eqref{eq:physical-detector-identity} and the amplification of absolute count errors in \eqref{eq:raw-flux-line-error}.
	
	The reconstruction uses a family of measurements over the prescribed propagation directions, relative velocities, spatial apertures, and angular filters, together with the stated resolution limits. This family is not a finite set of scalar measurements. The temporal test collects contributions from reactions along the same chord, while the boundary and velocity tests select the chord and collision direction. The theorem concerns one resolved reaction channel. If several channels create the measured species from the same injected pair and their product signals cannot be separated, the leading measurement contains their sum. Recovery of the individual kernels then requires additional state resolution or injectivity of the map from these kernels to the combined product velocity data.

	\section*{Statements and declarations}
	
	\para{Data availability statement} No datasets were generated or analyzed in this study.
	
	\para{Conflict of interest} The authors declare no conflict of interest.
	
	\para{Acknowledgments} The first author is partially supported by the National Science and Technology Council (NSTC), Taiwan, under the project 113-2115-M-A49-017-MY3. The first author also acknowledges financial support from the Alexander von Humboldt Foundation through the Henriette Herz Scouting Programme, hosted by Universit\"at Duisburg-Essen, Germany. The second author was supported by the Hong Kong RGC General Research Funds (projects 11311122, 11304224 and 11303125). The author acknowledges using AI tools. All mathematical arguments and proofs in the final manuscript were checked and written by the author.
	
	\bibliographystyle{plain}
	\bibliography{refs}
	
\end{document}